\documentclass[11pt]{amsart}

\usepackage{amsmath}
\usepackage{amsthm,amsfonts}

\usepackage{amssymb}
\usepackage{epsfig}
\usepackage{amssymb,latexsym}
\usepackage[usenames,dvipsnames]{color}

\begin{document}
\theoremstyle{plain}
\newtheorem{theorem}{Theorem}[section]
\newtheorem{lemma}[theorem]{Lemma}
\newtheorem{corollary}[theorem]{Corollary}
\newtheorem{proposition}[theorem]{Proposition}

\newtheorem{question}[theorem]{Question}
\theoremstyle{definition}
\newtheorem{notations}[theorem]{Notations}
\newtheorem{notation}[theorem]{Notation}
\newtheorem{remark}[theorem]{Remark}
\newtheorem{remarks}[theorem]{Remarks}
\newtheorem{definition}[theorem]{Definition}
\newtheorem{claim}[theorem]{Claim}
\newtheorem{assumption}[theorem]{Assumption}
\numberwithin{equation}{section}
\newtheorem{examplerm}[theorem]{Example}
\newtheorem{propositionrm}[theorem]{Proposition}

\newtheorem{example}[theorem]{Example}
\newtheorem{examples}[theorem]{Examples}

\newcommand{\binomial}[2]{\left(\begin{array}{c}#1\\#2\end{array}\right)}
\newcommand{\zar}{{\rm zar}}
\newcommand{\an}{{\rm an}}
\newcommand{\red}{{\rm red}}
\newcommand{\codim}{{\rm codim}}
\newcommand{\rank}{{\rm rank}}
\newcommand{\Pic}{{\rm Pic}}
\newcommand{\Div}{{\rm Div}}
\newcommand{\Hom}{{\rm Hom}}
\newcommand{\im}{{\rm im}}
\newcommand{\Spec}{{\rm Spec}}
\newcommand{\sing}{{\rm sing}}
\newcommand{\reg}{{\rm reg}}
\newcommand{\Char}{{\rm char}}
\newcommand{\Tr}{{\rm Tr}}
\newcommand{\tr}{{\rm tr}}
\newcommand{\supp}{{\rm supp}}
\newcommand{\Gal}{{\rm Gal}}
\newcommand{\Min}{{\rm Min \ }}
\newcommand{\Max}{{\rm Max \ }}
\newcommand{\Span}{{\rm Span  }}

\newcommand{\Frob}{{\rm Frob}}
\newcommand{\lcm}{{\rm lcm}}

\newcommand{\soplus}[1]{\stackrel{#1}{\oplus}}
\newcommand{\dlog}{{\rm dlog}\,}    
\newcommand{\limdir}[1]{{\displaystyle{\mathop{\rm
lim}_{\buildrel\longrightarrow\over{#1}}}}\,}
\newcommand{\liminv}[1]{{\displaystyle{\mathop{\rm
lim}_{\buildrel\longleftarrow\over{#1}}}}\,}
\newcommand{\boxtensor}{{\Box\kern-9.03pt\raise1.42pt\hbox{$\times$}}}
\newcommand{\sext}{\mbox{${\mathcal E}xt\,$}}
\newcommand{\shom}{\mbox{${\mathcal H}om\,$}}
\newcommand{\coker}{{\rm coker}\,}
\renewcommand{\iff}{\mbox{ $\Longleftrightarrow$ }}
\newcommand{\onto}{\mbox{$\,\>>>\hspace{-.5cm}\to\hspace{.15cm}$}}

\newenvironment{pf}{\noindent\textbf{Proof.}\quad}{\hfill{$\Box$}}

\newcommand{\sA}{{\mathcal A}}
\newcommand{\sB}{{\mathcal B}}
\newcommand{\sC}{{\mathcal C}}
\newcommand{\sD}{{\mathcal D}}
\newcommand{\sE}{{\mathcal E}}
\newcommand{\sF}{{\mathcal F}}
\newcommand{\sG}{{\mathcal G}}
\newcommand{\sH}{{\mathcal H}}
\newcommand{\sI}{{\mathcal I}}
\newcommand{\sJ}{{\mathcal J}}
\newcommand{\sK}{{\mathcal K}}
\newcommand{\sL}{{\mathcal L}}
\newcommand{\sM}{{\mathcal M}}
\newcommand{\sN}{{\mathcal N}}
\newcommand{\sO}{{\mathcal O}}
\newcommand{\sP}{{\mathcal P}}
\newcommand{\sQ}{{\mathcal Q}}
\newcommand{\sR}{{\mathcal R}}
\newcommand{\sS}{{\mathcal S}}
\newcommand{\sT}{{\mathcal T}}
\newcommand{\sU}{{\mathcal U}}
\newcommand{\sV}{{\mathcal V}}
\newcommand{\sW}{{\mathcal W}}
\newcommand{\sX}{{\mathcal X}}
\newcommand{\sY}{{\mathcal Y}}
\newcommand{\sZ}{{\mathcal Z}}

\newcommand{\A}{{\mathbb A}}
\newcommand{\B}{{\mathbb B}}
\newcommand{\C}{{\mathbb C}}
\newcommand{\D}{{\mathbb D}}
\newcommand{\E}{{\mathbb E}}
\newcommand{\F}{{\mathbb F}}
\newcommand{\G}{{\mathbb G}}
\newcommand{\HH}{{\mathbb H}}
\newcommand{\I}{{\mathbb I}}
\newcommand{\J}{{\mathbb J}}
\newcommand{\M}{{\mathbb M}}
\newcommand{\N}{{\mathbb N}}
\renewcommand{\P}{{\mathbb P}}
\newcommand{\Q}{{\mathbb Q}}
\newcommand{\T}{{\mathbb T}}
\newcommand{\U}{{\mathbb U}}
\newcommand{\V}{{\mathbb V}}
\newcommand{\W}{{\mathbb W}}
\newcommand{\X}{{\mathbb X}}
\newcommand{\Y}{{\mathbb Y}}
\newcommand{\Z}{{\mathbb Z}}

\newcommand{\be}{\begin{eqnarray}}
\newcommand{\ee}{\end{eqnarray}}
\newcommand{\nn}{{\nonumber}}
\newcommand{\dd}{\displaystyle}
\newcommand{\ra}{\rightarrow}
\newcommand{\bigmid}[1][12]{\mathrel{\left| \rule{0pt}{#1pt}\right.}}
\newcommand{\cl}{${\rm \ell}$}
\newcommand{\clp}{${\rm \ell^\prime}$}


\newcommand{\myConstant}{\gamma}
\newcommand{\id}[1]{\left\langle #1 \right\rangle}
\newcommand{\oldver}[1]{ {\color{magenta}{Old Version: #1}} }
\newcommand{\Hakan}[1]{ {\color{blue}{\underline{\textbf{Hakan's Comment:}} #1}} }
\newcommand{\Change}[1]{ {\color{blue}{#1}} }
\newcommand{\Details}[1]{ {\color{Violet}{#1}} }

\renewcommand{\Details}[1]{ }

\newcommand{\Real}{\mathbb{R}}
\newcommand{\textiti}[1]{\textit{#1}\index{#1}}
\newcommand{\cent}[2]{\bold{C}_#1(#2)}



\newcommand{\nck}[2]{{#1 \choose #2}}
\renewcommand{\T}{\frac{GR(2^a,m)[x]}{\id{x^{2^s}-1}}}
\newcommand{\Tm}{\mathcal{T}_m}

\newcommand{\GR}{GR(p^a,m)}
\newcommand{\fR}{\frac{GR(p^a,m)[x]}{\id{f(x)}}}
\newcommand{\Rf}{\mathcal{R}}
\newcommand{\Tor}{\textrm{Tor}}
\newcommand{\Rtwo}{GR(p^2,m)}
\newcommand{\fRtwo}{\frac{\Rtwo[x]}{\id{f(x)}}}
\newcommand{\cRtwo}{\frac{\Rtwo[x]}{\id{x^{p^s}-1}}}
\newcommand{\Rtwof}{\mathcal{R}_2}

\newcommand{\R}{GR(p^a,m)[x]}
\renewcommand{\T}{GR(p,m)[x]}

\newcommand{\Code}{ \ensuremath{C \vartriangleleft \sR}}

\newcommand{\newn}{ {\eta} }

\title[Polycyclic codes with applications to repeated-root codes]{Polycyclic codes over Galois rings with applications to repeated-root
constacyclic codes}

\author[L\'opez-Permouth, \"{O}zadam, \"{O}zbudak, Szabo]{Sergio R. L\'opez-Permouth$^1$, Hakan \"{O}zadam$^2$, Ferruh \"Ozbudak$^2$, Steve Szabo$^1$}

\maketitle

\begin{center}
$^1$
Department of Mathematics\\
Ohio University, Athens, Ohio, 45701, USA\\
\{lopez,szabo\}@math.ohiou.edu
\\ \ \\
$^2$
Department of Mathematics and Institute of Applied Mathematics \\
Middle East Technical University, \.{I}n\"on\"u Bulvar{\i}, 06531, Ankara, Turkey \\
\{ozhakan,ozbudak\}@metu.edu.tr
\end{center}

\vspace{0.5cm}

\abstract

Cyclic, negacyclic and constacyclic codes are part of a larger class of codes 
called polycyclic codes; namely, those codes  which can be viewed as ideals of 
a factor ring of a polynomial ring. The structure of the ambient ring of polycyclic codes over $GR(p^a,m)$
and generating sets for its ideals are considered. 
It is shown that these generating sets are strong Groebner bases. 
A method for finding such sets in the case that $a=2$ is also given. 
The Hamming distance of certain constacyclic codes of length $\newn p^s$\ and $2\newn p^s$ over $\F_{p^m}$\ is computed. 
A method, which determines the Hamming distance of the constacyclic codes of length $\newn p^s$\ 
and $2\newn p^s$ over $GR(p^a,m)$, where $(\newn,p)=1$, is described. 
In particular, the Hamming distance of all cyclic codes of length $p^s$\ over $GR(p^2,m)$\ and 
all negacyclic codes of length $2p^s$\ over $\F_{p^m}$ is determined explicitly.\ \\

    
\endabstract




\vspace{0.5cm}
\section{Introduction}
\label{Section.Introduction}

Important applications of modules over finite rings to error-correcting codes 
and sequences were introduced in \cite{Hammons_et_al_IEEE} and \cite{Kummar_et_al_IEEE_1996}.
In particular, \cite{Hammons_et_al_IEEE} motivated
the study of cyclic and negacyclic codes over Galois rings 
(see, for example, \cite{Abulrub_Oehmke_IEEE_2003,Blackford_Chaudri_IEEE_2000,Blackford_IEEE_2003,Dinh-Lopez_2004,Kanwar_Lopez_FFA_1997,Vega_Wolfmann_FFA_2004,Pless_IEEE_1996,Wolfmann_IEEE_1999,Wolfmann_IEEE_2001}). 
For a recent survey on this topic, we refer the reader to \cite{Dinh_Lopez_Szabo_Cimpa}.
Cyclic codes can be grouped into two classes: simple-root cyclic codes, where
the codeword length and the characteristic of the alphabet are
coprime, and repeated-root cyclic codes, where the codeword
length and the characteristic of the alphabet are not coprime.
The structure of simple-root cyclic codes over rings was studied throughly in
\cite{Pless_IEEE_1996,Kanwar_Lopez_FFA_1997,Calderbank_Sloane_Modular_Cyclic_Codes,Dinh-Lopez_2004}
and certain special generating sets for these codes were determined therein. 
On the other hand, repeated-root cyclic codes are also interesting
as they allow very simple syndrome-forming and decoding circuitry 
and because in some cases (see \cite{CMSS2,Roth_Seroussi_IEEE_1986}) they are maximum distance separable. 
A partial list of references for the theory of repeated root cyclic codes 
includes
\cite{CMSS1,D2005,D2007,D2008,Dinh_IEEE_2009,Dinh_Algebra_2010,Dougherty-Park_2007,Szabo-Lopez_2009_HW,CMSS2,Roth_Seroussi_IEEE_1986,SanLing-etal_2008,Salagean_Disc_2006,Tang_et_al_IEEE,Van_Lint_IEEE_Rep_Root,Zimmermann_IEEE}. 
Amongst these studies, generating sets that are similar to those in 
\cite{Pless_IEEE_1996,Kanwar_Lopez_FFA_1997,Calderbank_Sloane_Modular_Cyclic_Codes,Dinh-Lopez_2004} are studied in
\cite{Szabo-Lopez_2009_HW,Dougherty-Park_2007,SanLing-etal_2008} 
for cyclic codes of length $p^s$\ over an alphabet whose characteristic is a power $p$.
In \cite{Dougherty-Park_2007,SanLing-etal_2008}, the notion of torsional codes is used to study
generators of these codes.
The structural properties of cyclic codes are studied in a more general setting in
\cite{Salagean_AAECC,Salagean_IEEE,Salagean_Bulletin,Salagean_FFA,Salagean_Disc_2006} from a Groebner basis perspective.
Our study unifies
the two approaches above and generalizes them in the following
sense: we show that codes in a wider class of linear codes
called polycyclic codes have generating sets sharing the same
properties as those described in
\cite{Salagean_AAECC,Salagean_IEEE,Salagean_Bulletin,Salagean_FFA,Salagean_Disc_2006,Dougherty-Park_2007,SanLing-etal_2008}.
This allows us to study the ideal structure of cyclic codes
without the restriction that the codes must be simple-root.
In particular, we compute the Hamming distance of certain
constacyclic codes of length  $\newn p^s$\ and $2\newn p^s$, where $(\newn, p) =1$,
over a finite field of characteristic $p$. 
Then using this result together with
the above generating sets, we give a method to determine the
Hamming distance of certain constacyclic codes of length  $p^s$
and $2p^s$ over a Galois ring of characteristic a power of $p$.
As another particular case, we explicitly determine the Hamming
distance of all cyclic codes of length $p^s$ over $GR(p^2,m)$ which
generalizes the results of a recent study \cite{zhu_kai_2010}.

We study linear codes over Galois rings
that have the additional structure that they can be described as an
ideal of a quotient ring, specifically a quotient ring of a
polynomial ring over a Galois ring where the ideal being factored out is generated by a regular polynomial.
We begin with studying the structure of the ring $\frac{GR(p^a,m)[x]}{\id{g(x)}}$\ where $g(x)$ is a regular primary polynomial.
We show that $\frac{GR(p^a,m)[x]}{\id{g(x)}}$\ is a local ring  with a simple socle and we determine its maximal ideal and socle.
We give necessary and sufficient conditions for $\frac{GR(p^a,m)[x]}{\id{g(x)}}$\ to be a chain ring.
Next, we use the results on these rings to study the structure of $\fR$\ where $f(x)$ is a regular polynomial.
This work uses a factorization given by \cite{McDonald_Book} 
of regular polynomials into regular primary polynomials and also the Chinese Remainder Theorem.
Via this ring decomposition, we give details on the structure of $\fR$.
This provides information on the structure of the polycyclic codes, and in particular cyclic and constacyclic codes,
as their ambient spaces are of the form of $\fR$.
as their ambient spaces are of the form of $\fR$.

Some special generating sets, for cyclic codes of length $p^s$ over $GR(p^a,m)$, were studied in \cite{Dougherty-Park_2007}
by employing torsional degrees and torsional codes.
Later, in \cite{SanLing-etal_2008}, Kiah et. al. came up with a unique set of generators for such codes.
We generalize their results to polycyclic codes. 
More explicitly, we extend the notion of torsional degree and torsional code to polycyclic codes and
we show that polycyclic codes have generating sets with the same properties as in 
\cite{Dougherty-Park_2007} and \cite{SanLing-etal_2008}.
Furthermore, we observe that the unique generating set studied in \cite{SanLing-etal_2008} is actually
a strong Groebner basis which is studied in a series of papers 
\cite{Salagean_FFA,Salagean_IEEE,Salagean_AAECC,Salagean_Bulletin,Salagean_Disc_2006}  
by S\u{a}l\u{a}gean and Norton.
We show that a minimal strong Groebner basis actually gives us all the torsional degrees of a polycyclic code.
This allows us to describe how to obtain a generating set in standard form, which is a minimal strong Groebner basis,
from the unique generating set introduced in \cite{SanLing-etal_2008} and vice versa.
Also the torsional degrees, equivalently a minimal strong Groebner basis, can be used to
determine the Hamming distance of a polycyclic code when the Hamming distance of the residue code is known.

We use the above results to study
some constacyclic codes of length $\newn p^s$\ and $2\newn p^s$ over $GR(p^a,m)$.
First we compute the Hamming distance of these codes over the residue field.
Then, we give the ideal structure and the Hamming distance of these codes by using
a generating set in standard form. 
In some cases, our results give the Hamming distance of all such constacyclic codes.

As another application of our results, we generalize a recent result of \cite{zhu_kai_2010}
on the Hamming distance of cyclic codes of length $2^s$\ over $\Z_4$.
We classify all polycyclic codes over $GR(p^2,m)$ 
which gives us a classification of all cyclic codes of length $p^s$.
Then we determine the torsional degrees of these codes in each case
yielding the Hamming distance of all cyclic codes of length $p^s$\ over $GR(p^a,m)$.

This paper is organized as follows. 
In Section \ref{Section.Algebraic.Background},
we give some preliminaries and fix our notation.
In Section \ref{Local.Subambients},
we study the subambient rings of polycyclic codes along with their torsional degrees and strong Groebner bases.
We further study these subambients in characteristic $p^2$\ and determine their torsional degrees and Hamming distance in 
Section \ref{Subambiets.in.char.p2}.
We study the structure of the ambient ring of polycyclic codes in Section \ref{sect.main}.
We give some preliminaries for the computation of the Hamming distance of some constacyclic codes 
over a finite field in Section \ref{Preliminaries.For.Computations}.
Then we compute the Hamming distance of certain constacyclic codes of length $\newn p^s$\ over $\F_{p^m}$ 
and we describe how to determine the Hamming distance of these codes over $GR(p^a,m)$ 
in Section \ref{Section.Irreducible.Case.Finite.Fields}.
Finally, in Section \ref{Section.Reducible.Case.Finite.Fields}, we carry out similar
computations for certain constacyclic codes of length $2\newn p^s$.

\section{Algebraic Background}
\label{Section.Algebraic.Background}
In this section we state some basic facts about finite chain rings, polynomials over
Galois rings and we fix our notation on cyclic and polycyclic codes.
For a detailed treatment of the theory of Galois rings, we refer the reader to
\cite{Bini_Flamini_Book} or \cite{McDonald_Book}.

Let $p$\ be a prime number and $a,m \ge 1$ be integers. Then $\F_{p^m}$\ denotes the finite field with $p^m$\ elements and
$GR(p^a,m)$\ denotes the Galois ring of characteristic $p^a$\ with $p^{am}$\ elements.

Let $R$\ be a commutative ring with a unit.
$R$ is called a \textit{local ring} if it has a unique maximal ideal.
An element $r \in R$\ is said to be \textit{nilpotent} with \textit{nilpotency index} $t$\ if
$r^t = 0$\ and $t$\ is the least nonnegative integer with respect to this property.
The intersection of all maximal ideals of $R$\ is called the \textit{Jacobson} of $R$\ and is denoted by
$J(R)$. The \textit{socle} of $R$, denoted by $soc(R)$, is the sum of all ideals of $R$ containing only themselves and the
zero ideal.
$R$ is called a \textit{chain ring} if its ideals are linearly ordered under set inclusion.
In \cite{Dinh-Lopez_2004}, a useful characterization of finite chain rings is given.

\begin{lemma}[{\cite[Proposition 2.1]{Dinh-Lopez_2004}}]
 \label{Algebraic.Beackground.Finite.Chain.Ring.Criterion}
 Let $R$\ be a finite commutative ring.
 The following are equivalent.
 \begin{enumerate}
  \item $R$\ is a chain  ring.
  \item $R$\ is a local principal ideal ring.
  \item $R$\ is a local ring and the maximal ideal of $R$ is principal.
 \end{enumerate}
 Furthermore, if $R$\ is a finite commutative chain ring with the maximal ideal $\id{\nu}$,
  then the ideals of $R$\ are exactly $\id{\nu^i}$\ where $i \in \{ 0,1,\dots, t \}$\
  and $t$\ is the nilpotency index of $\nu$.
\end{lemma}

It is well-known that the Galois ring $GR(p^a,m)$\ is a local ring with the maximal ideal
$\id{p}$.
Moreover $GR(p^a,m)$\ is a finite chain ring and its ideals are
$\id{p^i}$\ where $i \in \{ 0,1 , \dots, a \}$.
Let $\zeta$\ be a generator of the multiplicative group $\F_{p^m}\setminus \{0\}$.
The fact that $\Z_{p^a}[\zeta] \cong GR(p^a,m)$\ is a classical result of finite ring theory.
We can express an element $z \in GR(p^a,m)$\ as
$z = \sum_{j=0}^{p^m-2}v_j\zeta^{j}$
where $v_j \in \Z_{p^a}$.
Let $\sT_m = \{ 0,1,\zeta,\dots, \zeta^{p^m-2} \}$.
The set $\sT_m$\ is called the Teichm\"uller set.
Alternatively, we can uniquely express $z \in GR(p^a,m)$\ as
\be
    z = z_0 + pz_1 + \cdots + p^{a-1}z_{a-1}, \quad z_i \in \sT_m,  \nn
\ee
which is called the $p$\textit{-adic expansion} of $z$.
The map $\mu:\quad GR(p^a,m) \rightarrow \F_{p^m}$\ defined by
$\mu (z) = z_0$\ is a ring epimorphism with the kernel $\id{p}$.
Hence
$\frac{GR(p^a,m)}{\id{p}} \cong \F_{p^m}$.
The finite field $\F_{p^m}$\ is called the \textit{residue field} of $GR(p^a,m)$.
The map $\mu$\ is called the \textit{canonical projection}
and extends to a homomorphism between the polynomial rings
$GR(p^a,m)[x]$\ and $\F_{p^m}[x]$
 in a natural way as
$\mu(a_0 + a_1x + \cdots + a_nx^n) = \mu(a_0) + \mu(a_1)x + \cdots + \mu(a_n)x^n.$
We denote  $\mu(f(x))$\ by $\bar{f}(x)$.
Note also that $\mu$\ maps the ideals of $GR(p^a,m)[x]$\ to the ideals of $\F_{p^m}[x]$\ and
we denote the canonical projection of the ideal $I$\ by $\bar{I}$.

A polynomial $f(x) \in GR(p^a,m)[x]$\ is called \textit{regular} if 
$f(x)$\ is not a zero divisor.
Moreover, by the characterization given in \cite[Theorem XIII.2]{McDonald_Book},
$f(x)$\ is regular if and only if one of its coefficients is a unit in $GR(p^a,m)$.
If $f(x)$\ can  not be expressed as a product of two nonconstant polynomials,
then $f(x)$\ is called \textit{irreducible} and if in addition $\bar{f}(x)$\
is irreducible then $f(x)$\ is called \textit{basic irreducible}.

An ideal $I \vartriangleleft R$\ is called a \textit{primary ideal} if
for all $uv \in I$, we have $u^n\in I$\ or $v \in I$\ for some positive  integer $n$.
A polynomial $f(x)$\ is called \textit{primary}\ if $\id{f(x)}$\ is a primary ideal.
Besides, $I \vartriangleleft R$\ is called a \textit{prime ideal} if
for all $uv \in I$, we have $u\in I$\ or $v \in I$.

\begin{theorem}[{\cite[Theorem XIII.11]{McDonald_Book}}]
  \label{Theorem.McDonald.Factorization.and.Uniqueness}
 Let $f(x) \in GR(p^a,m)[x]$\ be a regular polynomial. Then
 $f(x) = \delta g_1(x) \cdots g_r(x)$ where $\delta$\ is a unit and
 $g_1(x),\dots , g_r(x)$\ are regular primary coprime polynomials.
 Moreover, this factorization is unique up to reordering terms and multiplication by units.
\end{theorem}

Now we recall the division algorithm in $\F_{p^m}[x]$\ and $GR(p^a,m)[x]$.
Since $\F_{p^m} [x]$\ is a Euclidean domain, for any $v(x)$\ and $ 0\neq g(x) \in \F_{p^m}[x]$,
there exist unique polynomials $y(x), r(x) \in \F_{p^m}[x] $\ such that
\be\nn
    v(x) = g(x)y(x) + r(x)
\ee
where either $0 \le \deg (r(x)) < \deg (g(x))$\ or $r(x) = 0$.
We define
$
    v(x) \mod g(x) = r(x),
$
and we use the notation $v(x) \equiv r(x) \mod g(x)$\ in the usual sense.

There is also a division algorithm for polynomials in $GR(p^a,m)[x]$
(see, for example, \cite[Exercise XIII.6]{McDonald_Book} or
\cite[Proposition 3.4.4]{Bini_Flamini_Book}).
Let $f(x) \in GR(p^a,m)[x]$\ and let $h(x) \in GR(p^a,m)[x]$ be a regular polynomial.
Then there exist polynomials $z(x), b(x) \in GR(p^a,m)[x]$\ such that
\be
  f(x) = z(x)h(x) + b(x)\nn
\ee
and $\deg(b(x)) < \deg(h(x))$\ or $b(x) = 0$.

Throughout this paper, $C$\ stands for a linear code over $GR(p^a,m)$ and we identify
a codeword $c = (c_0,c_1,\dots , c_{N-1})\in C$\ with the polynomial
$c(x) = c_0 + c_1x + \cdots + c_{N-1}x^{N-1} \in GR(p^a,m)[x]$.
Let $\lambda \in GR(p^a,m) \setminus \{ 0 \}$\ and $I = \langle x^N - \lambda \rangle$.
The $\lambda$-\textit{shift} of a codeword $c$\ is defined
to be $(\lambda c_{N-1}, c_0, c_1,\cdots,c_{N-2})$. If a linear code $C$\ is closed under
$\lambda$-shifts, then $C$\ is called a $\lambda$-cyclic code and in general,
such codes are called \textit{constacyclic} codes (c.f. \cite[Section 13.2]{BRLKMP}).
It is well-known that
$\lambda$-cyclic codes, of length $N$, over $GR(p^a,m)$\ correspond
to the ideals of the finite ring
\be\nn
    \sR_c = \frac{GR(p^a,m)[x]}{I}.
\ee
In particular, cyclic (respectively negacyclic) codes, of length $N$,
over $GR(p^a,m)$ correspond to the ideals of the ring
$\sR_{\mathfrak{a}} = GR(p^a,m)[x] / \mathfrak{a}$
(respectively $\sR_{\mathfrak{b}} = GR(p^a,m)[x] / \mathfrak{b}$), where
$\mathfrak{a} = \langle x^N -1 \rangle$ (respectively $\mathfrak{b} = \langle x^N + 1 \rangle $).
Additionally
if $N$\ is not divisible by $p$, then $C$\ is called a \textit{simple-root} constacyclic code and
if $N$\ is divisible by $p$, then $C$\ is said to be a \textit{repeated-root} constacyclic code.

Now we define a family of linear codes which is a generalization of constacyclic codes.
Let $f(x) \in GR(p^a,m)[x]$\ be an arbitrary regular polynomial, $J = \id{f(x)}$ and let
\be
  \sR = \frac{GR(p^a,m)[x]}{J}.\nn
\ee
As done above, identifying the codewords with polynomials, we see that the ideals of
$\sR$\ are linear codes and they are called \textit{polycyclic} codes.
Obviously, although the elements of $\sR$\ are equivalence classes (cosets), 
they can be uniquely identified with polynomials
with degree strictly less than $\deg{f(x)}$.
Consequently, for the rest of this paper, unless otherwise stated, 
we focus on the ideals of $\sR$\ containing $J$ and identify  $I/J$ with
$\{g(x): \quad g(x)\in I\quad \mbox{and} \quad \deg(g(x) < \deg(f(x)))\}$\
and, for all $g(x)$\ such that  $\deg(g(x)) < \deg(f(x))$, we identify
the equivalence class $g(x) + J$ with $g(x)$.

Let $\bar{\sR} = \frac{\F_{p^m}[x]}{ \bar{J} }$. The map $\mu$, defined above,
extends to an onto ring homomorphism as $\mu: \sR \rightarrow \bar{\sR}$\ where
$\mu( g(x) + J) = \bar{g}(x) + \bar{J}$.
For $r \in \sR$\ and $w \in \bar{\sR} $,
we define the scalar multiplication by $rw\ (mod\ p)$\ where we consider the multiplication in
$\sR$. This makes $\bar{\sR}$\
an $\sR$-module.

The \textit{Hamming weight} of a word is defined to be the number of nonzero 
entries of the word
and the \textit{Hamming weight} of a polynomial is defined to be the number of nonzero coefficients
of the polynomial. Let $c$\ and $c(x)$\ be as above. We denote the Hamming weight of
$c$\ and $c(x)$\ by $w_H(c)$\ and $w_H(c(x))$, respectively.
Obviously, the Hamming weight of a codeword and the Hamming weight of the corresponding
polynomial are equal, i.e., $w_H(c) = w_H(c(x))$.

The \textit{Hamming distance} of a linear code $C$\ is defined as
\be\nn
    d_H(C) = \min \{w_H(v):\quad 0 \neq v \in C \}.
\ee
The following lemma gives us some useful information on $d_H(C)$.

\begin{lemma}
 \label{Preliminaries.Lemma.Hamming.Weight.Projection}
 Let $\{ 0 \} \neq \Code$ be a constacyclic code of length greater than 1 over $GR(p^a,m)$
 with $C \neq \{0\}$\ and $C \neq \id{1}$, and let $\bar{C}  \vartriangleleft \bar{\sR}$\ be its canonical projection.
 Then $d_H(C) = d_H(\bar{C})$ as
 the $\sR$-modules
 $ p^{a-1}\sR$\ and $\bar{\sR}$\
 are isomorphic. Moreover $d_H(\bar{C}),d_H(C) \ge 2$.
\end{lemma}
\begin{proof}
 The isomorphism is established by sending
 $f(x) \in \bar{\sR}$\ to
 $p^{a-1}f(x) \in  p^{a-1}\sR$.
 The bound $d_H(\bar{C}), d_H(C) \ge 2$\ follows from the facts that $d_H(C) = d_H(\bar{C})$\ and a proper ideal can not contain a unit.
\end{proof}

\section{Local subambients of polycyclic codes}
\label{Local.Subambients}

In this section, the ring
\be\nn
  \sR = \frac{GR(p^a,m)[x]}{\id{f(x)}},
\ee
where $f(x)\in\R$ is a regular primary polynomial
which is not a unit, is studied.
The results of this section will be used to study the more general case,
where $f(x)$\ is not  necessarily primary in Section \ref{sect.main}.

First we show that $\sR$\ is a local ring and determine its maximal ideal,
we determine the socle of $\sR$,
for $a \ge 1$, we give necessary and sufficient conditions for $\sR$\ to be
a chain ring in Lemma \ref{Local.Subambients.Theorem.Chain.Condition}.
Then, using the notion of torsional code and torsional degree,
we determine a unique generating set for any ideal of $\sR$\
in Theorem \ref{Local.Subambients.Theorem.Unique.Generating.Set}.
Next we observe, in Corollary \ref{prop.ideal}, that
such a generating set is a strong Groebner basis and
if we remove the redundant generators,
we obtain a generating set in standard form which is a minimal strong Groebner basis.
Finally, we show that the torsional degrees of a polycyclic code
can immediately be obtained from a generating set in standard form.

In this section we assume $f(x)$ is a regular primary
polynomial that is not a unit.
By
\cite[Theorem XIII.6]{McDonald_Book}, $f(x)=vf^*(x)$ where $v$ is a unit and $f^*(x)$ is
monic and regular. Since $\id{f(x)}=\id{vf^*(x)}$ and because of our
interest in $\sR$, assume $f(x)$ is monic. By Proposition
\cite[XIII.12]{McDonald_Book}, $f(x)=\delta(x)h(x)^t+p\beta(x)$ for some
$\delta(x),h(x),\beta(x)\in\R$ where $\delta(x)$ is a unit and
$h(x)$ is a basic irreducible polynomial. Since $\delta(x)$ is a
unit, by \cite[Theorem XIII.2]{McDonald_Book},
$\delta(x)=\delta_0+p\delta'(x)$ for some $\delta_0\in GR(p^a,m)$
that is a unit and some $\delta'(x)\in\R$. Also, since $h(x)$ is
basic, $h(x)=\overline{h}(x)+p\alpha(x)$ for some $\alpha(x)\in \R$.
So, $\overline{f}(x)=\delta_0\overline{h}(x)^t$ and
$f(x)=\delta_0\overline{h}(x)^t+p\beta'(x)$ for some
$\beta'(x)\in\R$.

Assume $f(x)=\delta
h(x)^t+p\beta(x)$ where $\delta\in GR(p^a,m)$ is a unit and $h(x)$
is a basic irreducible such that $\overline{h}(x)=h(x)$. By the fact
that $f(x)$ is monic, we know that $t\deg{h(x)}>\deg{\beta(x)}$.
Furthermore, without loss of generality, we may assume $h(x)$ is
monic. By this assumption, $\delta=1$ since $f(x)$ is monic. Hence,
$f(x)$ is a monic regular primary polynomial such that
$f(x)=h(x)^t+p\beta(x)$ where $h(x)$ is a monic basic irreducible
polynomial such that $\bar{h}(x)=h(x)$.

We show that $\id{p,h(x)}$\ is the unique maximal ideal of $\sR$.

\begin{lemma}
  \label{Local.Subambients.Lemma.R.Local} The ring $\sR$\ is local
  with maximal ideal $J(\Rf)=\id{p+\id{f},h(x)+\id{f}}$.
\end{lemma}
\begin{proof}
  As discussed in page 262 of \cite{McDonald_Book}, any maximal ideal
  in $GR(p^a,m)[x]$\ is of the form $\id{p, g(x)}$\ where $g(x)$\ is a
  basic irreducible polynomial. Assume $f(x)\in\id{p,g(x)}$ where
  $g(x)\in\R$ is a basic irreducible polynomial. Then for some
  $a(x),b(x)\in\R$
  \begin{eqnarray*}
    f(x)&=&a(x)p+b(x)g(x),\\
     \bar{f}(x)&=&\bar{b}(x)\bar{g}(x),\\
    \bar{h}(x)^t&=&\bar{b}(x)\bar{g}(x).
  \end{eqnarray*}
  This shows that $\bar{h}(x)|\bar{g}(x)$ which implies
  $\bar{g}(x)|\bar{h}(x)$ and $g(x)=h(x)+pc(x)$ for some $c(x)\in\R$.
  So, $\id{p,g(x)}=\id{p,h(x)}$ meaning $\id{p,h(x)}$ is the only
  maximal ideal containing $f(x)$. Hence, $\id{p+\id{f},h(x)+\id{f}}$
  is the unique maximal ideal of $\Rf$.
\end{proof}

In the case of finite fields, $\sR$\ is a chain ring.

\begin{lemma}
  \label{Local.Subambients.Lemma.Finite.Field.Chain} The quotient ring
  $\frac{GR(p,m)[x]}{\id{f(x)}}$\ is a chain ring with exactly the
    following ideals
  \[
  \frac{GR(p,m)[x]}{\id{f(x)}}=\id{h(x)^0+\id{f}}\supsetneq\id{h(x)^1+\id{f}}\supsetneq\cdots
  \supsetneq\id{h(x)^t+\id{f}}=0.
  \]
\end{lemma}
\begin{proof}
  By Lemma \ref{Local.Subambients.Lemma.R.Local},
  $\frac{GR(p,m)[x]}{\id{f(x)}}$ is local with
  $J\left(\frac{GR(p,m)[x]}{\id{f(x)}}\right)=\id{h(x)+\id{f}}$. By
  Lemma \ref{Algebraic.Beackground.Finite.Chain.Ring.Criterion}, the
  result follows.
\end{proof}

Now we determine the socle of $\sR$\ and show that it is simple.

\begin{lemma}
  \label{Local.Subambients.Lemma.R.Simple.Socle} The ring $\sR$\ has
  simple socle with
  $soc\left(\Rf\right)=\id{p^{a-1}h(x)^{t-1}+\id{f}}$.
\end{lemma}
\begin{proof}
  Let $g(x)+\id{f}\in\Rf$. Let $\ell$\ be the largest integer such
  that $p^{\ell}(g(x)+\id{f})\neq 0$. By Lemma
  \ref{Local.Subambients.Lemma.R.Local},
  $J(\bar{\Rf})=\id{h(x)+\id{f}}$. By Lemma
  \ref{Preliminaries.Lemma.Hamming.Weight.Projection}
  and Lemma \ref{Local.Subambients.Lemma.Finite.Field.Chain} and the fact
  that $p^{\ell}(g(x)+\id{f})\in\id{p^{a-1}+\id{f}}$, it can be shown
  that $\id{p^{a-1}h(x)^{t-1}+\id{f}}\subset\id{g(x)+\id{f}}$. So
  $\id{p^{a-1}h(x)^{t-1}+\id{f}}$ is contained in any principal ideal.
  Since $J(\sR)$ annihilates $\id{p^{a-1}h(x)^{t-1}+\id{f}}$,
  $soc(\sR)=\id{p^{a-1}h(x)^{t-1}+\id{f}}$. It is clearly simple.
\end{proof}

Lemma
\ref{Local.Subambients.Lemma.Finite.Field.Chain} tells us
when the alphabet is a finite field, then
$\sR$\ is a chain ring. However, $\sR$\ is not a chain
ring in general. As a counter example, consider
$\frac{\Z_4[x]}{\id{x^2-1}}$. We have $x^2 -1 = (x+1)^2 - 2(x+1)$.
Clearly, $(x+1) \notin \id{2}$\ in $\frac{\Z_4[x]}{\id{x^2-1}}$.
Assume $2 \in \id{x+1}$. Then $2 = g_1(x)(x+1) + g_2(x)(x^2-1) \in
\Z_4[x]$. Evaluating at $x=-1$, we get $2 = 0$\ in $\Z_4$. This is a
contradiction. Thus we have shown $\id{2} \not \subset \id{x+1}$\
and $\id{x+1} \not \subset \id{2} $. By Lemma
\ref{Local.Subambients.Lemma.R.Local},
$J\left(\frac{\Z_4[x]}{\id{x^2-1}}\right)=\id{2,x+1}$. Since
$J\left(\frac{\Z_4[x]}{\id{x^2-1}}\right)$ is 2-generated, by Lemma
\ref{Algebraic.Beackground.Finite.Chain.Ring.Criterion}
$\frac{\Z_4[x]}{\id{x^2-1}}$ is not a chain ring.

The next theorem shows exactly when $\sR$ is a chain ring based on the
parameters $a,t,h(x)$ and $\beta(x)$ of $f(x)$.

\begin{theorem}
 \label{Local.Subambients.Theorem.Chain.Condition}
 The ring $\sR$\ is a chain ring if and only if any one of the conditions is met
 \begin{enumerate}
  \item
    $a = 1$
  \item
    $t = 1$
  \item
    $\beta(x) \not \in \id{p, h(x)}$.
 \end{enumerate}
\end{theorem}
\begin{proof}

Assume $a=1$. By Lemma \ref{Local.Subambients.Lemma.Finite.Field.Chain},
$\Rf$ is a chain ring.

Assume $t=1$ then $h(x)=f(x)-p\beta(x)\in \id{p,f(x)}$. So,
$h(x)+\id{f}\in \id{p+\id{f}}$. By Lemma
\ref{Local.Subambients.Lemma.R.Local},
$J\left(\Rf\right)=\id{p+\id{f}}$. Hence, by Lemma
\ref{Algebraic.Beackground.Finite.Chain.Ring.Criterion}, $\sR$ is a
chain ring.

Assume $\beta(x)\notin \id{p,h(x)}$. Then $\beta(x)+\id{f}\notin
J\left(\sR\right)$ which implies $\beta(x)+\id{f}$ is a unit in
$\sR$. So, $\id{p+\id{f}}=\id{h(x)^t+\id{f}}$ which implies
$p+\id{f}\in\id{h(x)+\id{f}}$. By Lemma
\ref{Local.Subambients.Lemma.R.Local},
$J\left(\sR\right)=\id{h(x)+\id{f}}$. Hence, by Lemma
\ref{Algebraic.Beackground.Finite.Chain.Ring.Criterion}, $\sR$ is a
chain ring.

Now assume $a>1$, $t>1$ and $\beta(x)\in \id{p,h(x)}$. We want to
show that $\sR$ is not a chain ring so assume the contrary. This
implies $\id{p+\id{f}}\subset\id{h(x)+\id{f}}$ or
$\id{h(x)+\id{f}}\subset\id{p+\id{f}}$. So, $p\in\id{h(x),f(x)}$ or
$h(x)\in\id{p,f(x)}$. First, assume $p\in\id{h(x),f(x)}$ which
implies $\beta(x)\in\id{p,h(x)}=\id{p,h(x),f(x)}=\id{h(x),f(x)}$.
So,
\[
f(x)=h(x)^t+p\beta(x)=h(x)^t+p(\gamma(x)h(x)+\alpha(x)f(x))
\]
for some $\gamma(x),\alpha(x)\in\R$ and
\[
f(x)(1-p\alpha(x))=h(x)\left(h(x)^{t-1}+p\gamma(x)\right).
\]
Since $(1-p\alpha(x))$ is invertible in $\R$, $f(x)\in\id{h(x)}$.
So, $p\in\id{h(x),f(x)}=\id{h(x)}$. Since $a>1$, $p\neq0$. This is a
contradiction since $p$ cannot be a nonzero multiple of $h(x)$.

Next, assume $h(x)\in\id{p,f(x)}$. Then,
\[
h(x)^t=[\gamma(x)p+\alpha(x)f(x)]^t=f(x)-p\beta(x)
\]
for some $\gamma(x),\alpha(x)\in\R$. This implies,
\[
\overline{\left[\alpha(x)f(x)\right]^t}=\overline{f(x)}.
\]
Since $t>1$, by comparing degrees we see this is a contradiction.
Hence, $\sR$ is not a chain.

 \end{proof}

Below are two examples that show the distinctions between the particular cases in
Theorem \ref{Local.Subambients.Theorem.Chain.Condition}.

\begin{example}
  Let $a > 1, p = 2, s > 0$\ and $f(x) = x^{2^s} + 1$. Then
  \be
    x^{2^s} + 1& = & (x+1-1)^{2^s} + 1\nn\\
    & = & (x+1)^{2^s} - {2^s \choose 2^s -1}(x+1)^{2^s -1} + \cdots - {2^s \choose 1}(x+1) +  1 + 1  \nn\\
    & = & (x+1)^{2^s} + 2\beta(x)  \nn
  \ee
  where $\beta(x) = (x+1)q(x) + 1$\ for some $q(x) \in \sR$.
  In \cite{D2005} it was shown that $\frac{GR(p^a,m)[x]}{\id{f(x)}}$\ is a chain ring with the maximal ideal $\id{x+1}$.
\end{example}

\begin{example}
 Let $a > 1, p=2, s>0$\ and $f(x) = x^{2^s} -1$. Then
 \be
  x^{2^s} - 1 & = & (x + 1 - 1)^{2^s} - 1 \nn\\
  & = & (x+1)^{2^s} - {2^s \choose 2^s -1}(x + 1)^{2^s -1} + \cdots - {2^s \choose 1}(x + 1) + 1 -1 \nn\\
  & = & (x+1)^{2^s} + 2 \beta(x)\nn
 \ee
 where $(x + 1) | \beta(x)$. In \cite{Szabo-Lopez_2009_HW} it was shown that $\frac{GR(p^a,m)[x]}{\id{f(x)}}$\ is local with the maximal ideal $\id{2,
 (x+1)}$ and is not a chain ring.
\end{example}

Theorem \ref{Local.Subambients.Theorem.Chain.Condition} shows that
$\sR$\ is not a principal ideal ring in general. Through the next
series of results we will show the existence of a particular
generating set which turns out to be a strong Groebner basis.

Let $g(x) \in GR(p^a,m)[x]$\ and $n$\ be the largest integer such that $\deg(g(x)) \ge n\deg(h(x))$.
By the division algorithm,
we can find $q_n(x), r_1(x) \in GR(p^a,m)[x]$
such that
\be
  g(x) = q_n(x)h(x)^n + r_1(x),\nn
\ee
where $r_1(x) = 0$\ or $\deg(r_1(x)) < n\deg(h(x))$.
Note that $\deg(q_n(x)) < \deg(h(x))$. Next we can find $q_{n-1}(x),r_2(x) \in GR(p^a,m)[x]$\ such that
\be
  r_1(x) = q_{n-1}(x)h(x)^{n-1} + r_2(x)\nn
\ee
where $r_2(x) = 0$\ or $\deg(r_2(x)) < (n-1)\deg(h(x))$.
Note that $\deg(q_{n-1}(x)) < \deg(h(x))$. We can continue this  process until
we have $q_n(x),q_{n-1}(x), \dots , q_0(x) \in GR(p^a,m)[x]$\ where
\be
  g(x) = q_n(x)h(x)^n + \cdots + q_1(x)h(x) + q_0(x)\nn
\ee where for $0 \le i \le n$, either $\deg(q_i(x)) < \deg(h(x))$\
or $q_i(x) = 0$. With some manipulation $g(x)$\ can be represented
in the following form \be
  \label{Local.Subambients.Equation.gx.in.terms.of.hx}
  g(x) = p^{j_0}h(x)^{i_0}\alpha_0(x) + \cdots + p^{j_r}h(x)^{i_r}\alpha_r(x)
\ee
where $0 \le r \le a-1$\ and
\begin{itemize}
 \item $\alpha_i(x) \not \in \id{p,h(x)}$
 \item $0 \le j_0 < \cdots < j_r \le a-1$
 \item $i_0 > \cdots > i_r \ge 0$.
\end{itemize}

Since $f(x)$ is regular and monic, $g(x)$ can be divided by $f(x)$ initially.
Then it is not hard to see that for some $q(x)\in\R$
\be
  \label{equ.form}
  g(x)=q(x)f(x)+p^{j_0}h(x)^{i_0}\alpha_0(x)+\dots+p^{j_r}h(x)^{i_r}\alpha_r(x)\nn
\ee
where
  $r$, $\alpha_i(x)$, $j_e$\ and $i_{\ell}$\ are as above with $t > i_0$.

In \cite{Dougherty-Park_2007} and \cite{SanLing-etal_2008}, a unique
generating set for an ideal of
$\frac{GR(p^a,m)[x]}{\id{x^{p^s}-1}}$\ was developed. The polynomial
$x^{p^s}-1$ is of the type $f(x)$ is. Notice
$x^{p^s}-1=(x-1)^{p^s}+p\beta(x)$. We will now find a similar
generating set for an ideal of $\Rf$.

\begin{definition}[{cf. \cite[Definition  6.1]{Dougherty-Park_2007}}]
  \label{Local.Subambients.Lemma.Tor.I}
 Let \Code .
 For $0 \le i \le a-1$, define
 \be\nn
  Tor_i(C) = \{ \mu(v): \quad p^iv \in C \}.
 \ee
\end{definition}
$Tor_i(C)$\ is called the $i^{th}$\ \textit{torsion code} of $C$.
$Tor_0(C) = \mu(C)$\ is usually called the residue code of $C$.
Note that for a code $C$\ over $GR(p^a,m)$, we have $Tor_i(C) \subset Tor_{i+1}(C)$.

\begin{lemma}
  \label{Local.Subambients.Lemma.Torsion.Ideal}
 Let \Code .
 Then
 \be
  Tor_i(C) = \id{h(x)^{T_i} + \id{f}}\subset \frac{GR(p,m)[x]}{ \id{f(x)} } \nn
 \ee
 for some $0 \le T_i \le t$.
\end{lemma}
\begin{proof}
 Since $C \vartriangleleft \sR$,
 $Tor_i(C) \vartriangleleft \frac{GR(p,m)[x]}{\id{\bar{f}(x)}}$.
 The claim follows by Lemma \ref{Local.Subambients.Lemma.Finite.Field.Chain}.
\end{proof}

\begin{definition}
  \label{Local.Subambients.Lemma.Torsional.Degree}
 In Lemma \ref{Local.Subambients.Lemma.Torsion.Ideal}, $T_i$\ is the
 $i^{th}$\ \textit{torsional degree} of $C$\ which we denote by $T_i(C)$.
 The torsional degrees form a non-increasing sequence, i.e.,
 $ t \ge T_0(C) \ge \cdots \ge T_{a-1}(C) \ge 0. $
\end{definition}

For any $ \xi(x) + \id{f} \in \Rf$, we can divide $\xi(x)$\ by
$f(x)$, as $f(x)$\ is regular, and get $\xi(x) = q(x)f(x) + r(x)$\
such that either $r(x) = 0$\ or $\deg(r(x)) < \deg(f(x))$. So
$\xi(x) + \id{f} = r(x) + \id{f}$. This implies that
$\Rf=\{a(x)+\id{f}:a(x)\in GR(p^a,m)[x],\deg(a(x))<\deg(f(x))\}$.
Throughout the remainder for this section, the elements of $\Rf$\
will be represented as polynomials of degree less than $\deg(f(x))$.

Definitions \ref{Local.Subambients.Lemma.Tor.I} and
\ref{Local.Subambients.Lemma.Torsional.Degree} and  Lemma
\ref{Local.Subambients.Lemma.Torsion.Ideal} are expansions to
polycyclic codes of the ideas first presented in Section 6 of
\cite{Dougherty-Park_2007} in the context of cyclic codes. The
following theorem is a generalization of Theorem 6.5 of
\cite{Dougherty-Park_2007}.

\begin{theorem}
  \label{Local.Subambients.Theorem.Generating.Set.Existence}
 Let \Code . Then
 $ C = \id{F_0(x), pF_1(x), \dots , p^{a-1}F_{a-1}(x)}$\ where
 $F_i(x) = 0$\ if $T_i(C) = t$, and $F_i(x) = h(x)^{T_i(C)} + p\gamma_i(x)$\ for some
 $\gamma_i(x) \in GR(p^a,m)[x]$, if $T_i(C) < t$.
\end{theorem}
\begin{proof}
 Denote $T_i(C)$ by $T_i$. If $ C = 0$, we are done.
 So assume $C \neq 0$. Let $r$\ be the smallest nonnegative integer such that $T_r < t$.
 For every $0 \le i \le r-1$, set $F_i(x) = 0$.
 For $r \le i \le a-1$, pick $F_i(x) \in GR(p^a,m)[x]$\ such that $p^iF_i(x) \in C$\ and
 $\mu (F_i(x)) = h(x)^{T_i}$. So, $F_i(x) = h(x)^{T_i} + p \gamma_i(x)$\ for some
 $\gamma_i(x) \in \Rf$. Note that such an $F_i(x)$\ exists because
 $Tor_i(C) = \id{h(x)^{T_i}} \vartriangleleft \frac{GR(p,m)[x]}{\id{f(x)}}$.
 Let $g(x) \in C$.
 As was shown earlier (see Equation (\ref{Local.Subambients.Equation.gx.in.terms.of.hx})),
 \be
  \label{Local.Subambients..Theorem.Generating.Set.Existence.Equation.gx}
  g(x) = p^{j_0}(h(x)^{i_0}\sigma_{j_0}(x) + p\beta_0(x))
 \ee
 for some $\sigma_{j_0}(x), \beta_0(x) \in GR(p^a,m)[x]$\ where $i_0 < t$\
 and $\sigma_{j_0}(x) \neq 0$. Let $\sigma_0(x) = \cdots = \sigma_{j_0 -1}(x) = 0$.
 Let
 \be
  g_1(x) = g(x) - p^{j_0}h(x)^{i_0 - T_{j_0}}\sigma_{j_0}(x)F_{j_0}(x). \nn
 \ee
 Note that since $Tor_{j_0}(C) = \id{h(x)^{T_{j_0}}}$, it follows by
 (\ref{Local.Subambients..Theorem.Generating.Set.Existence.Equation.gx})
 and the fact that $\sigma_{j_0}(x)$\ is a unit in $\frac{GR(p^a,m)[x]}{\id{f(x)}}$\
 that $i_0 \ge T_{j_0}$.
 Since $T_{j_0} < t$,
 we have
 \be g_1(x) & = & p^{j_0} (h(x)^{i_0}\sigma_{j_0}(x) + p\beta_0(x)) -
        p^{j_0}h(x)^{i_0 - T_{j_0}}\sigma_{j_0}(x)[h(x)^{T_{j_0}} + p\gamma_{j_0}(x)] \nn\\
        & = & p^{j_0 + 1}\beta_0(x) - p^{j_0 + 1}h(x)^{i_0 - T_{j_0}}\sigma_{j_0}(x)\gamma_{j_0}(x).  \nn
  \ee
 So, $g_1(x) \in \id{p^{j_0 + 1}} \cap C $. If $g_1(x) = 0$, let
 $\sigma_{j_0 + 1}(x) = \cdots = \sigma_{a-1}(x) = 0$\ and we are done.
 If not, then, as was done with $g(x)$, we can view $g_1(x)$\ as
 \be
  g_1(x) = p^{j_1}(h(x)^{i_1}\sigma_{j_1}(x) + p\beta_1(x)) \nn
 \ee
 for some $\sigma_{j_1}(x),\beta_1(x) \in GR(p^a,m)[x]$\ where $i_1 < t$, $j_0 < j_1 $\ and
 $\sigma_{j_1}(x) \neq 0$. Let $\sigma_{j_0 + 1}(x) = \cdots = \sigma_{j_1 - 1}(x) = 0$.
 Let
 \be
  g_2(x) = g_1(x) - p^{j_1}h(x)^{i_1 - T_{j_1}}\sigma_{j_1}(x)F_{j_1}(x).\nn
 \ee
 Since $T_{j_1} < t$,
 we have
 \be
  g_2(x) & = & p^{j_1}(h(x)^{i_1}\sigma_{j_1}(x)+ p\beta_1(x))
        - p^{j_1}h(x)^{i_1 - T_{j_1}}\sigma_1(x)[h(x)^{T_{j_1}} + p \gamma_{j_1}(x)] \nn\\
      & = & p^{j_1 + 2}\beta_1(x) - p^{j_1 + 1}h(x)^{i_1 - T_{j_1}}\sigma_1(x)\gamma_{j_1}(x).\nn
 \ee
 So $g_2(x) \in \id{p^{j_1 + 1}} \cap C $. If $g_2(x) = 0$, then let $\sigma_{j_1+1}(x) = \cdots = \sigma_{a-1}(x) = 0$.
 Note that since $j_0 < j_1 < a$, this is a finite process. So
 \be
  g(x) = \sum_{i = 0}^{a-1}p^ih(x)^{i - T_i}\sigma_i(x)F_i(x) \in \id{ F_0(x),pF_1(x), \dots, p^{a-1}F_{a-1} }. \nn
 \ee
 Hence $C \subset \id{F_0(x), pF_1(x), \dots, p^{a-1}F_{a-1}(x)}$.
 Since $p^iF_i(x)\in C$, for all $0 \le i \le a-1$, we have the equality
 $$ C = \id{F_0(x), pF_1(x), \dots, p^{a-1}F_{a-1}(x)}. $$
\end{proof}

As was stated in \cite{SanLing-etal_2008}, Theorem 6.5 of
\cite{Dougherty-Park_2007} does not provide a unique set of
generators. Neither does our generalization in Theorem
\ref{Local.Subambients.Theorem.Generating.Set.Existence}. We now
show, as in \cite{SanLing-etal_2008}, that there does exist a unique
set of generators given some extra constraints. 
Although this is a
generalization of Theorem 2.5 in \cite{SanLing-etal_2008}, 
the proof here only differs from that one in a few details. 
However, we present the
proof in its entirety here for the sake of completeness.

We would like to point out that there is a little inaccuracy in the statement of Theorem 2.5 in
\cite{SanLing-etal_2008}.
Let $\sT_m[u]$\ be the set of polynomials in $u$\ whose coefficients are in $\sT_m$. 
The $h_{j,\ell}(u)$\ in their theorem is said to be an element of $\sT_m[u]$\
which is not necessarily true. What is true is that $h_{j,\ell}(u)$\ is either 0 or a unit and that
\be
 h_{j,\ell}(u) = \sum_{k=0}^{T_{\ell + j} - 1}c_{k,j,\ell}(u-1)^k \nn
\ee
with $c_{k,j,\ell}\in \sT_m$\ and $c_{0,j,\ell} \neq 0$.
It should also be pointed out that $h_{j,\ell}(u)$\ is a unit precisely because $(u-1)$\
is nilpotent (which is not stated but fairly easy to show) and $c_{0,j,\ell}$\ is a unit.

\begin{theorem}
  \label{Local.Subambients.Theorem.Unique.Generating.Set}
  Let \Code .
  Then there exist $f_0(x),f_1(x), \dots , f_{a-1}(x) \in \sR$\
  such that
  $$ C = \id{ f_0(x), pf_1(x),\dots , p^{a-1}f_{a-1}(x) } $$
  where $f_i(x) = 0$, if $T_i(C) = t$\ otherwise
  $$ f_i(x) = h(x)^{T_i(C)} + \sum_{j=1}^{a-1-i}p^{j}h(x)^{t_{i,j}}\alpha_{i,j}(x) $$
  where $t_{i,j}\deg(h(x)) + \deg(\alpha_{i,j}(x)) < T_{i+j}(C)\deg(h(x))$\
  and each $\alpha_{i,j}(x)\notin\id{p,h(x)}\setminus\{0\}$.

  Furthermore, the set $\{ f_0(x), pf_1(x), \dots ,p^{a-1}f_{a-1}(x) \}$\ is the unique generating set
  with these properties.
\end{theorem}

\begin{proof}
Denote $T_i(C)$ by $T_i$. When $C=0$, the result holds. Assume
$C\neq 0$. By Theorem
\ref{Local.Subambients.Theorem.Generating.Set.Existence},
$C = \id{F_0(x),pF_1(x),\dots,p^{a-1}F_{a-1}(x)}$ where $F_i(x)=0$
when $T_i=t$, otherwise $F_i(x)=h(x)^{T_i}+p\gamma_i(x)$ for some
$\gamma_i(x)\in \R$. The torsional degrees of $C$ form the
non-increasing sequence $t\geq T_0\geq\dots \geq T_{a-1}\geq 0$.
Since $C\neq \{0\}$ there is a least positive integer $r$ such that
$t>T_r\geq\dots\geq T_{a-1}\geq 0$. For $0\leq i \leq r-1$,
$F_i(x)=0$. Let $f_i(x)=0$ for $0\leq i \leq r-1$. For $r\leq i \leq
a-1$, $F_i(x)\neq 0$. Since we are considering $p^iF_i(x)$ and
$F_i(x)$ can be put in the form as shown in equation
(\ref{Local.Subambients.Equation.gx.in.terms.of.hx}), without loss
of generality we can write
\[
F_i(x)=h(x)^{T_i}+\sum_{j=1}^{a-1-i}p^j\sum_{k=0}^{t-1}h(x)^kq_{i,j,k}(x)
\]
where $q_{i,j,k}(x)=\sum_{l=0}^{\deg{h}-1}b_{i,j,k,l}x^l$ with
$b_{i,j,k,l}\in \Tm$.

Let
\[
f_{a-1}(x)=F_{a-1}(x)=h(x)^{T_{a-1}}.
\]

Now,
\begin{eqnarray*}
  F_{a-2}(x)&=&h(x)^{T_{a-2}}+p\sum_{k=0}^{t-1}h(x)^kq_{a-2,1,k}(x)\\
  &=&h(x)^{T_{a-2}}\\
  &&+p\left[\sum_{k=0}^{T_{a-1}-1}h(x)^kq_{a-2,1,k}(x)+h(x)^{T_{a-1}}
      \sum_{k=T_{a-1}}^{t-1}h(x)^{k-T_{a-1}}q_{a-2,1,k}(x)\right].
\end{eqnarray*}
Let
\begin{eqnarray*}
f_{a-2}(x)&=&F_{a-2}(x)-pf_{a-1}(x)\sum_{k=T_{a-1}}^{t-1}h(x)^{k-T_{a-1}}q_{a-2,1,k}(x)\\
&=&F_{a-2}(x)-ph(x)^{T_{a-1}}\sum_{k=T_{a-1}}^{t-1}h(x)^{k-T_{a-1}}q_{a-2,1,k}(x)\\
&=&h(x)^{T_{a-2}}+p\sum_{k=0}^{T_{a-1}-1}h(x)^kq_{a-2,1,k}(x)\\
&=&h(x)^{T_{a-2}}+ph(x)^{t_{a-2,1}}\sum_{k=t_{a-2,1}}^{T_{a-1}-1}h(x)^{k-t_{a-2,1}}q_{a-2,1,k}(x)
\end{eqnarray*}
where $t_{a-2,1}$ is the smallest $k$ such that $q_{a-2,1,k}(x)\neq
0$ if such a $k$ exists, otherwise\\
$\sum_{k=t_{a-2,1}}^{T_{a-1}-1}h(x)^{k-t_{a-2,a-1}}q_{a-2,1,k}(x)=0$
and $t_{a-2,1}$ can be arbitrary. It is easy to see
\[
C=\id{F_0(x),pF_1(x),\dots,p^{a-3}F_{a-3}(x),p^{a-2}f_{a-2}(x),p^{a-1}f_{a-1}(x)}
\]
and that $f_{a-2}(x)$ and $f_{a-1}(x)$ satisfy the conditions in the
theorem.

We proceed by induction. Assume $f_{i+1}(x),\dots,f_{a-1}(x)$ satisfy the
conditions of the theorem and that
\[
C=\id{F_0(x),pF_1(x),\dots,p^iF_i(x),p^{i+1}f_{i+1}(x),\dots,p^{a-1}f_{a-1}(x)}.
\]
After subtracting appropriate multiples of
$p^{i+1}f_{i+1}(x),\dots,p^{a-1}f_{a-1}(x)$ from $F_i(x)$ we can find an
element $f_i(x)$ such that
\begin{eqnarray*}
f_i(x)&=&h(x)^{T_i}+\sum_{j=1}^{a-1-i}p^{j}\sum_{k=0}^{T_{i+j}-1}h(x)^{k}g_{i,j,k}(x)\\
&=&h(x)^{T_i}+\sum_{j=1}^{a-1-i}p^{j}h(x)^{t_{i,j}}\sum_{k=t_{i,j}}^{T_{i+j}-1}h(x)^{k-t_{i,j}}g_{i,j,k}(x)
\end{eqnarray*}
where $g_{i,j,k}(x)=\sum_{l=0}^{\deg{h}-1}c_{i,j,k,l}x^l$ for some
$c_{i,j,k,l}\in \Tm$ and for fixed $j$, $t_{i,j}$ is the smallest
$k$ such that $g_{i,j,k}(x) \neq 0$ if such a $k$ exists, otherwise
$\sum_{k=t_{i,j}}^{T_{i+j}-1}h(x)^{k-t_{i,j}}g_{i,j,k}(x)=0$ and
$t_{i,j}$ can be arbitrary. Let
$\alpha_{i,j}(x)=\sum_{k=t_{i,j}}^{T_{i+j}-1}h(x)^{k-t_{i,j}}g_{i,j,k}(x)$.
If $\alpha_{i,j}(x)\neq 0$, $\alpha_{i,j}(x)$ is a unit since
$\alpha_{i,j}(x)\notin\id{p,h(x)}$. It is easy to see that
\[
C=\id{F_0(x),pF_1(x),\dots,p^{i-1}F_{i-1}(x),p^{i}f_{i}(x),\dots,p^{a-1}f_{a-1}(x)}
\]
and $f_{i}(x),\dots,f_{a-1}(x)$ satisfy the conditions in the
theorem. Hence, we have $f_{0}(x),\dots,f_{a-1}(x)$ such that
\[
C = \id{f_{0}(x),pf_{1}(x),\dots,p^{a-1}f_{a-1}(x)}.
\]

Now we show the uniqueness of such a generating set. Assume that
$f^{'}_0(x),\cdots,f^{'}_{a-1}(x)$\ also satisfy the conditions in
the theorem. Say
\begin{equation*}
 f_i(x) =  h(x)^{T_i} + \sum_{j=1}^{a-1-i}p^j\sum_{k=0}^{T_{i+j} - 1}h(x)^k g_{i,j,k}(x)
\end{equation*}
and
\begin{equation*}
f_i^{'}(x)  =  h(x)^{T_i} + \sum_{j=1}^{a-1-i}p^j\sum_{k =
0}^{T_{i+j} - 1}h(x)^k g^{'}_{i,j,k}(x)
\end{equation*}
where $g_{i,j,k}(x),g'_{i,j,k}(x)\in\Tm[x]$ of degree less than
$h(x)$. Assume $f_i(x)-f_i^{'}(x)\neq 0$. Then for some $j,k$,
$g_{i,j,k}(x)-g^{'}_{i,j,k}(x)\neq 0$. Let $j_0$ be the smallest $j$
in the above sum such that $g_{i,j,k}(x)-g^{'}_{i,j,k}(x)\neq 0$.
Then
\begin{equation*}
p^{i}(f_i(x)-f_i^{'}(x))  =
p^{i+j_0}\sum_{j=j_0}^{a-1-i}p^{j-j_0}\sum_{k=0}^{T_{i+j} -
  1}h(x)^k(g_{i,j,k}(x)-g^{'}_{i,j,k}(x)).
\end{equation*}
Since the difference of two distinct elements of $\Tm$ is not
divisible by $p$, for all $j,k$ in the above sum, either
$g_{i,j,k}(x)-g^{'}_{i,j,k}(x)$ is $0$ or not divisible by $p$. By
the assumption on $j_0$ then, $p^{i}(f_i(x)-f_i^{'}(x))\in
C\bigcap\id{p^{i+j_0}}\setminus\id{p^{i+j_0+1}}$. Since this is a
nonzero element of $C$ with degree less than $T_{i+j_0}\deg(h)$,
this contradicts the definition of $T_{i+j_0}$. Hence
$f_i(x)=f_i^{'}(x)$.
\end{proof}

Now, in Corollary \ref{prop.ideal}, we show that
if we remove the redundant generators in
Theorem \ref{Local.Subambients.Theorem.Unique.Generating.Set},
then we obtain a result similar to \cite[Theorem 4.1]{Salagean_Disc_2006}.
There they prove it in a slightly
different setting namely $GR(p^a,m)$ is replaced by an arbitrary finite
chain ring and $f(x)$ is either $x^n-1$ or $x^n+1$ (i.e., cyclic and
negacyclic codes over a finite chain ring). We will also prove this
result later in the case that $f(x)$ is an arbitrary regular
polynomial.

\begin{definition}[{adapted from \cite[Definition 4.1]{Salagean_AAECC}}]
  \label{def.gen_set} Let $G=\{p^{j_0}f_{j_0}(x),\dots,p^{j_r}f_{j_r}(x)\}\subset
  \Rf$, for some $0\leq r\leq a-1$, such that
  \begin{enumerate}
    \item $0\leq j_0<\dots<j_{r}\leq a-1$ \label{def.gen_set.one},
    \item $t>k_{j_0}>\dots>k_{j_{r}}\geq 0$ \label{def.gen_set.two},
    \item $f_{j_i}(x)=h(x)^{k_{j_i}}+\sum_{{\ell}=1}^{a-1-{j_i}}p^{{\ell}}h(x)^{t_{{j_i},{\ell}}}\alpha_{{j_i},{\ell}}(x)$
    where
    $t_{{j_i},{\ell}}\deg(h(x))+\deg(\alpha_{{j_i},{\ell}}(x))<k_{j_i}\deg(h(x))$\
    and each $\alpha_{{j_i},{\ell}}(x)\notin\id{p,h(x)}\setminus\{0\}$,
    \label{def.gen_set.three}
    \item
    $p^{j_{i+1}}f_{j_i}(x)\in\id{p^{j_{i+1}}f_{j_{i+1}}(x),\dots,p^{j_{r}}f_{j_r}(x)}$,
    \label{def.gen_set.four}
    \item $p^{j_0}f(x)\in\id{p^{j_0}f_{j_0}(x),\dots,p^{j_r}f_{j_r}(x)}$ in
    $GR(p^a,m)[x]$. \label{def.gen_set.five}
  \end{enumerate}
  The set $G$ is called a \textit{generating set in standard form}.
  Moreover, by \cite[Theorem 5.4]{Salagean_FFA}, the set $G$\ is a minimal strong Groebner basis.
\end{definition}

\begin{corollary}
  \label{prop.ideal}
 Let \Code. There exists a generating set in standard form for $C$.
 \end{corollary}
\begin{proof}
Let $\{f_0(x),\dots,p^{a-1}f_{a-1}(x)\}$ be a generating set for $C$
as in Theorem \ref{Local.Subambients.Theorem.Unique.Generating.Set}.
Let $j_0=\min\{i|f_i(x)\neq 0\}$ and set $k_i=T_i(C)$. Then
\[
C=\id{p^{j_0}f_{j_0}(x),\dots,p^{a-1}f_{a-1}(x)}.
\]
Assume there exist Torsional degrees of $C$, $T_i,T_{i+1}$, such
that
  $T_i=T_{i+1}$ for some $i\geq j_0$. It should be clear that
  $p^{i+1}f_{i+1}(x)\in\id{p^if_i(x),p^{i+2}f_{i+2}(x),\dots,p^{a-1}f_{a-1}(x)}$.
  So after removing these unnecessary generators we have, for some $r$
  such that ${1\leq r \leq a-1}$, 
  \[
    C=\id{p^{j_0}f_{j_0}(x),\dots,p^{j_r}f_{j_r}(x)}.
  \]
 Then the properties (\ref{def.gen_set.one})-(\ref{def.gen_set.four}) of Definition \ref{def.gen_set} are satisfied.

  Now, assume $p^{j_0}f(x)\notin\id{p^{j_0}f_0(x),\dots,p^{j_r}f_{r}(x)}$ in
  $GR(p^a,m)[x]$.
    We consider
    \be
      g_{j_0}(x) & = & p^{j_0}f(x) - h(x)^{t-T_{j_0}}p^{j_0}f_{j_0}(x)\nn\\
          & = & p^{k_0}h(x)^{z_{k_0}}\alpha_{z_{k_0}}(x) + \cdots + p^{k_{e}}h(x)^{z_{k_e}}\alpha_{z_{k_e}}(x)
          \label{prop.ideal.gj0.canonical.form}
    \ee
  where the representation (\ref{prop.ideal.gj0.canonical.form})
  is as in (\ref{Local.Subambients.Equation.gx.in.terms.of.hx}).
  Note that $g_{j_0}(x)\in C$\ when we consider $g_{j_0}(x)$\ as an element of  $\sR$.
  If $k_0 < j_r$, say $j_{q-1} \le k_0 < j_q$\ for some $q \le r$,
  then $z_{k_0} \ge T_{j_{q-1}}$\ otherwise we get a contradiction to the torsional degree.
  Now, for an appropriate polynomial, say $\upsilon(x)$, we get
  \be
      g_{j_{q-1}} & = & g_{j_0}(x) - \upsilon(x)p^{j_{q-1}}f_{j_{q-1}}(x) \nn \\
          & = &  p^{\ell_0}h(x)^{y_{\ell_0}}\alpha_{y_{\ell_0}}(x) + \cdots +
              p^{\ell_{e^{'}}}h(x)^{y_{\ell_{e^{'}}}}\alpha_{y_{\ell_{e^{'}}}}(x)
                \label{prop.ideal.gjq1.canonical.form}
  \ee
  where the representation (\ref{prop.ideal.gjq1.canonical.form})
  is as in (\ref{Local.Subambients.Equation.gx.in.terms.of.hx}) and
  $\ell_0 > k_0$. Continuing like this, we obtain  a non-zero polynomial
  $g(x)\in\id{p^{j_r}}$\ such that
  \[
    p^{j_0}f(x)=\sum_{i=0}^{r}p^{j_i}f_i(x)\beta_i(x)+g(x),
  \]
  where $\deg{g(x)}<\deg{f_r(x)}$. Now, in $\Rf$
  \[
    g(x)=-\sum_{i=0}^{r}p^{j_i}f_i(x)\beta_i(x).
   \]
  So, $g(x)\in C$. But, $T_{j_r}\deg{h(x)}>\deg{g(x)}$ which is a
  contradiction of the torsional degree. Hence
  (\ref{def.gen_set.five}) of Definition \ref{def.gen_set}
  holds.
\end{proof}

\begin{corollary}
  Let \Code . Then $C$ is at most $\min\{a,t\}$-generated.
 \end{corollary}
\begin{proof}
  Follows from the facts that the number of distinct torsional degrees
  that are degrees of generators in the generating set in Corollary
  \ref{prop.ideal} is less than $t$ and that the number of generators
  there does not exceed $a$.
\end{proof}

Now we observe a relation
between the generating sets introduced in
\cite[Theorem 2.5]{SanLing-etal_2008} and 
generating sets in standard form
for cyclic codes studied in \cite{Salagean_FFA}.

\begin{remark}
  \label{remark.from.torsional.to.groebner}
  By \cite[Theorem 3.2]{Salagean_FFA} and Corollary \ref{prop.ideal},
  a generating set as in Theorem
  \ref{Local.Subambients.Theorem.Unique.Generating.Set} (and in
  particular, in \cite[Theorem 2.5]{SanLing-etal_2008}) for $C\lhd\Rf$
  is actually a strong Groebner basis (see \cite[Definition 3.8]{Salagean_Bulletin} for a definition). 
  Moreover, given a generating
  set $G$ as in Theorem
  \ref{Local.Subambients.Theorem.Unique.Generating.Set}, if we remove
  the redundant elements from $G$, as described in the proof of
  Corollary \ref{prop.ideal}, we obtain a generating set as in
  Corollary \ref{prop.ideal}, i.e., a generating  set in standard form
  which is a minimal strong Groebner basis, for $C$.
\end{remark}

Our final result of this section shows that if one can produce a
generating set in standard form, the torsional degrees can easily be
found.

\begin{theorem}
  \label{theorem.gen.set.std.form.gives.tor.degrees}
  Let $\{p^{j_0}f_{j_0}(x),\dots,p^{j_r}f_{j_r}(x)\}$ be a generating
  set in standard form for $C\lhd \Rf$\ where $ f_{j_i}(x) = h(x)^{k_{j_i}} + p\beta_{j_i}(x) $\ for some $\beta_{j_i}(x) \in \sR$. 
  Then for $e<j_0$, $T_e(C)=t$;
   for $j_i\leq e<j_{i+1}$, $T_{e}(C) = k_{j_i}$ and for $e \ge j_r $, $T_e(C) = k_{j_r}$.
\end{theorem}
\begin{proof}
  For $e<j_0$, $Tor_e(C)=0$ so $T_e(C)=t$. Clearly, $T_{j_i}(C)\leq
  k_{j_i}$ and $T_{j_0}(C)=k_{j_0}$. Now, let $j_i\leq e<j_{i+1}$ for
  some $i$. There exists a polynomial $f_e(x)=h(x)^{T_e(C)}+p\rho(x)$
  where $\deg(\rho(x))<\deg(h(x))T_e(C)$ such that $p^ef_e(x)\in C$.
  In the following we are working in $\R$. Since $e \ge j_0$, we have
  \[
  p^ef_e(x)\in
  \id{p^{j_0}f_{j_0}(x),\dots,p^{j_r}f_{j_r}(x),p^{j_0}f(x)}.
  \]
  By \ref{def.gen_set}(\ref{def.gen_set.five}),
  \[
  p^ef_e(x)\in\id{p^{j_0}f_{j_0}(x),\dots,p^{j_r}f_{j_r}(x)}.
  \]
  We know $T_e(C)\leq k_{j_i}$. Assume $T_e(C)<k_{j_i}$. By the
  properties in \ref{def.gen_set}(\ref{def.gen_set.two}) and
  \ref{def.gen_set}(\ref{def.gen_set.three}),
  $\deg{f_{j_0}(x)}>\dots>\deg{f_{j_i}(x)>\deg{f_e(x)}}$ which implies
  \[
  p^ef_e(x)\in\id{p^{j_{i+1}}f_{j_{i+1}}(x),\dots,p^{j_r}f_{j_r}(x)}.
  \]
  This is a contradiction since by the property
  \ref{def.gen_set}(\ref{def.gen_set.one}), $e<j_{i+1}<\dots<j_{r}\leq
  a-1$ which implies
  \[
  p^ef_e(x)\notin\id{p^{j_{i+1}}f_{j_{i+1}}(x),\dots,p^{j_r}f_{j_r}(x)}.
  \]
  So, $T_e(C)=k_{j_i}$. For $e \ge j_r$, the proof is similar.
\end{proof}

\begin{remark}
 Remark \ref{remark.from.torsional.to.groebner} and Theorem \ref{theorem.gen.set.std.form.gives.tor.degrees}
 imply that we can go back and forth between a generating set as in 
 Theorem \ref{Local.Subambients.Theorem.Unique.Generating.Set} and a generating set in standard form.
 Given a generating set as in  Theorem \ref{Local.Subambients.Theorem.Unique.Generating.Set},
 we can obtain a generating set in standard form as explained in Remark \ref{remark.from.torsional.to.groebner}.
 Conversely, suppose that we are given a generating set $G = \{p^{j_0}f_{j_0}(x),\dots,p^{j_r}f_{j_r}(x)\}$ in standard form.
 We know, by Theorem \ref{theorem.gen.set.std.form.gives.tor.degrees}, that 
 $f_{j_i}(x) = h(x)^{T_{j_i}} + p\beta_{j_i}(x)$.
 Define $F_e(x) = 0$ for $0 \le e < j_0$, $F_e(x) = p^{e}f_{j_i}(x)$ for $j_i \le e < j_{i+1}$\ and
 $F_e(x) = p^{j_r}f_{j_r}(x)$ for $j_r \le e < a$.
 Then, by Theorem \ref{theorem.gen.set.std.form.gives.tor.degrees}, the set
 $G^{'} = \{F_0(x),pF_1(x),\dots , p^{a-1}F_{a-1}(x)\}$\ is as in Theorem \ref{Local.Subambients.Theorem.Generating.Set.Existence}.
 Now applying the operations in the proof of Theorem \ref{Local.Subambients.Theorem.Unique.Generating.Set} to $G^{'}$,
 we obtain a generating set as in Theorem \ref{Local.Subambients.Theorem.Unique.Generating.Set}.
\end{remark}

\section{Subambients in characteristic $p^2$}
\label{Subambiets.in.char.p2}

Throughout this section, we work in characteristic $p^2$ and we assume $f(x)\in \Rtwo[x]$ is a regular
primary polynomial and let $\Rtwof=\fRtwo$.

Recently, the Hamming distance of cyclic codes of length $2^s$\ over 
$GR(4,1)$\ has been determined in \cite{zhu_kai_2010}.
Applying  the results of Section \ref{Local.Subambients}, we extend this result in two ways.
First, we consider the problem for a more general class of linear codes which are called polycyclic codes.
We show how to obtain the torsional degrees of polycyclic codes over
a Galois ring of characteristic $p^2$.
This gives us the Hamming distance if the Hamming distance of the residue code is known.
Second, we generalize this result of \cite{zhu_kai_2010}
to cyclic codes of length $p^s$ over any Galois ring of characteristic $p^2$.
We explicitly determine the Hamming distance of all cyclic codes of length $p^s$ over $GR(p^2,n)$.

First, in Lemma \ref{Lemma.Ideals.Polycyclic.Characteristic.p2}, 
we classify all polycyclic codes in characteristic $p^s$\ where $f(x)$\ is a regular primary polynomial.
This also gives us a classification of all cyclic codes of length $p^s$.
Then, in Lemma \ref{lemma.tors.degrees.1} and Lemma \ref{lemma.tors.degrees.2}, 
we determine the torsional degrees of polycyclic codes.
Using this together with some observations on the polynomial $x^{p^s} -1$,
we determine the Hamming distance of all cyclic codes of length $p^s$\
in characteristic $p^2$ in Lemma \ref{lemma.tors.degrees.cyc}.

As was explained in
Section \ref{Local.Subambients}, without loss of generality, we can
assume $f(x)$ is monic, $f(x) = h(x)^t + p\beta(x)$ where $\beta(x)
\in \Rtwo[x]$ and either $\beta(x)=0$ or
$\deg{\beta(x)}<t\deg{h(x)}$. Also, we may assume $h(x)$ is a monic
basic irreducible polynomial. Moreover, if $\beta(x)\neq 0$ we can
express $\beta(x)$ as $\beta(x)=h(x)^v\beta'(x)$ such that
$\beta'(x)=\sum_{j=0}^{t-1-v}\gamma_j(x)h^{j}(x)$ where $v<t$,
$\gamma_0(x)\neq 0$,$\gamma_0(x) \not \in \id{p}$, $\gamma_j(x)\in\Rtwo[x]$ and
$\deg(\gamma_j(x))<\deg(h(x))$ (see the explanation in Section \ref{Local.Subambients}). 
Since we are working in characteristic $p^2$ we may also assume that
$\gamma_j(x)\in\Tm[x]$. This can be seen by noting that
$p\gamma_j(x)=p\overline{\gamma_j}(x)$.

Assume $C\lhd\Rtwof$. Since $C$ is finite we have that
$C=\id{f_1(x),\dots,f_n(x)}$ for $f_i(x)\in\Rtwof$ where $\deg(f_i(x))<\deg(f(x))$,
i.e. $C$ is finitely generated. Without loss of generality we can
assume that if $p\nmid f_i(x)$ then $f_i(x)$ is monic and if $p|f_i(x)$ that
the leading coefficient of $f_i(x)$ is p. We consider two cases here,
when $C\nsubseteq\id{p}$ and $C\subseteq\id{p}$. First assume
$C\nsubseteq\id{p}$. In this case, it can be shown by looking at the
representation (\ref{Local.Subambients.Equation.gx.in.terms.of.hx})
that if $p\nmid f_i(x)$ then $f_i(x)=h(x)^{k_i} +
ph(x)^{\ell_i}\delta_i(x)$ and that if $p|f_i(x)$,
$f_i(x) = ph(x)^{\ell_i}\delta_i(x)$ where $\delta_i(x)$ is a unit with
$\ell_i\deg(h(x))+\deg(\delta_i(x))<k_i\deg(h(x))$ where at least one
generator is not divisible by $p$. Let $k_i=\infty$ if not defined.
Let $j$ be such that $k_j=min\{k_i\}_{i=1}^n$. Let
$g_i(x)=f_i(x)-f_j(x)h(x)^{k_i-k_j}$ if $p\nmid f_i(x)$ and $g_i(x)=f_i(x)$ if
$p|f_i(x)$. Now, we see that
$C=\id{g_1(x),\dots,g_{j-1}(x),f_j(x),g_{j+1}(x),\dots,g_n(x)}$. Notice
$g_i(x) \in \Rtwof \cap \id{p}$ for $i\neq j$. 
Again, without loss of generality we may assume for $i\neq j$
that $g_i(x)=ph(x)^{\ell_i'}$. Let $j'$ be such that
$\ell_{j'}=min\{\ell_i'\}_{i=1}^n$. So,
$g_i(x)-g_{j'}(x)h(x)^{\ell_i'-\ell_{j}'}=0$. Hence, $C=\id{f_j(x), g_{j'}(x)}$.
Finally, if $k_j\leq \ell_{j'}$ then $f_j(x)|g_{j'}(x)$ and $C=\id{f_j(x)}$.
Now, assume $C\subseteq\id{p}$. Then $f_i(x)=ph(x)^{\ell_i}\delta_i(x)$
is a unit. Without loss of generality, we can assume
$f_i(x)=ph(x)^{\ell_i}$. As above let $j$ be such that
$\ell_j=min\{\ell_i\}_{i=1}^n$. So, $f_i(x)-f_j(x)h(x)^{\ell_k-\ell_j}=0$.
Hence, $C=\id{f_j(x)}$. From this discussion we have the following
lemma.
\begin{lemma}
  \label{Lemma.Ideals.Polycyclic.Characteristic.p2}
  Let $C\lhd\Rtwof$. Then $C$ can be expressed in one of the following
  forms.
 \begin{enumerate}
 \item
    \label{Lemma.Ideals.Polycyclic.Characteristic.p2.Zero.Space}
    $\id{0}$,
 \item
    \label{Lemma.Ideals.Polycyclic.Characteristic.p2.Whole.Space}
    $\id{1}$,
 \item
    \label{Lemma.Ideals.Polycyclic.Characteristic.p2.Zero.Space.p.Power.x-1}
    $\id{ph(x)^n}$,
 \item $\id{h(x)^k}$,
     \label{Lemma.Ideals.Polycyclic.Characteristic.p2.Single.Monic.Generator.0}
 \item $\id{h(x)^k+ph(x)^\ell\delta(x)}$,
     \label{Lemma.Ideals.Polycyclic.Characteristic.p2.Single.Monic.Generator.1}
 \item
      \label{Lemma.Ideals.Polycyclic.Characteristic.p2.Two.Generators.0}
      $\id{h(x)^k, ph(x)^n}$,
 \item
      \label{Lemma.Ideals.Polycyclic.Characteristic.p2.Two.Generators.1}
      $\id{h(x)^k+ph(x)^\ell\delta(x), ph(x)^n}$
\end{enumerate}
where in any case $k,\ell,n<t$, $\ell < n<k$ and
$\delta(x)=\sum_{j=0}^{k-1-\ell}{\eta_{j}(x)}h(x)^j$ where
$\eta_j(x)\in\Tm[x]$, $\eta_0(x)\neq 0$  and
$\deg(\eta_j(x))<\deg(h(x))$.
\end{lemma}
\begin{proof}
The only thing that needs justification is the fact that
$\delta(x)=\sum_{j=0}^{k-1-\ell}{\eta_{j}(x)}h(x)^j$ where
$\eta_j(x)\in\Tm[x]$, $\eta_0(x)\neq 0$ and
$\deg(\eta_j(x))<\deg(h(x))$. By the discussion before this lemma,
$\delta(x)$ is a unit so, $\delta(x)\notin\id{p,h(x)}$. 
By the discussion in Section \ref{Local.Subambients},
$\delta(x)=\sum_{j=0}^{k-1-\ell}{\eta_{j}(x)}h(x)^j$ where
$\eta_j(x)\in\Rtwo[x]$, $\eta_0(x)\neq 0$ and
$\deg(\eta_j(x))<\deg(h(x))$. Finally, $\eta_j(x)\in\Tm[x]$ since we
are working in characteristic $p^2$ which means
$p\eta_j(x)=p\overline{\eta_j}(x)$.
\end{proof}

The results of Section \ref{Local.Subambients} assume the torsional degrees of a code
are known. The next three lemmas will focus on finding the torsional
degrees of a code so we can apply the results of Section \ref{Local.Subambients} with
the ultimate goal of this section being the determination of the
Hamming distance of a code. 
For the following recall form the beginning of this section that
$t,v,h(x),\beta(x),\beta'(x),\gamma_j(x)$ are parameters of $f(x)$.


\begin{lemma}
\label{lemma.tors.degrees.1} Let $C\lhd \Rtwof$ and $n<t$. If
$C=\id{ph(x)^n}$ then $T_0(C)=t$ and $T_1(C)=n$.
\end{lemma}
\begin{proof}
The result on $T_0(C)$ is obvious. Since every codeword is divisible
by $p$ and $h(x)^n$, clearly $T_1(C)=n$.
\end{proof}

\begin{lemma}
\label{lemma.tors.degrees.2} Assume $\beta(x)=0$. Let $C\lhd
\Rtwof$, $k,\ell,n<t$, $n<k$, $\delta(x)\notin\id{p,h(x)}$ and
$\deg(\delta(x))<(k-\ell)\deg(h(x))$.
\begin{enumerate}
\item If $C=\id{h(x)^k}$ then $T_0(C)=k$ and $T_1(C)=k$.
\label{lemma.tors.degrees.2.two}
\item If $C=\id{h(x)^k+ph(x)^\ell\delta(x)}$ then $T_0(C)=k$ and $T_1(C)=\min(k,t-k+\ell)$.
\label{lemma.tors.degrees.2.three}
\item If $C=\id{h(x)^k, ph(x)^n}$ then $T_0(C)=k$ and $T_1(C)=\min(k,n)$.
\label{lemma.tors.degrees.2.four}
\item If $C=\id{h(x)^k+ph(x)^\ell\delta(x), ph(x)^n}$ then $T_0(C)=k$ and $T_1(C)=\min(k,t-k+\ell,n)$.
\label{lemma.tors.degrees.2.five}
\end{enumerate}
\end{lemma}
\begin{proof}
The results on $T_0(C)$ are obvious. We concentrate on $T_1(C)$.

\noindent (\ref{lemma.tors.degrees.2.two}) The only way to create a
codeword divisible by $p$ is to multiply the generator by $p$ or by
a large enough power of $h(x)$. Since $h(x)^t=f(x)=0$ in $\Rtwof$,
$h(x)^kh(x)^{t-k}=h(x)^t=f(x)=0$. Multiplying by any smaller
multiple of $h(x)$ will not produce a polynomial divisible by $p$.
Hence any codeword divisible by $p$ is divisible by $ph(x)^k$ and so
$T_1(C)=k$.

\noindent (\ref{lemma.tors.degrees.2.three}) Noting that
$(h(x)^k+ph(x)^\ell\delta(x))h(x)^{t-k}=h(x)^t+ph(x)^{t-k+\ell}\delta(x)=ph(x)^{t-k+\ell}\delta(x)$
and $p\left(h(x)^k+ph(x)^\ell\delta(x)\right)=p(h(x)^k$ we see that
$T_1(C)=\min(k,t-k+\ell)$ following similar arguments as in
(\ref{lemma.tors.degrees.2.two}).

\noindent (\ref{lemma.tors.degrees.2.four}) This can be argued
similar to (\ref{lemma.tors.degrees.2.two}).

\noindent (\ref{lemma.tors.degrees.2.five}) This can be argued
similar to (\ref{lemma.tors.degrees.2.three}).
\end{proof}

\begin{lemma}
\label{lemma.tors.degrees.3} Assume $\beta(x)\neq 0$. Let $C\lhd
\Rtwof$, $k,\ell,n<t$, $n<k$ and
$\delta(x)=\sum_{j=0}^{k-1-\ell}{\eta_{j}(x)}h(x)^j$ where
 $\eta_j(x)\in\Tm[x]$, $\eta_0(x)\neq 0$ and $\deg(\eta_j(x))<\deg(h(x))$.
\begin{enumerate}
\item If $C=\id{h(x)^k}$ then $T_0(C)=k$ and $T_1(C)=\min(k,v)$.
\label{lemma.tors.degrees.3.two}
\item If $C=\id{h(x)^k+ph(x)^\ell\delta(x)}$ then $T_0(C)=k$ and
\[
T_1(C)=\left\{
\begin{array}{lll}
\min(k,v,t-k+\ell)&\textrm{ if }v\neq t-k+\ell\\
\min(k,v+z)&\textrm{ if }v= t-k+\ell\\
\end{array} \right.
\]
where $z=\min\left(\{j|\gamma_j(x)\neq\eta_j(x)\}\cup\{t\}\right)$.
\label{lemma.tors.degrees.3.three}
\item If $C=\id{h(x)^k, ph(x)^n}$ then $T_0(C)=k$ and $T_1(C)=\min(k,v,n)$.
\label{lemma.tors.degrees.3.four}
\item If $C=\id{h(x)^k+ph(x)^\ell\delta(x), ph(x)^n}$ then $T_0(C)=k$ and
\[
T_1(C)=\left\{
\begin{array}{lll}
\min(k,v,t-k+\ell,n)&\textrm{ if }v\neq t-k+\ell\\
\min(k,v+z,n)&\textrm{ if }v= t-k+\ell\\
\end{array} \right.
\]
where $z=\min\left(\{j|\gamma_j(x)\neq\eta_j(x)\}\cup\{t\}\right)$.
\label{lemma.tors.degrees.3.five}
\end{enumerate}
\end{lemma}
\begin{proof}
The results on $T_0(C)$ are obvious. We concentrate on $T_1(C)$.

\noindent (\ref{lemma.tors.degrees.3.two}) The only way to create a
codeword divisible by $p$ is to multiply the generator by $p$ or by
a large enough power of $h(x)$. Now,
$h(x)^{t-k}h(x)^k=h(x)^t=-ph(x)^v\beta'(x)$. We know $\beta'(x)$ is
a unit since $\gamma_0(x)\neq 0$ so, $T_1(C)=\min(k,v)$.

\noindent (\ref{lemma.tors.degrees.3.three}) First,
\begin{eqnarray*}
h(x)^{t-k}\left(h(x)^k+ph(x)^\ell\delta(x)\right)&=&h(x)^t+ph(x)^{t-k+\ell}\delta(x)\\
&=&-ph(x)^v\beta'(x)+ph(x)^{t-k+\ell}\delta(x).\\
\end{eqnarray*}
If $v<{t-k+\ell}$ then
\[
-ph(x)^v\beta'(x)+ph(x)^{t-k+\ell}\delta(x)=-ph(x)^v\left(\gamma_0(x)+\sum_{j=1}^{t-1-v}\gamma_j(x)h^{j}(x)-h(x)^{t-k+\ell-v}\sum_{j=0}^{k-1-\ell}{\eta_{j}(x)}h(x)^j\right).
\]
In this case $T_1(C)=\min(k,v)$. If $v>{t-k+\ell}$ then
\[
-ph(x)^v\beta'(x)+ph(x)^{t-k+\ell}\delta(x)=ph(x)^{t-k+\ell}\left(\eta_0(x)+\sum_{j=1}^{k-1-\ell}{\eta_{j}(x)}h(x)^j-h(x)^{v-(t-k+\ell)}\sum_{j=0}^{t-1-v}\gamma_j(x)h^{j}(x)\right).
\]
In this case $T_1(C)=\min(k,t-k+\ell)$. Next, consider the case
$v={t-k+\ell}$. Here, if $\beta'(x)=\delta(x)$ then
$-ph(x)^v\beta'(x)+ph(x)^{t-k+\ell}\delta(x)=0$ so $T_1(C)=k$.
Finally, if $\beta'(x)\neq\delta(x)$ then for some $0\leq j'<t$,
$\gamma_{j'}(x)\neq\eta_{j'}(x)$. Since
$\gamma_j(x),\eta_j(x)\in\Tm[x]$ we have that
$\gamma_{z}(x)-\eta_{z}(x)$ is not divisible by $p$ and is therefore
a unit. Then
\[
-ph(x)^v\beta'(x)+ph(x)^{t-k+\ell}\delta(x)=-ph(x)^{v+z}\left(\gamma_z(x)-\eta_z(x)+\sum_{j={z+1}}^{t-1-v}\gamma_j(x)h^{j-z}(x)-\sum_{j=z+1}^{k-1-\ell}\eta_j(x)h^{j-z}(x)\right).
\]
Since $z\leq t-1-v$, in this final case, $T_1(C)=\min(k,v+z)$.

\noindent (\ref{lemma.tors.degrees.3.four}) This can be argued
similar to (\ref{lemma.tors.degrees.3.two}).

\noindent (\ref{lemma.tors.degrees.3.five}) This can be argued
similar to (\ref{lemma.tors.degrees.3.three}).
\end{proof}

Now that the torsional degrees of any code can be computed, the
techniques in Section \ref{Local.Subambients} can be applied to produce a generating set
as in Theorem \ref{Local.Subambients.Theorem.Unique.Generating.Set} or 
Definition \ref{def.gen_set}. Our goal here is to show how the hamming
distance can be computed. Notice in Section \ref{Local.Subambients} that ultimately
$T_{a-1}(C)$ will determine the Hamming distance of $C$, i.e.,
$d_H(C)=d_H\left( \overline{ \id{h(x)^{T_1(C)}} }\right)$.

In the remaining part of this section, we study cyclic codes of
length $p^s$\ over $\Rtwo$ and show how to determine their
Hamming distances. To do so we apply the results from the beginning
of this section. The following two lemmas are  immediate
consequences of Kummer's Theorem (see \cite{Granville_1997} for the
statement) which we will need for our calculations.

\begin{lemma}
\label{Lemma.P.Power.Divisor.of.Binomial.Coefficient} Let $k<p^{e}$
and let $\ell$ be the largest integer such that $p^{\ell}|k$. Then
$p^{e-\ell}|{ p^e\choose k}$.
\end{lemma}

\begin{lemma}
\label{Lemma.Binomial.Coefficient.of.bx} Let $0<i<p$. We have
${p^s\choose ip^{s-1}}=pu\in\Rtwo$ where $p\nmid u$.
\end{lemma}

To apply the results of this section, we need to show that
the ambient ring is of the correct type. To do so, we only need to show that an
appropriate polynomial is used for the generator of the ideal being factored out .
For cyclic codes of length
$p^s$, this polynomial is $x^{p^s}-1$ of course. We now show why this
is an appropriate polynomial. By Lemma
\ref{Lemma.P.Power.Divisor.of.Binomial.Coefficient} and Lemma
\ref{Lemma.Binomial.Coefficient.of.bx} and the fact that we are
working in $\Rtwo$,
\begin{eqnarray*}
x^{p^s} - 1&=&((x-1)+1)^{p^s}-1\\
&=&(x-1)^{p^s}+{p^s\choose p^s-1}(x-1)^{p^s-1}+\cdots+{p^s\choose
1}(x-1)\\
&=&(x-1)^{p^s}+{p^s \choose (p-1)p^{s-1} }(x-1)^{(p-1)p^{s-1}} + \cdots + {p^s \choose p^{s-1}}(x-1)^{p^{s-1}}\\
&=&(x-1)^{p^s}+p(x-1)^{p^{s-1}}\sum_{i=0}^{p-2}\frac{{p^s\choose
(i+1)p^{s-1}}}{p}(x-1)^{ip^{s-1}}\\
\end{eqnarray*}

We want to show that we can express $x^{p^s}-1$ in the form needed
to use the results form this section. Let $t=p^s$, $v=p^{s-1}$,
$h(x)=x-1$ and
$\beta'(x)=\sum_{i=0}^{p-2}\gamma_{ip^{s-1}}(x-1)^{ip^{s-1}}$ where
$\gamma_{ip^{s-1}}=\frac{{p^s\choose (i+1)p^{s-1}}}{p}\pmod{p}$ for
$0\leq i<p-1$ and $\gamma_j=0$ for all other $j$. Note,
$\gamma_j\in\Tm$. This shows that $x^{p^s}-1$ is the type of
polynomial we need.

The following is a special case of Lemma
\ref{Lemma.Ideals.Polycyclic.Characteristic.p2}.

\begin{lemma}
Let $C\lhd\cRtwo$. Then $C$ can be expressed in one of the following
forms.
\begin{enumerate}
\label{lemma.cyc}
\item $\id{0}$,
\label{lemma.cyc.one}
\item $\id{1}$,
\label{lemma.cyc.two}
\item $\id{p(x-1)^n}$,
\label{lemma.cyc.three}
\item $\id{(x-1)^k}$,
\label{lemma.cyc.four}
\item $\id{(x-1)^k+p(x-1)^\ell\delta(x)}$,
\label{lemma.cyc.five}
\item $\id{(x-1)^k, p(x-1)^n}$,
\label{lemma.cyc.six}
\item $\id{(x-1)^k+p(x-1)^\ell\delta(x), p(x-1)^n}$
\label{lemma.cyc.seven}
\end{enumerate}
where in any case $k,\ell,n<p^s$, $n<k$ and
$\delta(x)=\sum_{j=0}^{k-1-\ell}\eta_{j}(x-1)^j$ where
$\eta_j\in\Tm$ and $\eta_0\neq 0$.
\end{lemma}


Now, restating Lemma \ref{lemma.tors.degrees.1} and Lemma \ref{lemma.tors.degrees.3} for cyclic codes of length $p^s$
and using the fact that $d_H(C) =  d_H( \overline{\id{(x-1)^{T_1(C)}}} )$, we determine the
Hamming distance of all cyclic codes of length $p^s$\ over $GR(p^2,m)$\ in the following lemma.
Note that $ \overline{\id{(x-1)^{T_1(C)}}} $\ is a cyclic code of length $p^s$\ over $\F_{p^m}$\ and its Hamming distance
is given in Theorem \ref{Irreducible.Theorem.Main}.

\begin{lemma}
  \label{lemma.tors.degrees.cyc} Let $C\lhd \cRtwo$, $k,\ell,n<p^s$,
  $n<k$ and $\delta(x)=\sum_{j=0}^{k-1-\ell}\eta_{j}(x-1)^j$ where
  $\eta_j\in\Tm$ and $\eta_0\neq 0$. 
    Then $d_H(C) = d_H(\overline{ \id{(x-1)^{T_1(C)}} })$\ where $T_0(C)$\ and $T_1(C)$\ are as follows.
  
  \begin{enumerate}
    \item If $C=\id{(x-1)^k}$ then $T_0(C)=k$ and $T_1(C)=\min(k,p^{s-1})$.
      \label{lemma.tors.degrees.cyc.two}
    \item If $C=\id{(x-1)^k+p(x-1)^\ell\delta(x)}$ then $T_0(C)=k$ and
      \[
	T_1(C)=\left\{
	\begin{array}{lll}
	\min(k,p^{s-1},p^s-k+\ell)&\textrm{ if }p^{s-1}\neq p^s-k+\ell\\
	\min(k,p^{s-1}+z)&\textrm{ if }p^{s-1}= p^s-k+\ell\\
	\end{array} \right.
      \]
  where $z=\min\left(\{j|\gamma_j\neq\eta_j\}\cup\{p^s\}\right)$.
  \label{lemma.tors.degrees.cyc.three}
  \item If $C=\id{(x-1)^k, p(x-1)^n}$ then $T_0(C)=k$ and $T_1(C)=\min(k,p^{s-1},n)$.
  \label{lemma.tors.degrees.cyc.four}
  \item If $C=\id{(x-1)^k+p(x-1)^\ell\delta(x), p(x-1)^n}$ then $T_0(C)=k$ and
    \[
    T_1(C)=\left\{
    \begin{array}{lll}
    \min(k,p^{s-1},p^s-k+\ell,n)&\textrm{ if }p^{s-1}\neq p^s-k+\ell\\
    \min(k,p^{s-1}+z,n)&\textrm{ if }p^{s-1}= p^s-k+\ell\\
    \end{array} \right.
    \]
    where $z=\min\left(\{j|\gamma_j\neq\eta_j\}\cup\{p^s\}\right)$.
  \label{lemma.tors.degrees.cyc.five}
  \item If
    $C=\id{p(x-1)^n}$ then $T_0(C)=p^s$ and $T_1(C)=n$.
  \end{enumerate}
\end{lemma}


\section{Structure of polycyclic code ambients}
\label{sect.main}
In this section, we study the structure of the
code ambient for polycyclic codes over a Galois ring which is the
ring $\fR$ where $f(x)$ is a regular monic polynomial. Throughout
this section assume that $f(x)\in\R$ is regular. 
By Theorem \ref{Theorem.McDonald.Factorization.and.Uniqueness}, 
$f(x)=\delta(x)f_1(x)\cdots f_s(x)$ where
the $\delta(x)\in\R$ is a unit and $\{f_i(x)\in\R\}_{i=1}^s$ is a
set of regular primary co-prime polynomials that are not units.
By the fact that $\delta(x)$ is a unit, we may assume without loss
of generality that $f_i(x)=h_i(x)^{t_i}+p\beta_i(x)$ where $h_i(x)$
is a monic basic irreducible polynomial such that
$\overline{h_i}(x)=h_i(x)$. We know that
$t_i\deg{h_i(x)}>\deg{\beta_i(x)}$. Since we are interested in $\fR$
and $\id{f(x)}=\id{\delta(x)^{-1}f(x)}$, we assume $\delta(x)=1$, so
$f(x)=f_1(x)\cdots f_s(x)$. Additionally, throughout this section
let $\Rf=\fR$ and let $\hat{f_i}(x)=\prod_{j=1,j\neq i}^{s}f_j(x)$
for $1\leq i\leq s$.

\begin{theorem}
  \label{prop.princ_code}
    For $\sR$, we have the following
  \begin{enumerate}
    \item $\Rf=\bigoplus_{i=1}^{s}\id{\hat{f_i}(x)+\id{f}}$ and $\id{\hat{f_i}(x)+\id{f}}\cong \frac{\R}{\id{f_i(x)}}$,
      \label{prop.princ_code.one}
    \item Any maximal ideal of $\Rf$ is of the form
      $\id{p\hat{f_i}(x)+f_i(x)+\id{f},h_i\hat{f_i}(x)+f_i(x)+\id{f}}
       = \id{p + \id{f}, h_i(x) + \id{f}} $
      for some $1\leq i\leq s$, \label{prop.princ_code.three}
    \item
      $J\left(\Rf\right)=\bigcap_{i=1}^s\id{p+\id{f},h_i(x)+\id{f}}=\id{p+\id{f},\prod_{i=1}^sh_i(x)+\id{f}}$,
      \label{prop.princ_code.four}
     \item
    $soc\left(\Rf\right)=\bigoplus_{i=1}^s\id{p^{a-1}h_i(x)^{t_i-1}\hat{f_i}(x)+\id{f}}=\id{p^{a-1}\prod_{i=1}^sh_i(x)^{t_i-1}+\id{f}}$.
    \label{prop.princ_code.five}
  \end{enumerate}
\end{theorem}
\begin{proof}
  \noindent(\ref{prop.princ_code.one}) It is not hard to see that
  since the $f_i(x)$ are co-prime,
  $\cap\id{f_i(x)}=\prod\id{f_i(x)}=\id{f(x)}$ (see discussion on pg. 94 in
  \cite{McDonald_Book}). By the Chinese Remainder Theorem,\\
  \[
  \Rf\cong\bigoplus_{i=1}^{s}\frac{\R}{\id{f_i(x)}}.
  \]
  Define $\phi_i:\Rf\rightarrow\frac{\R}{\id{f_i(x)}}$ via
  $\phi_i:a(x)+\id{f}\mapsto a(x)+\id{f_i}$. Since
  $\id{\hat{f_i}(x)+\id{f}}=\{a(x)\hat{f_i}(x)+\id{f}|\deg{a(x)}<\deg{f_i(x)}\}$
  we have that $\id{\hat{f_i}(x)+\id{f}}\cong \frac{\R}{\id{f_i(x)}}$.

  \noindent(\ref{prop.princ_code.three})-(\ref{prop.princ_code.five})
  There exists idempotents
  $\hat{e_i}(x)+\id{f}\in\id{\hat{f_i}(x)+\id{f}}$ for $1\leq i\leq s$
  such that $\id{\hat{e_i}(x)+\id{f}}=\id{\hat{f_i}(x)+\id{f}}$ and
  $1+\id{f}=\sum_{i=1}^s\hat{e_i}(x)+\id{f}$. So,
  \begin{eqnarray*}
    \id{f_j(x)+\id{f}}&=&\id{(f_j(x)+\id{f})\sum_{i=1}^s\hat{e_i}(x)+\id{f}}\\
    &=&\id{(f_j(x)+\id{f})\sum_{i=1,i\neq
    j}^s\hat{e_i}(x)+\id{f}}\\
    &=&\id{\sum_{i=1,i\neq
    j}^s\hat{e_i}(x)+\id{f}}\\
    &=&\id{\sum_{i=1,i\neq
    j}^s\hat{f_i}(x)+\id{f}}.\\
  \end{eqnarray*}
  Using (\ref{prop.princ_code.one}) and Lemmas
  \ref{Local.Subambients.Lemma.R.Local} and
  \ref{Local.Subambients.Lemma.R.Simple.Socle}, the results follow.
\end{proof}
\Details{
\begin{proof}
 This is a rather detailed proof to make sure that Theorem \ref{prop.princ_code}
 is correct.

 (1) This has been proved above.

 (2) Let $\hat{e}_i(x)$\ be as above.
 Define $\sR_i = \frac{GR(p^a,m)[x]}{ \id{f_i(x)} }$.
 We showed that $\sR = \bigoplus_{i=1}^s \id{\hat{f}_i(x) + \id{f} }$\ and
 $\id{\hat{e}_i(x) + \id{f} } = \id{\hat{f}_i(x) + \id{f} } \cong \sR_i$.
 The map
 \be
  \phi_i:  \sR_i & \rightarrow & \id{\hat{e}_i }  \nn \\
   g(x) + \id{f_i} & \mapsto & g(x)\hat{e}_i(x) + \id{f}   \nn
 \ee
 is a ring isomorphism as $\hat{e}_i(x) + \id{f}$\ is idempotent.
 Therefore the map
 \be
  \phi:  \bigoplus_{i=1}^s\sR_i & \rightarrow & \sR  \nn \\
    \phi(g_1(x) + \id{f_i}, \dots , g_s(x) + \id{f_s}) & \mapsto & \sum_{i=1}^s g_i(x)\hat{e}_i(x) + \id{f} \nn
 \ee
 is also a ring isomorphism.
 The maximal ideals of $\sR$\ are $M_1, \dots, M_s$\ and
 \be
  \label{prop.princ_code.Mi}
   \phi^{-1}(M_i) = \sR_1 \oplus \cdots \oplus \sR_{i-1} \oplus m_i \oplus \sR_{i+1} \oplus \cdots \sR_{s}
 \ee
 where $m_i$\ is the maximal ideal of $\sR_i$.
 By Lemma \ref{Local.Subambients.Lemma.R.Local}, we know that $m_i = \id{p + \id{f_i}, h_i(x) + \id{f_i}}$.
  Using $\phi$ and (\ref{prop.princ_code.Mi}), we see that
  $M_i = \id{p\hat{e}_i(x) + \sum_{i \neq j}\hat{e}_j(x) + \id{f} , h_i(x)\hat{e}_i(x) + \sum_{i \neq j}\hat{e}_j(x) + \id{f}} $.
  Since $\hat{e}_i(x)$'s and $\hat{f}_i(x)$'s differ by units, this also implies $M_i= \id{p\hat{f_i}(x)+f_i(x)+\id{f},h_i\hat{f_i}(x)+f_i(x)+\id{f}}$.
  By (\ref{prop.princ_code.Mi}), we get that $p = p\sum_{i=1}^s \hat{e}_i(x), h_i(x) = h_i(x)\sum_{i=1}^s \hat{e}_i(x) \in M_i $.
  So $\id{p,h_i(x)} \subset \id{p\hat{e}_i(x) + \sum_{i \neq j}\hat{e}_j(x) , h_i(x)\hat{e}_i(x) + \sum_{i \neq j}\hat{e}_j(x) }$.
  Since $f_i(x)$\ and $\hat{e}_j(x)$\ are coprime modulo $f$ for every $i \neq j$, we have
  $\id{\hat{e}_j(x) + \id{f} } = \id{ f_i(x)\hat{e}_j(x) + \id{f} }$.
  So $\id{f_i(x) + \id{f} } = \id{f_i(x)\sum_{j=1}^s\hat{e}_j(x) + \id{f} } = \id{\sum_{i\neq j} \hat{e}_j(x) + \id{f} }$.
  Since $f_i(x) = h_i(x)^{t_i} + p \beta_i(x) \in \id{p + \id{f}, h_i(x) + \id{f} }$, we get
  $p\hat{e}_i(x) + \sum_{i \neq j} \hat{e}_j(x) + \id{f}, p\hat{e}_i(x) + \sum_{i \neq j} \hat{e}_j(x) + \id{f} \in \id{p + \id{f},h_i(x) + \id{f}} $.
  Hence
  $$ \id{p + \id{f} ,h_i(x) + \id{f}} =  \id{p\hat{e}_i(x) + \sum_{i \neq j}\hat{e}_j(x) + \id{f} , h_i(x)\hat{e}_i(x) + \sum_{i \neq j}\hat{e}_j(x) + \id{f} } = M_i.  $$

  (3) Since $f_i(x)$\ and $f_j(x)$\ are coprime for every $i \neq j$, the polynomials $h_i(x)$\ and $h_j(x)$\ are also coprime.
  So $ \bigcap_{i=1}^s \id{ h_i(x) + \id{f} } = \id{ \prod_i=1^s h_i(x)  + \id{f}}$\ and the claim  follows.

  (4) It is not hard to see that $h_j(x)^{t_j-1}$\ is a unit in $\sR_i$\ for all $i \neq j$.
  As a result of this, $\prod_{i \neq j}h_{j}(x)^{t_j-1}$\ is a unit in $\sR_j$.
  So
  \be
     \bigoplus_{i=1}^s\id{p^{a-1}h_i(x)^{t_i-1}\hat{f}_i(x)+\id{f}} & = & \bigoplus_{i=1}^s\id{p^{a-1}h_i(x)^{t_i-1}\hat{e}_i(x)+\id{f}}\nn\\
    & = & \bigoplus_{i=1}^s\id{p^{a-1}h_i(x)^{t_i-1} \prod_{i \neq j}h_{j}^{t_j-1}(x)\hat{e}_i(x)+\id{f}}\nn \\
     & = & \id{p^{a-1} \prod_{i = 1}^sh_{i}(x)^{t_i-1}\hat{e}_i(x)+\id{f}}\nn \\
     & = & \id{p^{a-1} \prod_{i = 1}^sh_{i}(x)^{t_i-1}\hat{f_i}(x)+\id{f}}.\nn
  \ee
\end{proof}
}

\begin{theorem}
  \label{theo.pir} The following are equivalent:
  \begin{enumerate}
    \item\label{theo.pir.one} $\Rf$ is not a principal ideal ring.
    \item\label{theo.pir.two} $a>1$ and there exists a factor from a primary co-prime
      factorization of $f(x)$, $g(x)$, where $g(x)=h(x)^t+p\beta(x)$ and
      $h(x)$ is basic irreducible, $t>1$ and $\beta(x)\in\id{p,h(x)}$.
    \item\label{theo.pir.three} $a>1$, $\bar{f}(x)$ is not square free and if $\bar{f'}(x)$
      is the square free part of $\bar{f}(x)$, and we write
      $f(x)=f'(x)\alpha(x)+p\gamma(x)$ then $\bar{\gamma}(x)=0$ or $\bar{\alpha}(x)$ and
      $\bar{\gamma}(x)$ are not co-prime.
  \end{enumerate}
\end{theorem}
\begin{proof}
  (\ref{theo.pir.one})$\iff$(\ref{theo.pir.two}) 
  By Theorem \ref{Theorem.McDonald.Factorization.and.Uniqueness}, 
  there exists a primary coprime decomposition
  of $g(x)$. Then the result follows from Theorems \ref{prop.princ_code} and
   \ref{Local.Subambients.Theorem.Chain.Condition}.

  (\ref{theo.pir.two})$\Rightarrow$(\ref{theo.pir.three}) Since $t>1$,
  $\bar{f}(x)$ is not square free. This also shows $h(x)|\bar{f'}(x)$ and
  $h(x)|\bar{\alpha}(x)$. Since $\beta(x)\in\id{p,h(x)}$, we have
  $\bar{\beta}(x)\in\id{h}$. This implies $h(x)|(g(x) \pmod{p^2})$. Since
  $g(x)|f(x)$, we see $h(x)|\bar{\gamma}(x)$. So, $\bar{\alpha}(x)$ and
  $\bar{\gamma}(x)$ are not co-prime.

  (\ref{theo.pir.three})$\Rightarrow$(\ref{theo.pir.two}) Since
  $\bar{f}(x)$ is not square free and $\bar{\alpha}(x)$ and $\bar{\gamma}(x)$
  are not co-prime there exists a basic irreducible polynomial $h(x)$
  such that $h(x)^t|\bar{f}(x)$ for some $t>1$ and $h(x)|\bar{\gamma}(x)$.
  So there exists a factor $g(x)$ of $f(x)$ such that $g(x)=h(x)^t+p\beta(x)$
  for some $\beta(x)$. Since $h(x)|\bar{\gamma}(x)$, we have that
  $h(x)|\bar{\beta}(x)$. Hence, $\beta(x)\in\id{p,h(x)}$.
\end{proof}

\begin{remark}
The equivalence in Theorem \ref{theo.pir} of (\ref{theo.pir.one})
and (\ref{theo.pir.three}) was presented in
\cite{Salagean_Disc_2006} with an alternative proof.
\end{remark}

\begin{lemma}
\label{lemma.kgen} Let $R$ be a ring with direct sum decomposition
$R=\oplus_{i=1}^{n} R_i$. Assume, for any positive integer $i$, that
$I_i\lhd R_i$ is at most $k$-generated. Then $I\lhd R$ is at most
$k$-generated.
\end{lemma}
\begin{proof}
Let $I\lhd R$. Then $I=\oplus_{i=1}^n I_i$ for $I_i\in R_i$. Then
$I_i$ is generated by some $f_{i1},\dots,f_{ik}\in R_i$. Let
$g_j=f_{1j}+\dots+f_{nj}$ for $1\leq j\leq k$. Then
$\id{f_{1j},\dots,f_{nj}}=\id{g_j}$ and hence
$I=\id{g_1,\dots,g_k}$.
\end{proof}

Now we generalize Proposition \ref{prop.ideal} to the case where $f(x)$\
is an arbitrary regular polynomial.

\begin{theorem}
\label{theo.main1} Let $C\lhd\Rf$. Then
\[
  C=\id{p^{j_0}g_0(x),\dots,p^{j_r}g_{r}(x)}
\]
where  ${0\leq r \leq a-1}$ and
\begin{enumerate}
\item $0\leq j_0<\dots<j_{r}\leq a-1$ \label{theo.main1.one}
\item $g_i(x)$ monic for $i=0,\dots,r$,
\item $\deg{f(x)}>\deg{g_0(x)}>\dots>\deg{g_r(x)}$,
\item $p^{j_{i+1}}g_i(x)\in\id{p^{j_{i+1}}g_{i+1}(x),\dots,p^{j_{r}}g_{r}(x)}$
\label{theo.main1.four}
\item $p^{j_0}f(x)\in\id{p^{j_0}g_0(x),\dots,p^{j_r}g_{r}(x)}$ in $GR(p^a,m)[x]$. \label{theo.main1.five}
\end{enumerate}
\end{theorem}
\begin{proof}
Follows from Proposition \ref{prop.ideal}, Theorem
\ref{prop.princ_code} and Lemma \ref{lemma.kgen}.
\end{proof}


%

The structure of the ambient space of cyclic codes over finite chain
rings was studied in \cite{Salagean_IEEE}, \cite{Salagean_Bulletin},
\cite{Salagean_AAECC} and \cite{Salagean_Disc_2006}. For any ideal
of the ambient space, the authors of those papers came up with a
special generating set called \textit{strong Groebner basis (SGB)}.
They showed that SGB can be used to determine the Hamming distance
of the corresponding code. It is easy to see that their results also
hold for the ideals of $\Rf$. 
So we have the following result.

\begin{theorem}
  \label{Groebner.Corollary.Hamming.Distance.Using.SGB}
  Let \Code\
  where  $C = \id{p^{j_0}g_{j_0}(x),\dots,p^{j_r}g_{j_r}(x) }$\ is as in Theorem
\ref{theo.main1}.
    Then $d_H(C) = d_H(\id{ p^{a-1}g_{j_r}(x) }) = d_H( \overline{ \id{g_{j_r}(x)} })$.
\end{theorem}
\begin{proof}
  For $v(x)\in C$, if $p^kv(x)\neq 0$ then $w_H(v(x)) \ge w_H(p^kv(x))$. Let
  $c(x)\in C$ such that $d_H(I)=w_H(c(x))$. Let $\ell$ be the largest
  integer such that $p^\ell c(x)\neq 0$. Hence, $p^\ell c(x)\in
  C\bigcap\id{p^{a-1}}=\id{p^{a-1}g_{j_r}}$.
  Also $w_H(c(x)) = w_H( p^kc(x) ) $\ by the minimality of $c(x)$.
  Hence,
  $d_H( \id{ p^{a-1}g_{j_r}(x)} )= w_H(pc(x)) = d_H(C)$.
  The equality
  $d_H(\id{ p^{a-1}g_{j_r}(x) }) = d_H(\overline{ \id{g_{j_r}(x)} } )$
  follows from Lemma \ref{Preliminaries.Lemma.Hamming.Weight.Projection}.
\end{proof}
%
%
%


\section{On the Hamming weight of $(x^n + \gamma)^N$}
\label{Preliminaries.For.Computations}
We develop some tools, that we use in
Section \ref{Section.Irreducible.Case.Finite.Fields} and
Section \ref{Section.Reducible.Case.Finite.Fields},
to compute the Hamming distance of some constacyclic codes over finite  fields.

We begin by partitioning the set $\{ 1,2,\dots,p^s-1 \}$\ into three subsets.
These subsets arise naturally from the technicalities of our computations as described in
Section \ref{Section.Irreducible.Case.Finite.Fields} and
Section \ref{Section.Reducible.Case.Finite.Fields}.
If $i$\ is an integer satisfying $1 \le i \le (p-1)p^{s-1}$, then there exists a uniquely determined
integer $\beta$\ such that $0 \le \beta \le p-2$\ and
\be\nn
    \beta p^{s-1}+1 \le i \le (\beta + 1)p^{s-1}.
\ee
Moreover since
\be
    p^s-p^{s-1} < p^s-p^{s-2}< \cdots < p^s-p^{s-s}=p^s-1,
\nn\ee
for an integer $i$ satisfying $(p-1)p^{s-1}+1 = p^s-p^{s-1}+1 \le i \le p^s-1$, there exists a uniquely determined
integer $k$ such that $1 \le k \le s-1$ and
\be \label{condition.on.k}
    p^s-p^{s-k}+1 \le i \le p^s-p^{s-k-1}.
\ee
Besides if $i$ is an integer as above and $k$ is the integer satisfying $1 \le k \le s-1$ and (\ref{condition.on.k}), then we have
\be
    && p^s-p^{s-k} < p^s-p^{s-k}+p^{s-k-1} < p^s-p^{s-k}+2p^{s-k-1} < \cdots \nn \\
    && < p^s-p^{s-k}+(p-1)p^{s-k-1}
\nn\ee
and $p^s-p^{s-k}+(p-1)p^{s-k-1}=p^s-p^{s-k-1}$. So for such integers $i$ and $k$, there exists a uniquely determined integer $\tau$ with $1 \le \tau \le p-1$ such that
\be
    p^s-p^{s-k} + (\tau -1) p^{s-k-1} + 1 \le i \le  p^s - p^{s-k} + \tau p^{s-k-1}.
\nn\ee
Thus
\be\label{Preliminaries.Partition}
\begin{array}{rl}
        &\{1,2,\dots,p^{s-1}\} \sqcup \displaystyle\bigsqcup_{\beta = 1}^{p-2}\{i: \quad \beta p^{s-1} + 1 \le i \le (\beta + 1)p^{s-1} \}\\
        &\sqcup \displaystyle\bigsqcup_{k=1}^{s-1}\displaystyle\bigsqcup_{\tau = 1}^{p - 1}\{i:\quad  p^s-p^{s-k} + (\tau -1) p^{s-k-1} + 1 \le i \le  p^s - p^{s-k} + \tau p^{s-k-1} \}
    \end{array}
\ee
gives us a partition of the set $\{1,2,\dots , p^{s}-1\}$.

Throughout this section $q$\ denotes a power of $p$.
Let $N$\ be a positive integer and $\myConstant \in \F_{q} \setminus \{ 0 \}$.
Our computations in Section \ref{Section.Irreducible.Case.Finite.Fields} and
Section \ref{Section.Reducible.Case.Finite.Fields} are
based on expressing the Hamming weight of an arbitrary nonzero codeword in terms of
$w_H ((x^\newn + \myConstant)^N)$. In \cite{CMSS2}, the Hamming weight of the polynomial
$(x^\newn + \myConstant)^N$\ is given as described below.
Let $e,\newn,N$\ and $0\le b_0,b_1,\dots,b_{e-1} \le p-1$\
be positive integers such that $N < p^e$\ and let $\myConstant \in \F_{q} \setminus \{ 0 \}$.
Let $N = b_{e-1}p^{e-1} + \cdots + b_1p + b_0$, $ 0 \le b_i < p$,
be the p-adic expansion of $N$. Then, by \cite[Lemma 1]{CMSS2}, we have
\be\label{Equality.weight.x+c.N}
    w_H( (x + \myConstant)^N ) = \prod_{d=0}^{e-1}(b_d + 1).
\ee
As suggested in \cite{CMSS2}, identifying $x$\ with $x^\newn$ in (\ref{Equality.weight.x+c.N}), we obtain
\be\label{Equality.weight.xn+c.N}
    w_H( (x^\newn + \myConstant)^N ) = \prod_{d=0}^{e-1}(b_d + 1).
\ee

The following two lemmas are consequences of (\ref{Equality.weight.xn+c.N}) and we will
use them in our computations frequently.

\begin{lemma}\label{Preliminaries.Lemma.weight.beta}
    Let $m,\newn, 1 \le \beta \le p-2$\ be positive integers and $\myConstant \in \F_{q} \setminus \{0\}$.
    If $m < p^s - \beta p^{s-1} -1$, then $w_H( (x^\newn + \myConstant)^{m + \beta p^{s-1} + 1} ) \ge \beta + 2$.
\end{lemma}
\begin{proof}
    Since
    \be\nn
        m < p^s - \beta p^{s-1} -1 = (p - \beta -1)p^{s-1} + (p-1)p^{s-2} + \cdots + (p-1)p + p-1,
    \ee
    either
    \be\nn
        m & = & Lp^{s-1} + (p-1)p^{s-2} + \cdots + (p-1)p + p-1\quad \mbox{or}\\
        m & = & a_{s-1}p^{s-1} + \cdots + a_1p + a_0\nn
    \ee
    holds, where $0 \le L \le p- \beta - 2$, $0 \le a_0,a_1,\dots,a_{s-2} \le p-1$\
    and $0 \le a_{s-1} \le p - \beta -1$\ are integers such that $a_{\ell} < p-1$\ for some $0 \le \ell < s-1$.
    According to the p-adic expansion of $m$,
    we consider the following two cases.

    First, we assume that $m = Lp^{s-1} + (p-1)p^{s-2} + \cdots + (p-1)p + p-1$.
    Then $m + \beta p^{s-1} + 1 = (L + \beta +1)p^{s-1}$. So using (\ref{Equality.weight.xn+c.N}), we get
    $w_H( (x^\newn + \myConstant)^{m + \beta p^{s-1} + 1} ) = L + \beta + 2 \ge \beta + 2 $.

    Second, we assume that $m = a_{s-1}p^{s-1} + \cdots + a_1p + a_0$. Then the p-adic expansion of
    $m + \beta p^{s-1} + 1$\ is of the form $ m + \beta p^{s-1} + 1 = b_{s-1}p^{s-1} + \cdots + b_1p + b_0$
    where $0 \le b_0,b_1,\dots,b_{s-2} \le p-1$\ and
    \be\label{Preliminaries.Lemma.weight.beta.bs1}
        b_{s-1} = a_{s-1} + \beta.
    \ee
    Let $k$\ be the least nonnegative integer with $a_k < p-1$. Then it follows that
    \be\label{Preliminaries.Lemma.weight.beta.bk}
        0 < b_k \le p-1.
    \ee
    So, using (\ref{Equality.weight.xn+c.N}), (\ref{Preliminaries.Lemma.weight.beta.bs1}) and
    (\ref{Preliminaries.Lemma.weight.beta.bk}), we get
    \be\nn
        w_H( (x^\newn + \myConstant)^{m + \beta p^{s-1} + 1} )\ge (\beta + a_{s-1} + 1)(b_k + 1)
                                \ge (\beta + 1)2
                                > \beta + 2.
    \ee
\end{proof}

\begin{lemma}\label{Preliminaries.Lemma.weight.tau.k}
    Let $m,\newn, 1 \le \tau \le p-1,1\le k \le s-1 $\ be positive integers and
    $\myConstant \in \F_{q} \setminus \{0\} $. If $m < p^{s-k} - (\tau - 1)p^{s-k-1} - 1$, then
    $w_H( (x^{2\newn} + \myConstant)^{ m + p^s -p^{s-k} + (\tau -1)p^{s-k-1} + 1} ) \ge (\tau + 1)p^k$.
\end{lemma}
\begin{proof}
    Since
    \be
        m & < & p^{s-k} - (\tau - 1)p^{s-k-1} -1 \nn\\
        & = & (p - \tau +1)p^{s - k - 1} - 1\nn\\
        & = & (p - \tau)p^{s-k-1} + (p-1)p^{s-k-2}+ \cdots + (p-1)p + p-1,\nn
    \ee
    either
    \be
        m & = & Lp^{s-k-1} + (p-1)p^{s-k-2} + \cdots + (p-1)p + p-1\quad \mbox{or}\nn\\
        m & = & a_{s-k-1}p^{s-k-1} + \cdots + a_1p + a_0\nn
    \ee
    holds, where $0 \le L \le p - \tau -1$, $0 \le a_0,a_1,\dots,a_{s-k-2}\le p-1$\ and
    $ 0 \le a_{s-k-1} \le p - \tau$\ are some integers such that $0 \le a_{\ell} < p-1$\
    for some $0 \le \ell < s-k-1$. According to the p-adic expansion of $m$,
    we consider the following two cases.

    First, we assume that $m = Lp^{s-k-1} + (p-1)p^{s-k-2} + \cdots + (p-1)p + p-1$.
    Then the p-adic expansion of $m + p^s - p^{s-k} + (\tau -1)p^{s-k-1} + 1$\ is of the form
    \be\nn
        m + p^s - p^{s-k} + (\tau -1)p^{s-k-1} + 1 = (p-1)p^{s-1} + \cdots + (p-1)p^{s-k} + (L + \tau)p^{s-k-1}.
    \ee
    So, using (\ref{Equality.weight.xn+c.N}), we get
    $ w_H( (x^\newn + \myConstant)^{m+ p^s - p^{s-k} + (\tau -1)p^{s-k-1}+1} ) \ge (\tau +1)p^k $.

    Second, we assume that $m= a_{s-k-1}p^{s-k-1} + \cdots + a_1p + a_0$.
    Then the p-adic expansion of $m + p^s - p^{s-k} + (\tau -1)p^{s-k-1} + 1$\ is of the form
    \be\nn
        m + p^s - p^{s-k} + (\tau -1)p^{s-k-1} + 1 & = &
            (p-1)p^{s-1} + \cdots + (p-1)p^{s-k}\\
            & & + b_{s-k-1}p^{s-k-1}+ \cdots + b_1p + b_0\nn
    \ee
    where $0 \le b_0, b_1,\dots ,b_{s-k-1}\le p-1 $\ are integers.
    It is easy to see that
    \be\label{Preliminaries.Lemma.weight.tau.k.bsk1}
        b_{s-k-1} = a_{s-k-1}  + \tau -1.
    \ee
    Let $\ell_{0}$\ be the least nonnegative  integer with $0 \le a_{\ell_{0}} < p-1$. Then
    \be\label{Preliminaries.Lemma.weight.tau.k.bell0}
        0 < b_{\ell_0} \le  p-1.
    \ee
    Using (\ref{Preliminaries.Lemma.weight.tau.k.bsk1}), (\ref{Preliminaries.Lemma.weight.tau.k.bell0})
    and (\ref{Equality.weight.xn+c.N}), we get
    \be
        w_H( (x^\newn + \myConstant)^{m + p^s - p^{s-k} (\tau - 1)p^{s-k-1} + 1 } )
             & \ge &  p^k(b_{s-k-1} + 1)(b_{\ell_0} +1) \nn \\
            & \ge & 2 \tau p^k\nn\\
            & \ge & (\tau + 1)p^k.\nn
    \ee
\end{proof}

In \cite{CMSS2}, the authors have shown that the polynomial $(x^\newn + \myConstant)^N$\
has the so-called \textit{``weight retaining property''} (see \cite[Theorem 1.1]{CMSS2}).
As a result of this, they gave a lower bound for the Hamming weight of the polynomial
$g(x)(x^\newn + \myConstant)^N$\ where $g(x)$\ is any element of $\F_{q}[x]$.
Let $\newn,N,\myConstant $\ and $g(x)$\ be as above. Then, by  \cite[Theorem 1.3 and Theorem 6.3]{CMSS2}, the Hamming weight of $g(x)(x^\newn + \myConstant)^N$\
satisfies
\be\label{Inequality.Lower.Bound.g(x).xn+c.N}
    w_H(g(x)(x^\newn + \myConstant )^N) \ge w_H( g(x)\mod x^\newn + \myConstant )\cdot w_H((x^\newn + \myConstant )^N).
\ee

Now we examine the Hamming weight of the polynomials
$(x^\newn + \myConstant_1)^{p^s}(x^\newn + \myConstant_2)^i$, over $\F_{q}[x]$, where $0 < i < p^s$.
Let $0 < i < p^s$\ be an integer and
$\myConstant_1, \myConstant_2 \in \F_{q} \setminus\{ 0 \}$. Let
\be\nn
    (x^\newn + \myConstant_2)^i = a_ix^{\newn i} + a_{i-1}x^{\newn (i-1)} + \cdots + a_0\myConstant_2^i
\ee
where $a_0,a_1,\dots ,a_i $\ are the binomial coefficients.
Note that
\be
    (x^\newn + \myConstant_1)^{p^s}(x^\newn + \myConstant_2)^i
        & = & (x^{\newn p^s} + \myConstant_1^{p^s})(a_ix^{\newn i} +  a_{i-1}x^{\newn (i-1)}\myConstant_2 + \cdots + a_0 \myConstant_{2}^{i}) \nn \\
        & = & a_ix^{\newn (i+p^s)} + a_{i-1}x^{\newn (i-1+p^s)}\myConstant_2 + \cdots + a_0x^{\newn p^s}\myConstant_2^i  \nn\\
        & & + a_i\myConstant_1^{p^s}x^{\newn i} + a_{i-1}\myConstant_1^{p^s}x^{\newn (i-1)} + \cdots + a_0 \myConstant_1^{p^s}\myConstant_2^i. \nn
\ee
Therefore
$w_H( (x^\newn + \myConstant_1)^{p^s}(x^\newn + \myConstant_2)^i ) = 2w_H( (x^\newn + \myConstant_2)^i )$.


\section{Certain constacyclic codes of length $\newn p^s$}
\label{Section.Irreducible.Case.Finite.Fields}

Let $\newn$\ and $s$\ be positive integers.
Let $\myConstant, \lambda\in \F_{p^m} \setminus\{ 0 \}$\ such that $\myConstant^{p^s} = - \lambda$.
All $\lambda$-cyclic codes, of length $\newn p^s$, over $\F_{p^m}$ correspond
to the ideals of the finite ring
\be\nn
    \sR = \frac{\F_{p^m}[x]}{\langle x^{\newn p^s} - \lambda \rangle}.
\ee
Suppose that $x^\newn + \myConstant $\ is irreducible over $\F_{p^m}$.
Then the monic divisors of $x^{\newn p^s} - \lambda = (x^\newn + \myConstant)^{p^s}$\
are exactly the elements of the set
$
\{ (x^\newn + \myConstant)^i: \quad 0 \le i \le p^s \}
$.
So if $x^\newn + \lambda$\ is irreducible over $\F_{p^m}$,
then the $\lambda$-cyclic codes, of length $\newn p^s$,
over $\F_{p^m}$, are of the form $\langle (x^\newn + \myConstant)^i\rangle$\
where $0  \le i \le p^s$.
In this section, we determine the Hamming distance of all $\lambda$-cyclic
codes of length $\newn p^s$\ over $\F_{p^m}$\ and $GR(p^a,m)$.
In Theorem \ref{Irreducible.Theorem.Main}, we determine the Hamming distance of
$\id{(x^\newn + \myConstant)^i}$.
As a particular case, we obtain the Hamming distance of negacyclic codes of length $2p^s$
over $\F_{p^m}$\ where $x^2 + 1$\ is irreducible over $\F_{p^m}[x]$.
Using  Theorem \ref{Irreducible.Theorem.Main} together with the results of
Section \ref{Local.Subambients} and Section \ref{sect.main},
we determine the Hamming distance of a cyclic code of length $p^s$\ over $GR(p^a,m)$.

Let $C = \langle (x^\newn + \myConstant)^i \rangle$\ where $0 \le i \le p^s$\ is an integer
and $x^\newn + \myConstant \in \F_{p^m}[x]$\ is irreducible.
Obviously if $i=0$, then $C = \sR$, i.e., $C$\ is the whole space $\F_{p^m}^{\newn p^s}$,
and if $i = p^s$, then $C = \{ 0 \}$. For the remaining values of $i$,
we consider the partition of the set $\{1,2,\dots,p^s-1\}$ given in (\ref{Preliminaries.Partition}).

If $0 < i \le p^{s-1}$, then $d_H(C)$\ is 2 as shown in Lemma \ref{Irreducible.Lemma.Weight.Exactly.Two}.

For $p^{s-1} < i < p^s$, we first find a lower bound on the Hamming weight of an arbitrary
nonzero codeword of $C$\ in Lemma \ref{Irreducible.Lemma.Weight.Greater.Than.beta.plus.two} and
Lemma \ref{Irreducible.Lemma.Weight.Greater.Than.tau.k}.
Next in Corollary \ref{Irreducible.Corollary.Weight.Exactly.beta.plus.two}
and Corollary \ref{Irreducible.Corollary.Weight.Greater.Than.tau.k}, we show that
there exist codewords in $C$, achieving these previously found lower bounds.
This gives us the Hamming distance of $C$.

We summarize our results on $\sR$ in Theorem \ref{Irreducible.Theorem.Main}.
We observe that Theorem \ref{Irreducible.Theorem.Main}
gives the Hamming distance of negacyclic codes, of length $2p^s$, over
$\F_{p^m}$\ where $p \equiv 3 \mod 4$\ and $m$\ is an odd number.
We close this section by describing how to determine the Hamming distance of
certain polycyclic codes, and in particular constacyclic codes, of length
$\newn p^s$\ over $GR(p^a,m)$.

\begin{lemma}\label{Irreducible.Lemma.Weight.Exactly.Two}
    Let $1 \le i \le p^{s-1}$\ be an integer and let $C = \langle (x^\newn + \myConstant)^i \rangle$.
    Then $d_H(C) = 2$.
\end{lemma}
\begin{proof}
    The claim follows from Lemma \ref{Preliminaries.Lemma.Hamming.Weight.Projection}
    and the fact that
    $$(x^\newn + \myConstant)^{p^{s-1}-i}  (x^\newn + \myConstant)^i = (x^\newn + \myConstant)^{p^{s-1}} =
        x^{\newn p^{s-1}} + \myConstant ^{p^{s-1}} \in C.$$
\end{proof}

Let $C = \langle (x^\newn + \myConstant)^i\rangle $ for some integer $0 < i < p^s$. For any $0 \neq c(x) \in C$,
there exists a $0 \neq f(x) \in \F_{q}[x]$\ such that $c(x) \equiv f(x)(x^\newn + \myConstant)^i \mod (x^\newn+ \myConstant)^{p^s}$.
Dividing $f(x)$\ by $(x^\newn + \myConstant)^{p^s-i}$, we get
$f(x) = q(x)(x^\newn + \myConstant)^{p^s-i} + r(x)$
where $q(x), r(x) \in \F_{q}[x]$\ and $0 \le \deg (r(x)) < \newn p^s - \newn i$\ or $r(x) = 0$ .
We observe that
\be\nn
    c(x) & \equiv & f(x)(x^\newn + \myConstant)^i\\
        & \equiv & (q(x)(x^\newn + \myConstant)^{p^s-i} + r(x))(x^\newn + \myConstant)^i \nn \\
        & \equiv & q(x)(x^\newn + \myConstant )^{p^s} + r(x)(x^\newn + \myConstant )^i \nn \\
        & \equiv & r(x)(x^\newn + \myConstant )^i \mod (x^\newn + \myConstant )^{p^s}.\nn
\ee
Consequently, for any $0\neq c(x) \in C$, there exists $0 \neq r(x) \in \F_{p^m}[x]$\ with
$\deg (r(x)) < \newn p^s-\newn i$\ such that $c(x) = r(x)(x^\newn + \myConstant)^i$,
where we consider this equality in $\F_{p^m}[x]$.
Therefore the Hamming weight of
$c \in C$\ is equal to the nonzero coefficients of $r(x)(x^\newn + \myConstant )^i \in \F_{q}[x]$,
i.e., $w_H(c) = w_H(r(x)(x^\newn + \myConstant)^i)$.

In the following lemma, we give a lower bound on $d_H(C)$\ when $p^{s-1} < i $.
\begin{lemma}\label{Irreducible.Lemma.Weight.Greater.Than.beta.plus.two}
    Let $1 \le \beta \le p-2$\ be an integer and let $C = \langle (x^\newn + \myConstant)^{\beta p^{s-1} + 1} \rangle $.
    Then $d_H(C) \ge \beta + 2 $.
\end{lemma}
\begin{proof}
    Let $0 \neq c(x) \in C$, then there exists $0 \neq f(x) \in \F_{q}[x]$\ such that
    \be\nn
        c(x) \equiv f(x)(x^\newn + \myConstant)^{\beta p ^{s-1} + 1} \mod (x^\newn + \myConstant)^{p^s}.
    \ee
    We may assume that
    $\deg (f(x)) < \newn p^s - \newn \beta p^{s-1} - \newn = (p - \beta)\newn p^{s-1} - \newn$.
    We choose $m$\ to be the largest nonnegative integer with $(x^\newn + \myConstant)^m | f(x)$.
    Clearly $\deg (f(x)) < (p - \beta)\newn p^{s-1} -\newn$\ implies $m < (p - \beta)p^{s-1} -1$.
    So, by Lemma \ref{Preliminaries.Lemma.weight.beta}, we get
    \be\label{Irreducible.Lemma.Weight.Greater.Than.beta.plus.two.Weight.xn.Myconstant}
        w_H( (x^\newn + \myConstant )^{m + \beta p^{s-1}+1} )  \ge \beta + 2.
    \ee
    For $f(x) = g(x)(x^\newn + \myConstant)^m$, we have $g(x)\mod x^\newn + \myConstant \neq 0$
    by our choice of $m$, so
    \be\label{Irreducible.Lemma.Weight.Greater.Than.beta.plus.two.Modulo.Weight.Positive}
        w_H( g(x) \mod (x^\newn + \myConstant) ) > 0.
    \ee
    Now using
    (\ref{Irreducible.Lemma.Weight.Greater.Than.beta.plus.two.Weight.xn.Myconstant}),
    (\ref{Irreducible.Lemma.Weight.Greater.Than.beta.plus.two.Modulo.Weight.Positive})
    and
    (\ref{Inequality.Lower.Bound.g(x).xn+c.N}),
    we obtain
    \be\nn
        w_H(c(x)) & = & w_H(g(x)(x^\newn + \myConstant)^{m + \beta p^{s-1}+1})\\
            & \ge & w_H( g(x)\mod (x^\newn + \myConstant ) ) w_H( (x^\newn + \myConstant)^m )\nn \\
            & \ge & \beta + 2 \nn .
    \ee
\end{proof}

Next we show that the lower bound given in Lemma \ref{Irreducible.Lemma.Weight.Greater.Than.beta.plus.two}
is achieved
when $p^{s-1} < i \le (p-1)p^{s-1}$\
and this gives us the exact value of $d_H(C)$.
\begin{corollary}\label{Irreducible.Corollary.Weight.Exactly.beta.plus.two}
    Let $1 \le \beta \le p-2$, $\beta p^{s-1} + 1 \le i \le (\beta + 1)p^{s-1} $\
    be integers and
    let $C = \langle (x^\newn + \myConstant)^i \rangle $.
    Then $d_H(C) = \beta + 2$.
\end{corollary}
\begin{proof}
    Lemma \ref{Irreducible.Lemma.Weight.Greater.Than.beta.plus.two} and
     $C \subset \langle (x^\newn + \myConstant)^{\beta p^{s-1} + 1} \rangle$\ imply
     $d_H(C) \ge \beta + 2$.
    We know, by (\ref{Equality.weight.xn+c.N}), that
    $w_H( (x^\newn + \myConstant)^{(\beta + 1) p^{s-1}} ) = \beta + 2$.
    Clearly $(x^\newn + \myConstant)^{(\beta + 1)p^{s-1}} \in C $\ as $(\beta + 1)p^{s-1} \ge i$.
    Thus
    $d_H(C) \le \beta + 2.$
    Hence
    $d_H(C) = \beta + 2$.

\end{proof}

Having covered the range $p^{s-1} < i \le (p-1)p^{s-1}$, now
we give a lower bound on $d_H(C)$ when $(p-1)p^{s-1} < i < p^s $\ in the following lemma.
\begin{lemma}\label{Irreducible.Lemma.Weight.Greater.Than.tau.k}
    Let $1 \le \tau \le p-1$, $1 \le k \le s-1$\ be
    integers and let $C = \langle (x^\newn + \myConstant)^{ p^s - p^{s-k} + (\tau -1)p^{s-k-1} + 1} \rangle $.
    Then $d_H(C) \ge (\tau + 1)p^k$.
\end{lemma}
\begin{proof}
    Let $0 \neq c(x) \in C$, then there is $0 \neq f(x) \in \F_{p^m}[x]$\ such that
    \be \nn
        c(x) \equiv f(x)(x^\newn + \myConstant)^{p^s - p^{s-k} + (\tau -1)p^{s-k-1} + 1} \mod (x^\newn + \myConstant)^{p^s}.
    \ee
    We may assume that
    \be \label{Irreducible.Lemma.Weight.Greater.Than.tau.k.deg.f}
        \deg (f(x)) < \newn p^{s-k} -\newn(\tau -1)p^{s-k-1} - \newn .
    \ee
    Let $m$\ be the largest nonnegative integer with $(x^\newn + \myConstant)^m | f(x)$.
    Then there exists $g(x) \in \F_{p^m}[x]$\ such that $f(x) = g(x)(x^\newn + \myConstant)^m$.
    By (\ref{Irreducible.Lemma.Weight.Greater.Than.tau.k.deg.f}), we have $m < p^{s-k} - (\tau -1)p^{s-k-1} -1$.
    So, by Lemma \ref{Preliminaries.Lemma.weight.tau.k}, we get
    \be\label{Irreducible.Lemma.Weight.Greater.Than.tau.k.Weight.x^n.myConstant}
        w_H( (x^\newn + \myConstant)^{m + p^s - p^{s-k} + (\tau -1)p^{s-k-1} + 1} )
        \ge p^k(\tau + 1).
    \ee
    The maximality of $m$\ implies $x^\newn + \myConstant \nmid g(x)$\ and therefore
    $g(x) \mod x^\newn + \myConstant \neq 0$. So we have
    \be\label{Irreducible.Lemma.Weight.Greater.Than.tau.k.Modulo.Positive}
        w_H( g(x) \mod x^\newn + \myConstant ) > 0.
    \ee
    Now using (\ref{Inequality.Lower.Bound.g(x).xn+c.N}),
    (\ref{Irreducible.Lemma.Weight.Greater.Than.tau.k.Weight.x^n.myConstant}) and
    (\ref{Irreducible.Lemma.Weight.Greater.Than.tau.k.Modulo.Positive}),
    we obtain
    \be
        w_H( c(x) ) & = & w_H( g(x)(x^\newn + \myConstant)^{m + p^s - p^{s-k} + (\tau -1)p^{s-k-1} + 1} ) \nn \\
        & \ge & w_H( g(x) \mod x^\newn + \myConstant )w_H( (x^\newn + \myConstant)^{p^s - p^{s-k} + (\tau -1)p^{s-k-1} + 1 + m} ) \nn \\
        & \ge & p^k (\tau + 1)\nn.
    \ee
    This completes the proof.
\end{proof}

For $(p-1)p^{s-1} < i < p^s$, we determine $d_H(C)$\ in
Corollary \ref{Irreducible.Corollary.Weight.Greater.Than.tau.k} where we show the existence
of a codeword that achieves the lower bound given in
Lemma \ref{Irreducible.Lemma.Weight.Greater.Than.tau.k}.
\begin{corollary}\label{Irreducible.Corollary.Weight.Greater.Than.tau.k}
    Let $1 \le \tau \le p-1$, $1 \le k \le s-1$\ and $i$\ be integers such that
    \be
        p^s-p^{s-k}+(\tau -1)p^{s-k-1}+1 \le i \le p^s-p^{s-k}+\tau p^{s-k-1}.\nn
    \ee
    Let $C = \langle (x^\newn + \myConstant)^i \rangle$. Then $d_H(C) = (\tau + 1)p^k$.
\end{corollary}
\begin{proof}
    Lemma \ref{Irreducible.Lemma.Weight.Greater.Than.tau.k} and
    $C \subset \langle (x^\newn + \myConstant)^{p^s-p^{s-k}+(\tau - 1) p^{s-k-1} + 1} \rangle $\ implies
    $d_H(C) \ge (\tau + 1)p^k$.
    We know, by (\ref{Equality.weight.xn+c.N}), that
    $w_H( (x^\newn + \myConstant)^{p^s-p^{s-k}+ \tau p^{s-k-1}} ) = (\tau + 1)p^k$.
    Clearly $(x^\newn + \myConstant)^{p^s-p^{s-k}+ \tau p^{s-k-1}} \in C$\ as
    $p^s-p^{s-k} + \tau p^{s-k-1} \ge i$.
    So $d_H(C) \le (\tau + 1)p^k$.
    Thus we have shown
    $d_H(C) = (\tau+1)p^k$.
\end{proof}

We summarize our results in the following theorem.

\begin{theorem}\label{Irreducible.Theorem.Main}
    Let $p$\ be a prime number, $\F_{p^m}$\ a finite field of characteristic $p$,
    $\myConstant \in \F_{q}\setminus \{ 0 \}$\ and $\newn$\ be a positive integer.
    Suppose that $x^\newn + \myConstant \in \F_{q}[x]$\ is irreducible. Then the $\lambda$-cyclic codes
    over $\F_{q}$, of length $\newn p^s$, are of the form $C[i] = \langle (x^\newn + \myConstant)^i \rangle$,
    where $0 \le i \le p^s$\ and $\lambda = - \myConstant ^{p^{s}}  $.
    If $i = 0$, then $C$\ is the whole space $\F_{p^m}^{\newn p^s}$\ and if
    $i = p^s$, then $C$\ is the zero space $\{ \mathbf{0} \}$.
    For the remaining values of $i$, if $p = 2$, then
    $$ 
    \begin{array}{l}
      \dd d_H(C[i])
        = \left\lbrace
      \begin{array}{ll}
        1, &  \mbox{if}\ \ i=0,\\
        2, &  \mbox{if}\ \ 1 \le i \le 2^{s-1},\\
        2^{k+1}, & \mbox{if}\ \ 2^s-2^{s-k}+1 \le i \le 2^s-2^{s-k}+\tau 2^{s-k-1}\\
                & \mbox{where}\ \ 1 \le k \le s-1,
      \end{array} \right.
    \end{array}
    $$
    if $p$\ is odd, then
     $$ 
    \begin{array}{l}
            \dd d_H(C[i])
             = \left\lbrace
            \begin{array}{ll}
                2, &  \mbox{if}\ \ 1 \le i \le p^{s-1},\\
                \beta+2, & \mbox{if}\ \ \beta p ^{s-1} +1 \le i \le (\beta+1)p^{s-1}\ \mbox{where}\ \ 1 \le \beta \le p-2, \\
                (\tau + 1)p^k, & \mbox{if}\ \ p^s-p^{s-k}+(\tau -1)p^{s-k-1}+1 \le i \le p^s-p^{s-k}+\tau p^{s-k-1}\\
                    & \mbox{where}\ \ 1 \leq \tau \le p-1\ \ \mbox{and}\ \ 1 \le k \le s-1.
            \end{array} \right.
    \end{array}
    $$
\end{theorem}

\begin{remark}
    If we replace $\newn$\ with $1$ and $\myConstant$\ with $-1$\ in Theorem \ref{Irreducible.Theorem.Main},
    then we obtain the main results of \cite{D2008} and \cite{OZOZ_1}.
    Namely, we obtain \cite[Theorem 4.11]{D2008} and \cite[Theorem 3.4]{OZOZ_1}.
\end{remark}

Theorem \ref{Irreducible.Theorem.Main} is still useful
when the polynomial $x^\newn + \myConstant $\ is reducible over the alphabet $\F_{p^m}$.
\begin{remark}\label{Remark.nps.Results.Hold.Even.Reducible}
  Note that  $\id{(x^\newn + \myConstant )^i}$, $0 \le i \le p^s$\ are ideals of $\sR$\ independent
  of the fact that $x^\newn + \myConstant $\ is irreducible.
  So our results from
  Lemma \ref{Irreducible.Lemma.Weight.Exactly.Two} to
  Corollary \ref{Irreducible.Corollary.Weight.Greater.Than.tau.k} hold even when
  the polynomial $x^\newn + \myConstant $\ is reducible over $\F_{p^m}$.
  But then, the cases considered above do not cover all the $\lambda$-cyclic codes of length $p^s$.
  In other words, if $x^\newn + \myConstant $\ is reducible, then there are $\lambda$-cyclic codes
  other than $\id{(x^\newn + \myConstant )^i}$, $0 \le i \le p^s$ and their Hamming distance is not determined here.
\end{remark}

Now we will apply Theorem \ref{Irreducible.Theorem.Main} to a particular case.
Namely, we will consider the negacyclic codes
over $\F_{p^m}$ of length $2p^s$\ where $p$\ is an odd prime. In order to apply Theorem \ref{Irreducible.Theorem.Main},
the polynomial $x^2  + 1$\ must be irreducible over $\F_{p^m}$. A complete irreducibility criterion
for $x^2+1$\ is given in the following lemma.

\begin{lemma}\label{Irreducible.Lemma.Irreducibility.Criterion}
    Let $p$\ be an odd prime and $m$\ be a positive integer.
    The polynomial $x^2 + 1 \in \F_{p^m}[x]$\ is irreducible
    if and only if $p = 4k + 3$\ for some $k \in \N $\ and
    $m$\ is odd.
\end{lemma}
\begin{proof}
 Follows from the order of the multiplicative group of $\F_{p^m}$.
\end{proof}

Let $C$\ be a negacyclic code of length $2p^s$\ over $\F_{p^m}$.
If $x^2 + 1$ is irreducible over $\F_{p^m}$,
then the Hamming distance of $C$\ is given in the following theorem.

\begin{theorem}\label{Irreducible.Theorem.Particular}
    Let $p = 4k + 3$\ be a prime for some $k \in \N$ and let $m\in \N$\ be an odd number.
    Then the negacyclic codes
    over $\F_{p^m}$, of length $2p^s$, are of the form $C[i] = \langle (x^2 + 1)^i \rangle$,
    where $0 \le i \le p^s$, and
     $$ 
    \begin{array}{l}
            \dd d_H(C[i])
             = \left\lbrace
            \begin{array}{ll}
                2, &  \mbox{if}\ \ 1 \le i \le p^{s-1},\\
                \beta+2, & \mbox{if}\ \ \beta p ^{s-1} +1 \le i \le (\beta+1)p^{s-1}\ \mbox{where}\ \ 1 \le \beta \le p-2, \\
                (\tau + 1)p^k, & \mbox{if}\ \ p^s-p^{s-k}+(\tau -1)p^{s-k-1}+1 \le i \le p^s-p^{s-k}+\tau p^{s-k-1}\\
                    & \mbox{where}\ \ 1 \leq \tau \le p-1\ \ \mbox{and}\ \ 1 \le k \le s-1.
            \end{array} \right.
    \end{array}
    $$
\end{theorem}

For the other values of $p$\ and $m$, $x^2 + 1$\
is reducible over $\F_{p^m}$\ and in this case, we determine the minimum
Hamming distance of $C$\ in Section \ref{Section.Reducible.Case.Finite.Fields}.

Now we describe how to determine the Hamming distance of certain
polycyclic codes of length $\newn p^s$\ over $GR(p^a,m)$ and, in particular, this gives us the Hamming distance of
certain constacyclic codes of length $\newn p^s$.
Let $\gamma_0, \lambda_0 \in GR(p^a,m)$\ be units such that
$\overline{\gamma}_0 = \gamma$, $\overline{\lambda}_0 = \lambda$\ and $\gamma_0^{p^s} = -\lambda_0$.
According to our assumption in the beginning of this section, we have that $x^\newn + \overline{\gamma}_0$\
is irreducible.

Let $f(x) = (x^\newn + \myConstant_0 )^{p^s} + p\beta(x) \in GR(p^a,m)[x]$\ with $\deg( \beta(x) ) < \newn p^s$.
Note that $f(x)$\ in this form is a primary regular polynomial so the techniques of
Section \ref{Local.Subambients} can be applied.

Let
$\sR_0 = \frac{GR(p^a,m)[x]}{\id{f(x)}}$.
Let $C = \id{p^{j_0}g_0(x), \dots, p^{j_r}g_r(x)} \vartriangleleft \sR_0$\
 where the generators are as in Theorem \ref{theo.main1}.
As was done in (\ref{Local.Subambients.Equation.gx.in.terms.of.hx}),
we can express $g_r(x)$\ in the canonical form
\be\nn
  g_r(x) = p^0(x^\newn + \gamma_0)^{e_0}\alpha_0(x) + \cdots + p^{a-1}(x^\newn + \gamma_0)^{e_{a-1}}\alpha_{a-1}(x)
\ee
where each $\alpha_i(x)$ is either a unit or 0.
For $0 \neq g_r(x)$, we have $\alpha_0(x) \neq 0$\ since $p \nmid g_r(x)$.
Therefore $\alpha_0(x)$\ is a unit.
So, by Theorem \ref{Groebner.Corollary.Hamming.Distance.Using.SGB}, we deduce that
$d_H(C) = d_H( \overline{ \id{g_r(x)} }  ) = d_H(\overline{ \id{( x^\newn + \gamma )^{e_0}}}  )$.
Now $d_H( \overline{ \id{( x^\newn + \gamma )^{e_0}}} )$\ can be determined using Theorem \ref{Irreducible.Theorem.Main}.
  \begin{remark}
    Let $\myConstant, \myConstant_0, \lambda, \lambda_0$\ be as above.
    The $\lambda_0$-cyclic codes of length $\newn p^s$\ over $GR(p^a,m)$\ are
    the ideals of the ring $\frac{GR(p^a,m)[x]}{ \id{ x^{\newn p^s} - \lambda_0 } }$.
    Since $x^{\newn p^s} - \lambda_0 =  (x^\newn + \myConstant_0 )^{p^s} + p\beta^{'}(x)$,
    for some $\beta(x) \in GR(p^a,m)[x]$\ with $\deg (\beta^{'}(x)) <  \newn p^s$,
    we can determine the Hamming  distance of the
    $\lambda_0$-cyclic codes of length $\newn p^s$\ over $GR(p^a,m)$ as described above.
  \end{remark}


\section{Certain constacyclic codes of length $2\newn p^s$}
\label{Section.Reducible.Case.Finite.Fields}
We assume that $p$\ is an odd prime number, $\newn$\ and $s$\ are positive integers,
$\F_{p^m}$\ is a finite field of characteristic $p$\ and
$\lambda, \xi, \psi \in \F_{p^m} \setminus \{ 0 \}$ throughout this section.

Suppose that $\psi ^{p^s} = \lambda$\ and $x^{2\newn } - \psi$\ factors into two irreducible
polynomials over $\F_{p^m}$\ as
\be\label{Reducible.Equation.Factorization}
    x^{2\newn } - \psi = (x^\newn - \xi)(x^\newn + \xi).
\ee

In this section, we compute the Hamming distance of
$\lambda$-cyclic codes, of length $2\newn p^s$, over $\F_{p^m}$\ where
(\ref{Reducible.Equation.Factorization}) is satisfied.
Next, we determine the Hamming distance of certain  polycyclic codes,
and in particular certain constacyclic codes, of length
$\newn p^s$\ over $GR(p^a,m)$.
We know that $\lambda$-cyclic codes of length $2\newn p^s$\ over $\F_{p^m}$ correspond to the ideals
of the finite ring

\be
    \sR = \frac{\F_{p^m}[x]}{ \langle x^{2\newn p^s} - \lambda \rangle }.\nn
\ee
Note that, by Proposition \ref{prop.princ_code}, we have
$  \sR =  \id{x^{\newn p^s} + \xi^{p^s}} \oplus \id{x^{\newn p^s} - \xi^{p^s}}   $
and $\id{x^{\newn p^s} + \xi^{p^s}} \cong \frac{\F_{p^m}[x]}{ \id{x^{\newn p^s} - \xi^{p^s} }}$,
$\id{x^{\newn p^s} - \xi^{p^s}} \cong \frac{\F_{p^m}[x]}{ \id{x^{\newn p^s} + \xi^{p^s} }}$.
Moreover, by Proposition \ref{prop.princ_code},
the maximal ideals of $\sR$\ are $\id{x^\newn - \xi}$\ and $\id{x^\newn + \xi}$.
Since the monic polynomials dividing $x^{2\newn p^s} - \lambda$\ are exactly
the elements of the set $\{ (x^\newn - \xi)^i(x^\newn + \xi)^j: \quad 0 \le i,j \le p^s \}$,
the $\lambda$-cyclic codes, of length $2\newn p^s$, over $\F_{p^m}$\ are of the form
$\langle (x^\newn  - \xi)^i(x^\newn  + \xi)^j \rangle$, where $0 \le i,j \le p^s$\ are integers.

Let $C = \langle (x^\newn  - \xi)^i(x^\newn  + \xi)^j \rangle$.
If $(i,j) = (0,0)$, then $C = \sR$. If $(i,j) = (p^s, p^s)$, then
$C = \{ 0 \}$. For the remaining values of $(i,j)$, we consider the partition of the set
$\{ 1,2,\dots,p^s-1 \}$\ given in (\ref{Preliminaries.Partition}).

In order to simplify and improve the presentation of our results,
from Lemma \ref{Reducible.Lemma.Weight.3} till Corollary \ref{Reducible.Corollary.ips.j.tau.K},
we consider only the cases where $i \ge j$ explicitly.
We do so because the cases where $j > i$\ can be treated similarly as the
corresponding case of $i> j$.

Now we give an overview of the results in this section.
If $i = 0$, or $j= 0$, or $0 \le i,j \le p^{s-1}$, then the Hamming distance of
$C$\ can easily found to be 2 as shown in Lemma \ref{Reducible.Lemma.Weight.i.is.0.or.j.is.0}
and Lemma \ref{Reducible.Lemma.Weight.i.and.j.small}.

If $0 < j \le p^{s-1}$\ and $p^{s-1} + 1 \le i \le p^s$, then $d_H(C)$\ is computed in
Lemma \ref{Reducible.Lemma.Weight.3}, Corollary \ref{Reducible.Corollary.Weight.3},
Lemma \ref{Reducible.Lemma.Weight.4}\ and Corollary \ref{Reducible.Corollary.Weight.4}.

If $p^{s-1} + 1 \le j \le i \le (p-1)p^{s-1}$, then $d_H(C)$\ is computed in
Lemma \ref{Reducible.Lemma.beta.beta.prime} and Corollary \ref{Reducible.Corollary.beta.beta.prime}.

If $p^{s-1} + 1 \le j \le (p-1)p^{s-1} < i \le p^s -1$, then $d_H(C)$\ is computed
in Lemma \ref{Reducible.Lemma.beta.tauK} and Corollary \ref{Reducible.Corollary.beta.tauK}.

If $(p-1)p^{s-1} + 1 \le j \le i \le p^s -1$, then $d_H(C)$\ is computed in
Lemma \ref{Reducible.Lemma.tauK.same.k}, Corollary \ref{Reducible.Corollary.tauK.same.k},
Lemma \ref{Reducible.Lemma.tauK.different.k} and Corollary \ref{Reducible.Corollary.tauK.different.k}.

Finally if $i=p^s$\ and $0 < j < p^s-1$, then $d_H(C)$\ is computed from
Lemma \ref{Reducible.Lemma.ips.j.small}
till Corollary \ref{Reducible.Corollary.ips.j.tau.K}.

At the end of this section, we summarize our results in Theorem \ref{Reducible.Theorem.Main}.

We begin our computations with the case where $i = 0$\ or $j=0$.

\begin{lemma}\label{Reducible.Lemma.Weight.i.is.0.or.j.is.0}
    Let $0 < i,j \le p^s$\ be integers, let $C = \langle (x^\newn  - \xi)^i \rangle $\ and
    $D = \langle (x^\newn  + \xi)^j \rangle $. Then $d_H(C) = d_H(D) = 2$.
\end{lemma}
\begin{proof}
    Since
    \be
        (x^\newn  - \xi)^{p^s - i}(x^\newn  - \xi)^{i} & = & x^{\newn p^s} - \xi^{p^s} \in C \quad \mbox{and} \nn\\
        (x^\newn  + \xi)^{p^s - j}(x^\newn  + \xi)^{j} & = & x^{\newn p^s} + \xi^{p^s} \in D, \nn
    \ee
    we have $d_H(C), d_H(D) \le 2$. On the other hand, $d_H(C), d_H(D) \ge 2$\ by Lemma
    \ref{Preliminaries.Lemma.Hamming.Weight.Projection}.
    Hence $d_H(C) = d_H(D) = 2$.
\end{proof}

\begin{lemma}\label{Reducible.Lemma.Weight.i.and.j.small}
    Let $C = \langle (x^\newn  - \xi)^i(x^\newn  + \xi)^j \rangle$, for some integers $0 \le i,j \le p^{s-1}$\
    with $(i,j) \neq (0,0)$. Then $d_H(C) = 2$.
\end{lemma}
\begin{proof}
    By Lemma \ref{Preliminaries.Lemma.Hamming.Weight.Projection}, we have $d_H(C) \ge 2$
    and
    \be\nn
         (x^\newn  - \xi)^{i} (x^\newn  + \xi)^{j}(x^\newn  - \xi)^{p^{s-1} - i} (x^\newn  + \xi)^{p^{s-1} - j}
         = x^{2\newn p^{s-1}} - \xi^{2p^{s-1}} \in C
    \ee
    implies that $d_H(C) \le 2$. Hence $d_H(C) = 2$.
\end{proof}

Let $C = \langle(x^\newn  - \xi)^i(x^\newn  + \xi)^j \rangle$\ for some integers $0 \le i,j \le p^s$\ with $(0,0)\neq (i,j) \neq (p^s,p^s) $.
Let $0 \neq c(x) \in C$, then there exists $0 \neq f(x) \in \F_{p^m}[x]$\ such that
$c(x) \equiv f(x)(x^\newn  - \xi)^i(x^\newn  + \xi)^j \mod x^{2\newn p^s} - \lambda $.
Dividing $f(x)$\ by $(x^\newn  - \xi)^{p^{s}-i}(x^\newn  + \xi)^{p^s -j}$, we get
\be\nn
    f(x) = q(x)(x^\newn  - \xi)^{p^s - i}(x^\newn  + \xi)^{p^s -j} + r(x)
\ee
where $q(x)$, $r(x) \in \F_{q}[x]$\ and, either $r(x) = 0$\ or $\deg (r(x)) < 2\newn p^s - \newn i - \newn j$.
Since
\be
    c(x) & \equiv & f(x)(x^\newn  - \xi)^i(x^\newn  + \xi)^j \nn \\
    & \equiv & (q(x)(x^\newn  - \xi)^{p^s - i}(x^\newn + \xi)^{p^s - j} + r(x) )(x^\newn  - \xi)^i(x^\newn  + \xi)^j \nn \\
    & \equiv & q(x)(x^\newn  - \xi)^{p^s}(x^\newn + \xi)^{p^s} + r(x)(x^\newn  - \xi)^i(x^\newn  + \xi)^j \nn \\
    & \equiv & r(x)(x^\newn  - \xi)^i(x^\newn  + \xi)^j \mod x^{2\newn p^s} - \lambda , \nn
\ee
we may assume, without loss of generality, that $\deg (f(x)) < 2\newn p^s - \newn i - \newn j$.
Moreover $w_H( r(x)(x^\newn  - \xi)^i(x^\newn  + \xi)^j ) = w_H(c) $\ as
$\deg (r(x)(x^\newn  - \xi)^i(x^\newn  + \xi)^j) < 2\newn p^s$.

Let $i_0$\ and $j_0$\ be the largest integers with $(x^\newn  - \xi)^{i_0} | f(x)$\ and
$(x^\newn  + \xi)^{j_0} | f(x)$. Then there exists $g(x) \in  \F_{p^m}[x]$\ such that
$f(x) = (x^\newn  - \xi)^{i_0}(x^\newn  + \xi)^{j_0}g(x)$\ and $(x^\newn  - \xi) \nmid g(x)$, $(x^\newn  + \xi) \nmid g(x)$.
Clearly $\deg (f(x)) < 2\newn p^s - \newn i - \newn j $\ implies $i_0 + j_0 < 2p^s - i -j$.
Therefore $i_0 < p^s - i$\ or $j_0 < p^s - j$\ must hold.

So if $i_0 \ge p^s - i $, then $j_0 < p^s - j$. For such cases,
the following lemma will be used in our computations.

\begin{lemma}\label{Reducible.Lemma.Weight.Bound.When.i0.large}
    Let $i,j,i_0,j_0$\ be nonnegative integers such that $i \ge j$, $i_0 \ge p^s - i$\ and
    $j_0 < p^s -j$. Let $c(x) = (x^\newn  - \xi)^{i_0 + i} (x^\newn  + \xi)^{j_0 + j} g(x)$\ with
    $x^\newn  - \xi \nmid g(x)$\ and $x^\newn  + \xi \nmid g(x)$.
    Then $w_H( c(x) ) \ge 2w_H( (x^{2\newn } - \xi ^2)^{j_0 + j} )$.
\end{lemma}
\begin{proof}
    Since $i_0 \ge p^s - i$\ and $-j_0 \ge -p^s + j + 1$, we have $i_0 - j_0 \ge j - i + 1$\ or equivalently $i_0 - j_0 + i - j \ge 1$.
    So
    $c(x) = (x^{2\newn } - \xi^2)^{j_0 + j}(x^\newn  - \xi)^{i_0 - j_0 + i - j}g(x)$.
    Dividing $(x^\newn  - \xi)^{i_0 - j_0 + i - j}g(x)$\ by $x^{2\newn }- \xi^2$, we get
    \be\label{Reducible.Lemma.Weight.Bound.When.i0.large.Division}
        (x^\newn  - \xi)^{i_0 - j_0 + i - j}g(x) = (x^{2\newn } - \xi^2)q(x) + r(x)
    \ee
    for some $q(x), r(x) \in \F_{q}[x]$\ with $r(x) = 0$\ or $\deg (r(x)) < 2\newn $.
    Let $\theta_1$\ and $\theta_2$\ be any roots of $x^\newn  - \xi$\ and $x^\newn  + \xi$,
    respectively, in some extension of $\F_{p^m}$.
    Obviously $\theta_1$\ and $\theta_2$\ are roots of $(x^{2\newn } - \xi^2)q(x)$.
    First we observe that $r(\theta _1) = 0$\ as $\theta_1$\ is a root of LHS of
    (\ref{Reducible.Lemma.Weight.Bound.When.i0.large.Division}).
    Second we observe that $r(\theta_2) \neq 0$\ as $\theta_2$\ is not a root of
    LHS of (\ref{Reducible.Lemma.Weight.Bound.When.i0.large.Division}).
    So it follows that $r(x)$\ is a nonzero and nonconstant polynomial implying
    $w_H(r(x)) \ge 2$. Therefore
    \be\label{Reducible.Lemma.Weight.Bound.When.i0.large.Bound}
        w_H( (x^\newn  - \xi)^{i_0 - j_0 + i - j}g(x)\mod x^{2\newn } - \xi^2 )
        = w_H( r(x) ) \ge 2.
    \ee
    Using (\ref{Inequality.Lower.Bound.g(x).xn+c.N}) and
    (\ref{Reducible.Lemma.Weight.Bound.When.i0.large.Bound}), we obtain
    \be\nn
        w_H( c(x) ) & = & w_H( (x^{2\newn } - \xi^2)^{j_0 + j}(x^\newn  - \xi)^{i_0 - j_0 + i - j}g(x) )\\
            & \ge & w_H( (x^{2\newn } - \xi^2)^{j_0 + j} )
                w_H( (x^\newn  - \xi)^{i_0 - j_0 + i - j}g(x)\mod x^{2\newn } - \xi^2 )\nn \\
            & \ge & 2w_H( (x^{2\newn } - \xi^2)^{j_0 + j} ).\nn
    \ee
\end{proof}

Now we have the machinery to obtain the Hamming distance of $C$\
for the ranges $p^{s-1} < i \le p^s$\ and $0 < j\le p^s$.

In what follows, for a particular range of $i$\ and $j$, we first give a lower bound on $d_H(C)$\
in the related lemma. Then in the next corollary, we determine $d_H(C)$\ by showing the
existence of a codeword that achieves the previously found lower bound.

We compute $d_H(C)$\ when $0 < j \le p^{s-1} < i \le 2p^{s-1}$\ in the following lemma and corollary.

\begin{lemma}\label{Reducible.Lemma.Weight.3}
    Let $C = \langle (x^\newn  - \xi)^{p^{s-1} + 1}(x^\newn  + \xi) \rangle$.
    Then $d_H(C) \ge 3$.
\end{lemma}
\begin{proof}
    Pick $0 \neq c(x) \in C$\ where $c(x) \equiv f(x)(x^\newn  - \xi)^{p^{s-1} + 1}(x^\newn  + \xi) \mod x^{2\newn p^s} - \lambda$\
    for some $0 \neq f(x) \in \F_{p^m}[x]$\ with $\deg (f(x)) < 2\newn p^s - \newn p^{s-1} - 2\newn $.
    Let $i_0$\ and $j_0$\ be the largest integers with $(x^\newn  - \xi)^{i_0} | f(x)$\ and
    $ (x^\newn  + \xi)^{j_0} | f(x)$. Then $f(x)$\ is of the form
    $f(x) = (x^\newn  - \xi)^{i_0}(x^\newn  + \xi)^{j_0}g(x)$\ for some $g(x) \in \F_{p^m}[x]$\ with
    $ x^\newn  - \xi \nmid g(x)$\ and $ x^\newn  + \xi \nmid g(x)$.
    Note that $i_0 < p^s - p^{s-1} -1$\ or $j_0 < p^s - 1$\ holds.

    If $i_0 < p^s -p^{s-1} -1$, then, by Lemma \ref{Preliminaries.Lemma.weight.beta},
    \be\label{Reducible.Lemma.Weight.3.Case.1.xn.xi.gt3}
        w_H( (x^\newn  - \xi)^{i_0 + p^{s-1} + 1} ) \ge 3.
    \ee
    Moreover the inequality
    \be\label{Reducible.Lemma.Weight.3.Case1.mod.Ineq}
        w_H( g(x)(x^\newn  + \xi)^{j_0 + 1}\mod x^\newn  - \xi ) > 0
    \ee
    holds since $x^\newn  - \xi \nmid g(x)$.
    Now using (\ref{Inequality.Lower.Bound.g(x).xn+c.N}), (\ref{Reducible.Lemma.Weight.3.Case.1.xn.xi.gt3}) and
    (\ref{Reducible.Lemma.Weight.3.Case1.mod.Ineq}), we obtain

    \be\label{Reducible.Lemma.Weight.3.Case.1}
        \begin{array}{rcl}
            w_H(c(x)) & = & w_H( f(x)(x^\newn  - \xi)^{p^{s-1}+1}(x^\newn  + \xi) )\\
                & = & w_H( (x^\newn  - \xi)^{i_0 + p^{s-1}  + 1} (x^\newn  + \xi)^{j_0 + 1}g(x))\\
                & \ge & w_H( (x^\newn  - \xi)^{i_0 + p^{s-1} + 1} )
                    w_H( (x^\newn  + \xi)^{j_0 + 1}g(x)\mod x^\newn  - \xi )\\
                & \ge& 3.
        \end{array}
    \ee

    If $i_0 \ge p^s - p^{s-1} -1$, then $j_0 < p^s -1$.
    Clearly $w_H( (x^{2\newn } - \xi ^2)^{j_0 + j} ) \ge 2$.
    So, by Lemma \ref{Reducible.Lemma.Weight.Bound.When.i0.large}, we have
    \be\label{Reducible.Lemma.Weight.3.Case.2}
        w_H( c(x) ) \ge 2 w_H( (x^{2\newn } - \xi ^2)^{j_0 + j} ) \ge 4.
    \ee

    Now combining (\ref{Reducible.Lemma.Weight.3.Case.1}) and
    (\ref{Reducible.Lemma.Weight.3.Case.2}), we obtain
    $w_H(c(x)) \ge 3$, and hence $d_H(C) \ge 3$.
\end{proof}
\begin{corollary}\label{Reducible.Corollary.Weight.3}
    Let $i,j$\ be integers with $2p^{s-1} \ge i > p^{s-1} \ge j > 0$\ and
    let $C = \langle (x^\newn  - \xi)^{i}(x^\newn  + \xi)^{j} \rangle$. Then $d_H(C) = 3$.
\end{corollary}
\begin{proof}
    Since $C \subset \langle (x^\newn  - \xi)^{p^{s-1} + 1}(x^\newn  + \xi) \rangle$, we know,
    by Lemma \ref{Reducible.Lemma.Weight.3}, that $d_H(C) \ge 3$.
    For $(x^\newn  - \xi)^{2p^{s-1}}(x^\newn  + \xi)^{2p^{s-1}} \in C$, we have
    \be\nn
        (x^\newn  - \xi)^{2p^{s-1}}(x^\newn  + \xi)^{2p^{s-1}} = (x^{2\newn } - \xi^2)^{2p^{s-1}}
        = x^{4\newn p^{s-1}} - 2\xi^{2p^{s-1}} x^{2\newn p^{s-1}} + \xi^{4p^{s-1}} .
    \ee
    So $d_H(C) \le 3$\ and hence $d_H(C) = 3$.
\end{proof}

For $2p^{s-1} < i < p^s$\ and $0 < j \le p^{s-1}$, $d_H(C)$\ is computed in the following lemma
and corollary.
\begin{lemma}\label{Reducible.Lemma.Weight.4}
    Let $C = \langle (x^\newn  - \xi)^{2p^{s-1} + 1}(x^\newn  + \xi) \rangle $.
    Then $d_H(C) \ge 4$.
\end{lemma}
\begin{proof}
    Pick $0 \neq c(x) \in C$\ where $c(x) \equiv f(x)(x^\newn  - \xi)^{2p^{s-1} + 1}(x^\newn  + \xi) \mod x^{2\newn p^s} - \lambda$\
    for some $0 \neq f(x) \in \F_{p^m}[x]$\ with $\deg (f(x)) < 2\newn p^s - 2\newn p^{s-1} - 2\newn $.
    Let $i_0$\ and $j_0$\ be the largest integers with $(x^\newn  - \xi)^{i_0} | f(x)$\ and
    $(x^\newn  + \xi)^{j_0} | f(x)$. Then $f(x)$\ is of the form
    $f(x) = (x^\newn  - \xi)^{i_0}(x^\newn  + \xi)^{j_0}g(x)$\ for some $g(x) \in \F_{p^m}[x]$\ with $x^\newn  - \xi \nmid g(x)$\
    and $x^\newn  + \xi \nmid g(x)$. Note that $i_0 < p^s - 2p^{s-1} -1$\ or $j_0 < p^s - 1$\ holds since
    $\deg (f(x)) < 2\newn p^s - 2\newn p^{s-1} - 2\newn $.

    If $i_0 < p^s - 2p^{s-1} - 1$, then, by Lemma \ref{Preliminaries.Lemma.weight.beta}, we have
    \be\label{Reducible.Lemma.Weight.4.Ineq1}
        w_H( (x^\newn  - \xi)^{i_0 + 2p^{s-1} + 1} ) \ge 4.
    \ee
    Since $x^\newn  - \xi \nmid g(x)$,
    \be\label{Reducible.Lemma.Weight.4.Ineq2}
        w_H( g(x)(x^\newn  + \xi)^{j_0 + 1} \mod x^\newn  -\xi ) > 0
    \ee
    holds. Now using (\ref{Reducible.Lemma.Weight.4.Ineq1}), (\ref{Reducible.Lemma.Weight.4.Ineq2}) and
    (\ref{Inequality.Lower.Bound.g(x).xn+c.N}), we obtain
    \be
        w_H( c(x) ) & = & w_H( f(x)(x^\newn  - \xi)^{2p^{s-1}+1}(x^\newn  + \xi) )\nn\\
            & = & w_H( (x^\newn  - \xi)^{i_0 + 2p^{s-1} + 1}(x^\newn  + \xi)^{j_0 + 1}g(x) )\nn\\
            & \ge & w_H( (x^\newn  + \xi)^{j_0 + 1}g(x) \mod x^\newn  - \xi )w_H( (x^\newn  - \xi)^{i_0 + 2p^{s-1} +1} ) \nn\\
            & \ge & 4.\nn
    \ee

    If $i_0 \ge p^s - 2p^{s-1} -1$, then $j_0 < p^s -1$.
    Clearly $w_H( (x^{2\newn }- \xi^2)^{j_0 + 1} )\ge 2$.
    So, by Lemma \ref{Reducible.Lemma.Weight.Bound.When.i0.large}, we have
    $w_H( c(x) ) \ge 2w_H( (x^{2\newn } - \xi^2)^{j_0 + 1} )\ge 4$.
    Hence $d_H(C) \ge 4$.
\end{proof}

\begin{corollary}\label{Reducible.Corollary.Weight.4}
    Let $2p^{s-1} < i < p^s$\ and $0 < j \le p^{s-1}$\ be integers, and let
    $C = \langle (x^\newn  - \xi)^i(x^\newn  + \xi)^j \rangle$. Then $d_H(C) = 4$.
\end{corollary}
\begin{proof}
    Since $C \subset \langle (x^\newn  - \xi)^{2p^{s-1} + 1}(x^\newn  + \xi) \rangle $,
    we know, by Lemma \ref{Reducible.Lemma.Weight.4}, that $d_H(C) \ge 4$.
    For $(x^\newn  - \xi)^{p^s}(x^\newn  + \xi)^{p^{s-1}}\in C$, we have
    $w_H( (x^\newn  - \xi)^{p^s}(x^\newn  + \xi)^{p^{s-1}} ) = 4$. Thus $d_H(C) \le 4$\ and hence $d_H(C) = 4$.
\end{proof}

Next we consider the cases where $p^{s-1} < j \le i \le p^s$. We begin with computing
$d_H(C)$\ when $p^{s-1} < j \le i \le (p-1)p^{s-1}$\ in the following lemma and corollary.

\begin{lemma}\label{Reducible.Lemma.beta.beta.prime}
    Let $1 \le \beta^{'} \le \beta \le p-2$\ be integers and
    $C = \langle (x^\newn  - \xi)^{\beta p^{s-1} + 1}(x^\newn  + \xi)^{\beta^{'}p^{s-1} + 1}  \rangle$.
    Then $d_H(C) \ge \min \{ \beta + 2, 2(\beta^{'} + 2) \}$.
\end{lemma}
\begin{proof}
    Let $0 \neq c(x) \in C$. Then there exists $0 \neq f(x) \in \F_{p^m}[x]$\ such that
    $c(x) \equiv f(x)(x^\newn  - \xi)^{\beta p^{s-1} + 1}(x^\newn  + \xi)^{\beta^{'} p^{s-1} + 1} \mod x^{2\newn p^s} - \lambda$.
    We may assume that $\deg (f(x)) < 2\newn p^s - \newn \beta p^{s-1} - \newn \beta^{'}p^{s-1} -2\newn $.
    We consider the cases $\beta = \beta^{'}$\ and $\beta < \beta^{'}$\ separately.

    First, we assume that $\beta = \beta^{'}$.
    Then $C = \langle (x^\newn  - \xi)^{\beta p^{s-1} + 1}(x^\newn  + \xi)^{\beta^{'}p^{s-1} + 1} \rangle =
        \langle (x^{2\newn } - \xi^2)^{\beta p^{s-1} + 1}\rangle$.
    Let $m$\ be the largest nonnegative integer with $(x^{2\newn } - \xi^2)^m | f(x)$.
    We have $m < p^s - \beta p^{s-1} - 1$\ as $\deg (f(x)) < 2\newn p^s - 2\newn \beta p^{s-1}- 2\newn $.
    So, by Lemma \ref{Preliminaries.Lemma.weight.beta}, we get
    \be\label{Reducible.Lemma.beta.beta.prime.Weight.beta.beta.prime.equal}
        w_H( (x^{2\newn } - \xi^2)^{\beta p^{s-1} + 1 + m} ) \ge \beta + 2.
    \ee
    Clearly $f(x)$\ is of the form $f(x) = (x^{2\newn } - \xi^2)^m g(x)$\ for some $g(x) \in \F_{p^m}[x]$\
    where $x^{2\newn } - \xi^2 \nmid g(x)$. So $ g(x) \mod x^{2\newn } - \xi^2 \neq 0$\ and therefore
    \be\label{Reducible.Lemma.beta.beta.prime.Weight.Modulo.beta.beta.prime.equal}
        w_H (g(x) \mod x^{2\newn } - \xi^2) > 0.
    \ee
    So if $\beta = \beta^{'}$, then using
    (\ref{Reducible.Lemma.beta.beta.prime.Weight.beta.beta.prime.equal}),
    (\ref{Reducible.Lemma.beta.beta.prime.Weight.Modulo.beta.beta.prime.equal}) and
    (\ref{Inequality.Lower.Bound.g(x).xn+c.N}), we get
    \be\nn\label{Reducible.Lemma.beta.beta.prime.Case.beta.beta.prime.equal}
        w_H( c(x) ) & = & w_H( (x^{2\newn } - \xi^2)^{m + \beta p^{s-1} + 1} g(x) )\\
            & \ge &  w_H (g(x) \mod x^{2\newn } - \xi^2) w_H( (x^{2\newn } - \xi^2)^{m + \beta p^{s-1} + 1} )\nn\\
            & \ge & \beta + 2 \nn.
    \ee

    Second, we assume that $\beta^{'} < \beta$.
    For $c(x) \equiv f(x)(x^\newn  - \xi)^{\beta p^{s-1} + 1}(x^\newn  + \xi)^{\beta^{'}p^{s-1} + 1} \mod x^{2\newn p^s} - \lambda$,
    let $i_0$\ and $j_0$\ be the largest integers with $(x^\newn  - \xi)^{i_0} | f(x)$\ and $(x^\newn  + \xi)^{j_0} | f(x)$.
    Since $\deg (f(x)) < 2\newn p^s - \newn \beta p^{s-1} - \newn \beta^{'}p^{s-1} - 2\newn $, we have
    $i_0 + j_0 < 2p^s - \beta p^{s-1} - \beta^{'}p^{s-1} -2$.
    Thus $i_0 < p^s - \beta p^{s-1} - 1$\ or $j_0 < p^s - \beta^{'}p^{s-1} - 1$\ holds.

    If $i_0 < p^s - \beta p^{s-1} - 1$, then, by Lemma \ref{Preliminaries.Lemma.weight.beta}, we have
    \be\label{Reducible.Lemma.beta.beta.prime.weight.i0.Small}
        w_H( (x^\newn  - \xi)^{i_0 + \beta p^{s-1} + 1} ) \ge \beta + 2.
    \ee
    Note that $(x^\newn  + \xi)^{j_0 + \beta^{'}p^{s-1} + 1}g(x) \mod x^\newn  - \xi \neq 0$\ since
    $x^\newn  - \xi \nmid (x^\newn  + \xi)^{j_0 + \beta^{'}p^{s-1} + 1}g(x)$. Therefore
    \be\label{Reducible.Lemma.beta.beta.prime.Weight.Modulo.beta.large}
        w_H ((x^\newn  + \xi)^{j_0 + \beta^{'}p^{s-1} + 1}g(x) \mod x^\newn  - \xi ) > 0.
    \ee
    Using (\ref{Inequality.Lower.Bound.g(x).xn+c.N}),
    (\ref{Reducible.Lemma.beta.beta.prime.weight.i0.Small}) and
    (\ref{Reducible.Lemma.beta.beta.prime.Weight.Modulo.beta.large}), we obtain
    \be\label{Reducible.Lemma.beta.beta.prime.Case2.1}
        \begin{array}{rcl}
            w_H( c(x) ) & = & w_H( (x^\newn  - \xi)^{i_0 + \beta p^{s-1} + 1}(x^\newn  + \xi)^{j_0 + \beta^{'} p^{s-1} + 1}g(x) )\\
            & \ge & w_H( (x^\newn  + \xi)^{j_0 + \beta^{'}p^{s-1} + 1}g(x) \mod x^\newn  - \xi )w_H( (x^\newn  - \xi)^{i_0 + \beta p^{s-1} + 1} )\\
            & \ge &  \beta + 2.
        \end{array}
    \ee

    If $i_0 \ge p^s - \beta p^{s-1} - 1$, then $j_0 < p^s - \beta^{'}p^{s-1} - 1$.
    By Lemma \ref{Reducible.Lemma.Weight.Bound.When.i0.large}, we get
    \be\label{Reducible.Lemma.beta.beta.prime.weight.PreBound.i0.Large}
        w_H( c(x) )\ge 2w_H( (x^{2\newn } - \xi^2)^{j_0 + \beta^{'}p^{s-1} + 1} ).
    \ee
    For $w_H( (x^{2\newn } - \xi^2)^{j_0 + \beta^{'}p^{s-1} + 1} )$,
    we use Lemma \ref{Preliminaries.Lemma.weight.beta} and get
    \be\label{Reducible.Lemma.beta.beta.prime.weight.beta.Prime}
        w_H( (x^{2\newn } - \xi^2)^{j_0 + \beta^{'}p^{s-1} + 1} ) \ge \beta^{'} + 2.
    \ee
    Combining (\ref{Reducible.Lemma.beta.beta.prime.weight.PreBound.i0.Large}) and
    (\ref{Reducible.Lemma.beta.beta.prime.weight.beta.Prime}), we obtain
    \be\label{Reducible.Lemma.beta.beta.prime.Weight.i0Large}
        w_H( c(x) ) \ge 2(\beta^{'} + 2).
    \ee
    So if $\beta^{'} < \beta$, then, by (\ref{Reducible.Lemma.beta.beta.prime.Case2.1}) and
    (\ref{Reducible.Lemma.beta.beta.prime.Weight.i0Large}), we get that
    $w_H( c(x) ) \ge \min \{ \beta + 2, 2(\beta^{'} + 2) \}$.
    In both cases, namely $\beta = \beta^{'}$\ and $\beta^{'} < \beta $,
    we have shown that
    $d_H(C) \ge \min \{ \beta + 2, 2(\beta^{'} + 2) \}$.
\end{proof}

\begin{corollary}\label{Reducible.Corollary.beta.beta.prime}
    Let $j \le i $, $1\le \beta^{'} \le \beta \le p - 2$\ be integers
    such that
    \be\nn
        \begin{array}{rcccl}
            \beta p^{s-1} + 1 & \le & i & \le & (\beta + 1 )p^{s-1}\quad \mbox{and}\\
            \beta^{'}p^{s-1} + 1 & \le & j & \le & (\beta^{'} + 1 )p^{s-1}.
        \end{array}
    \ee
    Let $C = \langle (x^\newn  - \xi)^i(x^\newn  + \xi)^j \rangle$.
    Then $d_H(C) = \min \{ \beta + 2, 2(\beta^{'} + 2) \}$.
\end{corollary}
\begin{proof}
    We know, by Lemma \ref{Reducible.Lemma.beta.beta.prime},
    that $d_H(C) \ge \min \{ \beta + 2, 2(\beta^{'} + 2) \}$.
    So it suffices to show $d_H(C) \le \min \{ \beta + 2, 2(\beta^{'} + 2) \}$.

    First, $(\beta + 1)p^{s-1} \ge i,j$\ implies that
    $(x^\newn  - \xi)^{(\beta + 1)p^{s-1}}(x^\newn  + \xi)^{(\beta + 1)p^{s-1}} =
     (x^{2\newn } - \xi^2)^{(\beta + 1)p^{s-1}} \in C$.
    By (\ref{Equality.weight.xn+c.N}), we get
    $w_H((x^{2\newn } - \xi^2)^{(\beta + 1)p^{s-1}}) = \beta + 2$.
    Therefore
    \be\label{Reducible.Corollary.beta.beta.prime.Bound.beta}
        d_H(C) \le \beta + 2.
    \ee

    Second, we consider $(x^\newn  - \xi)^{p^s}(x^\newn  + \xi)^{(\beta^{'} + 1)p^{s-1}} \in C$.
    Using (\ref{Equality.weight.xn+c.N}) and the fact that $p^s > (\beta^{'} + 1)p^{s-1}$,
    we get
    \be\nn
        w_H( (x^\newn  - \xi)^{p^s}(x^\newn  + \xi)^{(\beta^{'} + 1)p^{s-1}} )
        = 2w_H( (x^\newn  + \xi)^{(\beta^{'} + 1)p^{s-1}} ) = 2(\beta^{'} + 2).
    \ee
    So
    \be\label{Reducible.Corollary.beta.beta.prime.Bound.beta.prime}
        d_H(C) \le 2(\beta^{'} + 2).
    \ee
    Combining (\ref{Reducible.Corollary.beta.beta.prime.Bound.beta}) and
    (\ref{Reducible.Corollary.beta.beta.prime.Bound.beta.prime}),
    we deduce that $d_H(C) \le \min \{ \beta + 2, 2(\beta^{'} + 2) \}$.
    Therefore
        $d_H(C) = \min \{ \beta + 2, 2(\beta^{'} + 2) \}$.
\end{proof}

The following lemma and corollary deal with the case where $p^{s-1} < j \le (p-1)p^{s-1} < i < p^s$.

\begin{lemma}\label{Reducible.Lemma.beta.tauK}
    Let $1 \le \tau \le p-1$, $1\le \beta \le p-2$, $1\le k \le s-1$\ be integers and
    $C = \langle (x^\newn  - \xi)^{p^s - p^{s-k} + (\tau - 1)p^{s-k-1} + 1}(x^\newn  + \xi)^{\beta p^{s-1} +1} \rangle$. Then $d_H(C) \ge 2 (\beta + 2)$.
\end{lemma}
\begin{proof}
    Let $0 \neq c(x) \in C$. Then there exists $0 \neq f(x) \in \F_{q}[x]$\
    such that $c(x) \equiv (x^\newn  - \xi)^{p^s - p^{s-k} + (\tau -1)p^{s-k-1} + 1}(x^\newn  + \xi)^{\beta p^{s-1} + 1}f(x) \mod x^{2\newn p^s}- \lambda$\ and $\deg (f(x)) < \newn p^s + \newn p^{s-k} - \newn(\tau - 1)p^{s-k-1} - \newn \beta p^{s-1} -2\newn $. Let $i_0$\ and $j_0$\ be the largest integers with $(x^\newn  - \xi)^{i_0} | f(x)$\ and
    $(x^\newn  + \xi)^{j_0} | f(x)$. Then $f(x)$\ is of the form
    $f(x) = (x^\newn  - \xi)^{i_0}(x^\newn  + \xi)^{j_0}g(x)$\ for some $g(x) \in \F_{p^m}[x]$\ such that
    $x^\newn  - \xi \nmid g(x)$\ and $x^\newn  + \xi \nmid g(x)$.
    Clearly $i_0 + j_0 < p^s + p^{s-k} - (\tau - 1)p^{s-k-1} -  \beta p^{s-1} -2$.
    So $i_0 < p^{s-k} - (\tau -1)p^{s-k-1} - 1$\ or
    $j_0 < p^s - \beta p^{s-1} -1$\ holds.

    If $i_0 < p^{s-k} - (\tau -1)p^{s-k-1} - 1$, then, by Lemma
    \ref{Reducible.Lemma.Weight.Bound.When.i0.large},
    we have
    \be\label{Reducible.Lemma.beta.tauK.i0.small.Ineq1}
        w_H( (x^\newn  - \xi)^{i_0 + p^s - p^{s-k} + (\tau -1)p^{s-k-1} + 1} ) \ge (\tau + 1)p^k.
    \ee
    Since $x^\newn  - \xi \nmid g(x)$,
    \be\label{Reducible.Lemma.beta.tauK.i0.small.Ineq2}
        w_H( (x^\newn  + \xi)^{j_0 + \beta p^{s-1} + 1}g(x) \mod x^\newn  - \xi) > 0.
    \ee
    Using (\ref{Reducible.Lemma.beta.tauK.i0.small.Ineq1}),
    (\ref{Reducible.Lemma.beta.tauK.i0.small.Ineq2}) and
    (\ref{Inequality.Lower.Bound.g(x).xn+c.N}), we obtain
    \be\nn
        w_H( c(x) ) & = & w_H( (x^\newn  - \xi)^{i_0 + p^s - p^{s-k} + (\tau -1)p^{s-k-1} + 1}
                (x^\newn  + \xi)^{j_0 + \beta p^{s-1} + 1}g(x) )\nn\\
            & \ge & w_H( (x^\newn  + \xi)^{j_0 + \beta p^{s-1} + 1}g(x)\mod x^\newn  - \xi )
                w_H((x^\newn  - \xi)^{i_0 + p^s - p^{s-k} + (\tau -1)p^{s-k-1} + 1})\nn\\
            & \ge & (\tau + 1)p^k\nn\\
            & \ge & 2p\nn\\
            & \ge & 2(\beta + 2).\nn
    \ee

    If $i_0 \ge p^{s-k} - (\tau -1)p^{s-k-1} -1$, then
    $j_0 < p^s - \beta p^{s-1} -1$. So, by Lemma \ref{Reducible.Lemma.Weight.Bound.When.i0.large},
    we get
    \be\label{Reducible.Lemma.beta.tauK.io.Large.Ineq}
        w_H( c(x) ) \ge 2w_H( (x^{2\newn} - \xi^2)^{j_0 + \beta p^{s-1} + 1} ).
    \ee
    For $w_H( (x^{2\newn} - \xi^2)^{j_0 + \beta p^{s-1} + 1} )$,
    we use Lemma \ref{Preliminaries.Lemma.weight.beta} and get
    \be\label{Reducible.Lemma.beta.tauK.io.Large.Eq}
        w_H( (x^{2\newn } - \xi^2)^{j_0 + \beta p^{s-1} + 1} ) = \beta + 2.
    \ee
    Combining (\ref{Reducible.Lemma.beta.tauK.io.Large.Ineq})\ and
    (\ref{Reducible.Lemma.beta.tauK.io.Large.Eq}), we obtain
    $w_H( c(x) ) \ge 2(\beta + 2)$.
    So $d_H(C) \ge 2(\beta + 2)$.
\end{proof}

\begin{corollary}\label{Reducible.Corollary.beta.tauK}
    Let $i,j,1 \le \tau \le p-1, 1\le \beta \le p -2$\ and
    $1 \le k \le s-1$\ be integers such that
    \be\nn
        \begin{array}{rcccl}
            p^s - p^{s-k} + (\tau -1)p^{s-k-1} + 1 & \le & i & \le & p^s - p^{s-k} + \tau p^{s-k-1}\quad \mbox{and}\\
            \beta p^{s-1} + 1 & \le & j & \le & (\beta + 1)p^{s-1}.
        \end{array}
    \ee
    Let $C = \langle (x^\newn  - \xi)^i(x^\newn  + \xi)^j \rangle $.
    Then $d_H(C) = 2 (\beta + 2)$.
\end{corollary}
\begin{proof}
    Since $\langle (x^\newn  - \xi)p^{p^s - p^{s-k}+ (\tau -1)p^{s-k-1} +1}(x^\newn  + \xi)^{\beta p^{s-1} + 1} \rangle \supset C$, we know, by Lemma \ref{Reducible.Lemma.beta.tauK}, that
    $d_H(C) \ge 2(\beta + 2)$. So it suffices to show $d_H(C) \le 2(\beta + 2)$.
    We consider $(x^\newn  - \xi)^{p^s}(x^\newn  + \xi)^{(\beta + 1) p^{s-1}} \in C$.
    Note that
    $w_H( (x^\newn  - \xi)^{(\beta+1) p^{s-1}} ) = \beta + 2$
    by (\ref{Equality.weight.xn+c.N}). So, using the fact that
    $p^s > (\beta + 1) p^{s-1} $, we obtain
    $w_H( (x^\newn  - \xi)^{p^s}(x^\newn  + \xi)^{(\beta + 1) p^{s-1}} ) = 2 (\beta + 2)$.
    So $d_H(C) \le 2(\beta + 2)$, and hence $d_H(C) = 2(\beta + 2)$.
\end{proof}

From Lemma \ref{Reducible.Lemma.tauK.same.k} till Corollary \ref{Reducible.Corollary.tauK.different.k},
we compute $d_H(C)$\ when $(p-1)p^{s-1} < j \le i < p^s$.

\begin{lemma}\label{Reducible.Lemma.tauK.same.k}
    Let $1 \le k \le s-1, 1 \le \tau^{'} \le \tau \le p-1$,
    \be
        i & = & p^s - p^{s-k} + (\tau - 1)p^{s-k-1} + 1 \quad \mbox{and}\nn\\
        j & = & p^s - p^{s-k} + (\tau^{'} - 1)p^{s-k-1} + 1 \nn
    \ee
    be integers and $C = \langle (x^\newn  - \xi)^i(x^\newn  + \xi)^j \rangle$.
    Then $d_H(C) \ge \min \{ 2(\tau^{'}+1)p^k, (\tau + 1)p^k \}$.
\end{lemma}
\begin{proof}
    Let $0 \neq c(x) \in C$. Then there exists $0 \neq f(x) \in \F_{p^m}[x]$\ such that
    $c(x) \equiv f(x)(x^\newn  - \xi)^i(x^\newn  + \xi)^j \mod x^{2\newn p^s} - \lambda$\ and
    $\deg (f(x)) < 2\newn p^s - i\newn -j \newn$. Let $i_0$\ and $j_0$\ be the largest integers
    with $(x^\newn  - \xi)^{i_0} | f(x)$\ and $(x^\newn  + \xi)^{j_0} | f(x)$.
    Then $f(x)$\ is of the form
    $f(x) = (x^\newn  - \xi)^{i_0}(x^\newn  + \xi)^{j_0}g(x)$\ for some $g(x) \in \F_{p^m}[x]$\
    with $x^\newn  - \xi \nmid g(x)$\ and $x^\newn  + \xi \nmid g(x)$.
    Clearly $i_0 + j_0 < 2p^{s} - i - j$\ and therefore $i_0 < p^s -i$\ or $j_0 < p^s -j$
    holds.

    If $i_0 < p^s -i$, then by Lemma \ref{Preliminaries.Lemma.weight.tau.k}, we have
    \be\label{Reducible.Lemma.tauK.same.k.small.io.Ineq1}
        w_H( (x^\newn  - \xi)^{i_0 + i} ) \ge (\tau + 1)p^k.
    \ee
    Since $x^\newn  - \xi \nmid g(x)$, we have
    $g(x)(x^\newn  + \xi)^{j_0 + j} \not \equiv 0 \mod x^\newn  - \xi$\ and therefore
    \be\label{Reducible.Lemma.tauK.same.k.small.io.Ineq2}
        w_H( g(x)(x^\newn  + \xi)^{j + j_0} \mod x^\newn  - \xi ) > 0.
    \ee
    Using (\ref{Reducible.Lemma.tauK.same.k.small.io.Ineq1}),
    (\ref{Reducible.Lemma.tauK.same.k.small.io.Ineq2}) and
    (\ref{Inequality.Lower.Bound.g(x).xn+c.N}), we obtain
    \be\label{Reducible.Lemma.tauK.same.k.small.io}
        \begin{array}{rcl}
            w_H( c(x) ) & = & w_H( (x^\newn  - \xi)^{i + i_0}(x^\newn  + \xi)^{j + j_0}g(x) )\\
                & \ge & w_H( g(x)(x^\newn  + \xi)^{j + j_0} \mod x^\newn  - \xi )
                    w_H( (x^\newn  - \xi)^{i + i_0} ) \\
                & \ge & (\tau + 1)p^k.
        \end{array}
    \ee

    If $i_0 \ge p^s -i$, then $j_0 < p^s - j$.
    So, by Lemma \ref{Reducible.Lemma.Weight.Bound.When.i0.large},
    we have
    \be\label{Reducible.Lemma.tauK.same.k.large.io.Ineq1}
        w_H( c(x) ) \ge 2w_H( (x^{2\newn } - \xi^2)^{j_0 + j} ).
    \ee
    For $w_H( (x^{2\newn } - \xi^2)^{j_0 + j} )$,
    we use Lemma \ref{Preliminaries.Lemma.weight.tau.k} and get
    \be\label{Reducible.Lemma.tauK.same.k.large.io.Ineq2}
        w_H( (x^{2\newn } - \xi^2)^{j_0 + j} ) \ge (\tau^{'} + 1)p^{k}.
    \ee
    Combining (\ref{Reducible.Lemma.tauK.same.k.large.io.Ineq1}) and
    (\ref{Reducible.Lemma.tauK.same.k.large.io.Ineq2}), we obtain
    \be\label{Reducible.Lemma.tauK.same.k.large.io}
        w_H( c(x) ) \ge 2(\tau^{'} + 1)p^{k}.
    \ee
    Now, using (\ref{Reducible.Lemma.tauK.same.k.small.io}) and
    (\ref{Reducible.Lemma.tauK.same.k.large.io}), we deduce that
    $w_H(c(x)) \ge \min\{ 2(\tau^{'} + 1)p^{k}, (\tau + 1)p^k \}$.
    Hence $d_H(C) \ge \min \{ 2(\tau^{'}+1)p^k, (\tau + 1)p^k \}$.
\end{proof}

\begin{corollary}\label{Reducible.Corollary.tauK.same.k}
    Let $j \le i $, $1\le k \le s-1$, $1\le \tau^{'} \le \tau \le p-1 $\
    be integers such that
    \be\nn
        \begin{array}{rcccl}
            p^s - p^{s-k} + (\tau -1)p^{s-k-1} + 1 & \le & i & \le & p^s - p^{s-k} + \tau p^{s-k-1} \quad \mbox{and} \\
            p^s - p^{s-k} + (\tau^{'} -1)p^{s-k-1} + 1 & \le & j & \le & p^s - p^{s-k} + \tau^{'} p^{s-k-1}.
        \end{array}
    \ee
    Let $C = \langle (x^\newn  - \xi)^i(x^\newn  + \xi)^j \rangle$. Then
    $d_H( C ) = \min \{ 2(\tau^{'} +1)p^{k}, (\tau + 1)p^k \}$.
\end{corollary}
\begin{proof}
    Since $\langle (x^\newn  - \xi)^{p^s - p^{s-k} + (\tau -1)p^{s-k-1} + 1}(x^\newn  + \xi)^{p^s - p^{s-k} + (\tau^{'} -1)p^{s-k-1} + 1} \rangle \supset C$, we have,
    by Lemma \ref{Reducible.Lemma.tauK.same.k}, that
    $d_H(C) \ge \min \{ 2(\tau^{'} +1)p^{k}, (\tau + 1)p^k \}$.
    So it suffices to show $d_H(C) \le \min \{ 2(\tau^{'} +1)p^{k}, (\tau + 1)p^k \}$.

    First, we consider $(x^\newn  - \xi)^{p^s}(x^\newn  + \xi)^{p^s - p^{s-k} + \tau^{'}p^{s-k-1}}\in C$.
    Since
    \be\nn
        w_H( (x^\newn  + \xi)^{p^s - p^{s-1} + \tau^{'} p^{s-k-1}} ) = (\tau^{'} + 1)p^{k},
    \ee
    we have
    $w_H( (x^\newn  - \xi)^{p^s}(x^\newn  + \xi)^{p^s - p^{s-1} + (\tau^{'} -1)p^{s-k-1}} ) = 2(\tau ^{'} + 1)p^{k}$.
    So
    \be\label{Reducible.Corollary.tauK.same.k.Ineq1}
        d_H(C) \le 2(\tau^{'} + 1)p^{k}
    \ee

    Second, we consider $(x^{2\newn } - \xi^2)^{p^s - p^{s-k} + (\tau -1)p^{s-k-1} + 1} \in C$.
    By Lemma \ref{Equality.weight.xn+c.N}, we get
    \be\nn
        w_H( (x^{2\newn } - \xi^2)^{p^s - p^{s-k} + (\tau -1)p^{s-k-1} + 1} ) = (\tau + 1)p^k.
    \ee
    Thus
    \be\label{Reducible.Corollary.tauK.same.k.Ineq2}
        d_H(C) \le (\tau + 1)p^k.
    \ee
    Now combining (\ref{Reducible.Corollary.tauK.same.k.Ineq1}) and
    (\ref{Reducible.Corollary.tauK.same.k.Ineq2}), we deduce that
    $d_H(C) \le \min \{ 2(\tau^{'} +1)p^{k}, (\tau + 1)p^k \}$.
    Hence $d_H(C) = \min \{ 2(\tau^{'} +1)p^{k}, (\tau + 1)p^k \}$.
\end{proof}

\begin{lemma}\label{Reducible.Lemma.tauK.different.k}
    Let $1 \le k^{'} < k \le s-1$, $1\le \tau^{'},\tau < p-1$,
    \be
        i & = & p^s - p^{s-k} + (\tau -1)p^{s-k-1} + 1 \quad \mbox{and} \nn \\
        j & = & p^s - p^{s-k^{'}} + (\tau^{'} -1)p^{s-k^{'}-1} + 1 \nn
    \ee
    be integers and $C = \langle (x^\newn  - \xi)^i(x^\newn  + \xi)^j \rangle $.
    Then $d_H(C) \ge 2(\tau^{'} + 1)p^{k^{'}}$.
\end{lemma}
\begin{proof}
    Let $0 \neq c(x) \in C$. Then there exists $0 \neq f(x) \in \F_{p^m}[x]$\ such that
    $c(x) \equiv (x^\newn  - \xi)^i(x^\newn  + \xi)^jf(x) \mod x^{2\newn p^s} - \lambda$\ and
    $\deg (f(x)) < 2\newn p^s -i \newn - j \newn$.
    Let $i_0$\ and $j_0$\ be the largest integers with $(x^\newn  - \xi)^{i_0} | f(x)$\
    and $(x^\newn  + \xi)^{j_0} | f(x)$. Then $f(x)$\ is of the form
    $f(x) = (x^\newn  - \xi)^{i_0}(x^\newn  + \xi)^{j_0}g(x)$\ for some $g(x) \in \F_{p^m}[x]$\ with
    $x^\newn  - \xi \nmid g(x)$\ and $x^\newn  + \xi \nmid g(x)$.
    Clearly $i_0 + j_0 < 2p^s - i- j$. So $i_0 < p^s - i$\ or $j_0 < p^s - j$\ holds.

    If $i_0 < p^s - i$, then, by Lemma \ref{Preliminaries.Lemma.weight.tau.k}, we have
    \be\label{Reducible.Lemma.tauK.different.k.Small.i0.Ineq1}
        w_H( (x^\newn  - \xi)^{i + i_0} ) \ge (\tau + 1)p^k \ge 2(\tau^{'} + 1)p^{k^{'}}.
    \ee
    Since $x^\newn  - \xi \nmid g(x)$, we have
    $(x^\newn  + \xi)^{j_0 + j}g(x) \mod x^\newn  - \xi \neq 0$\ and therefore
    \be\label{Reducible.Lemma.tauK.different.k.Small.i0.Ineq2}
        w_H( (x^\newn  + \xi)^{j_0 + j}g(x) \mod x^\newn  - \xi ) > 0.
    \ee
    Using (\ref{Reducible.Lemma.tauK.different.k.Small.i0.Ineq1}),
    (\ref{Reducible.Lemma.tauK.different.k.Small.i0.Ineq2}) and
    (\ref{Inequality.Lower.Bound.g(x).xn+c.N}), we obtain
    \be
        w_H( c(x) ) & = & w_H( (x^\newn  - \xi)^{i_0 + i}(x^\newn  + \xi)^{j_0 + j}g(x) )\nn \\
            & \ge & w_H( (x^\newn  + \xi)^{j_0 + j}g(x) \mod x^\newn  - \xi )
                w_H( (x^\newn  - \xi)^{i_0 + i} )\nn \\
            & \ge & 2(\tau^{'} + 1)p^{k^{'}}.\nn
    \ee

    If $i_0 \ge p^s - i$, then $j_0 < p^s - j$.
    So, by Lemma \ref{Reducible.Lemma.Weight.Bound.When.i0.large}, we have
    \be\label{Reducible.Lemma.tauK.different.k.Large.i0.Ineq1}
        w_H( c(x) ) \ge 2w_H( (x^{2\newn } - \xi^2)^{j_0 + j} ).
    \ee
    For $w_H( (x^{2\newn } - \xi^2)^{j_0 + j} )$,
    we use Lemma \ref{Preliminaries.Lemma.weight.tau.k} and get
    \be\label{Reducible.Lemma.tauK.different.k.Large.i0.Ineq2}
        w_H( (x^{2\newn } - \xi^2)^{j_0 + j} ) \ge (\tau^{'} + 1)p^{k^{'}}.
    \ee
    Now combining (\ref{Reducible.Lemma.tauK.different.k.Large.i0.Ineq1})
    and (\ref{Reducible.Lemma.tauK.different.k.Large.i0.Ineq2}), we obtain
    $ w_H( c(x) ) \ge 2(\tau ^ {'} + 1)p^{k^{'}}$.
    Hence $d_H(C) \ge 2 (\tau^{'} + 1)p^{k^{'}}$.
\end{proof}

\begin{corollary}\label{Reducible.Corollary.tauK.different.k}
    Let $i,j,1\le k^{'} < k \le s-1, 1\le \tau^{'},\tau \le p-1$\ be integers
    such that
    \be\nn
        \begin{array}{rcccl}
             p^s - p^{s-k} + (\tau - 1)p^{s-k-1} + 1 & \le & i & \le & p^s - p^{s-k} + \tau p^{s-k-1} \quad \mbox{and}\\
             p^s - p^{s-k^{'}} + (\tau^{'} - 1)p^{s-k^{'}-1} + 1&\le & j & \le & p^s - p^{s-k^{'}} + \tau^{'} p^{s-k^{'}-1}.
        \end{array}
    \ee
    Let $C = \langle (x^\newn  - \xi)^i (x^\newn  + \xi)^j \rangle$.
    Then $d_H(C) = 2 (\tau^{'} + 1)p^{k^{'}}$.
\end{corollary}
\begin{proof}
    Since $\langle (x^\newn  - \xi)^{p^s - p^{s-k} + (\tau -1)p^{s-k-1}+1} (x^\newn  + \xi)^{p^s - p^{s-k^{'}} + (\tau^{'} -1)p^{s-k^{'}-1}+1} \rangle \supset C$, we know, by Lemma
    \ref{Reducible.Lemma.tauK.different.k}, that $d_H(C) \ge 2(\tau^{'} + 1)p^{k^{'}}$.
    So it suffices to show $d_H(C) \le 2(\tau^{'} + 1)p^{k^{'}}$.
    We consider $(x^\newn  - \xi)^{p^s}(x^\newn  + \xi)^{p^s - p^{s-k^{'}} +  \tau^{'} p^{s-k^{'}-1}} \in C$.
    By (\ref{Equality.weight.xn+c.N}), we have
    \be\nn
        w_H( (x^\newn  + \xi)^{p^s - p^{s-k^{'}} + \tau^{'} p^{s-k^{'}-1}} ) = (\tau ^{'} + 1)p^{k^{'}}.
    \ee
    Moreover since $(x^\newn  - \xi)^{p^s} = x^{\newn p^s} - \xi ^{p^s}$\ and
    $p^s > p^s - p^{s-k^{'}} + \tau^{'} p^{s-k^{'}-1}$, we get
    \be\nn
        w_H( (x^\newn  - \xi)^{p^s}(x^\newn  + \xi)^{p^s - p^{s-k^{'}} + \tau^{'} p^{s-k^{'}-1}} ) = 2(\tau ^{'} + 1)p^{k^{'}}.
    \ee
    So $d_H(C) \le 2(\tau ^{'} + 1)p^{k^{'}}$\ and therefore $d_H(C) = 2(\tau ^{'} + 1)p^{k^{'}}$.
\end{proof}

Finally it remains to consider the cases where $i = p^s$\ and $0 < j < p^s$.

\begin{lemma}\label{Reducible.Lemma.ips.j.small}
    Let $C = \langle (x^\newn  - \xi)^{p^s}(x^\newn  + \xi) \rangle$.
    Then $d_H(C) \ge 4$.
\end{lemma}
\begin{proof}
    Pick $0 \neq c(x) \in C$. Then there exists $0 \neq f(x) \in \F_{p^m}[x]$\
    such that $c(x) \equiv f(x)(x^\newn  - \xi)^{p^s}(x^\newn  + \xi) \mod x^{2\newn p^s} - \lambda$\ and
    $\deg (f(x)) < 2\newn p^s - \newn p^s - \newn = \newn p^s - \newn$. Let $i_0$\ and $j_0$\ be the largest
    nonnegative integers such that $(x^\newn  - \xi)^{i_0} | f(x)$\ and $(x^\newn  + \xi)^{j_0} | f(x)$.
    Clearly $i_0 + j_0 < p^s -1$\ as $\deg (f(x)) < \newn p^s - \newn$.
    So, since $i_0 \ge p^s - p^s = 0$\ and $j_0 < p^s -1$,
    by Lemma \ref{Reducible.Lemma.Weight.Bound.When.i0.large},
    we get
      $w_H( c(x) ) \ge 2 w_H( (x^{2\newn } - \xi^2)^{j_0 + 1} )$.
    Obviously $w_H((x^{2\newn } - \xi^2)^{j_0 + 1}) \ge 2$\ and therefore
    $w_H( c(x) ) \ge 4$. Hence $d_H( C ) \ge 4$.
\end{proof}
\begin{corollary}\label{Reducible.Corollary.ips.j.small}
    Let $0 < j \le p^{s-1}$\ be an integer and
    $C= \langle (x^\newn  - \xi)^{p^s}(x^\newn  + \xi)^j \rangle$.
    Then $d_H(C) = 4$.
\end{corollary}
\begin{proof}
    Since $\langle (x^\newn  - \xi)^{p^s}(x^\newn  + \xi) \rangle \supset C$,
    we know, by Lemma \ref{Reducible.Lemma.ips.j.small}, that $d_H(C) \ge 4$.
    So it suffices to show $d_H(C) \le 4$.
    We consider $(x^\newn  - \xi)^{p^s}(x^\newn  + \xi)^{p^{s-1}} \in C$.
    Clearly $w_H( (x^\newn  - \xi)^{p^s}(x^\newn  + \xi)^{p^{s-1}} ) = 4$.
    So $d_H(C) \le 4$\ and hence $d_H(C) = 4$.
\end{proof}

For $i = p^s$\ and $p^{s-1} < j < p^s$, the Hamming distance of $C$\ is
computed in the following lemmas and corollaries. Their proofs are similar to those of
Lemma \ref{Reducible.Lemma.ips.j.small} and Corollary \ref{Reducible.Lemma.ips.j.small}.

\begin{lemma}\label{Reducible.Lemma.ips.j.beta}
    Let $1 \le \beta \le p-2$\ be an integer and
    $C = \langle (x^\newn  - \xi)^{p^s}(x^\newn  + \xi)^{\beta p^{s-1} +1} \rangle$.
    Then $d_H(C) \ge 2(\beta + 2)$.
\end{lemma}

\begin{corollary}\label{Reducible.Corollary.ips.j.beta}
    Let $1 \le \beta \le p-2$, $\beta p^{s-1} + 1 \le j \le (\beta + 1)p^{s-1}$\
    be integers. Let $C = \langle (x^\newn  - \xi)^{p^s}(x^\newn  + \xi)^{j} \rangle $.
    Then $d_H(C) = 2(\beta + 2)$.
\end{corollary}

\begin{lemma}\label{Reducible.Lemma.ips.j.tau.K}
    Let $1 \le \tau \le p-1, 1\le k \le s-1,j$\ be integers and
    $C = \langle (x^\newn  - \xi)^{p^s}(x^\newn  + \xi)^{p^s - p^{s-k} + (\tau -1)p^{s-k-1} + 1} \rangle$.
    Then $d_H(C) \ge 2 (\tau + 1)p^k $.
\end{lemma}

\begin{corollary}\label{Reducible.Corollary.ips.j.tau.K}
    Let $1 \le \tau \le p-1, 1\le k \le s-1,j$\ be integers such that
    \be\nn
        p^s - p^{s-k} + (\tau -1)p^{s-k-1} + 1 \le j \le p^s - p^{s-k} + \tau p^{s-k-1}.
    \ee
    Let $C = \langle (x^\newn  - \xi)^{p^s}(x^\newn  + \xi)^j \rangle$.
    Then $d_H(C) = 2 (\tau + 1)p^k$.
\end{corollary}

We summarize our results in the following theorem.

\begin{theorem}\label{Reducible.Theorem.Main}
    Let $p$\ be an odd prime, $a,s,n$\ be arbitrary positive integers.
    Let $\lambda, \xi, \psi \in \F_{p^m} \setminus \{ 0 \}$\ such that
    $\lambda = \psi^{p^s}$. Suppose that the polynomial $x^{2\newn } - \psi$\ factors
    into two irreducible polynomials as
    $x^{2\newn } - \psi = (x^\newn  - \xi)(x^\newn  + \xi)$.
    Then all $\lambda$-cyclic codes, of length $2\newn p^s$, over $\F_{p^m}$\
    are of the form $\id{ (x^\newn  - \xi)^i(x^\newn  + \xi)^j } \subset
    \F_{p^m}[x] / \langle x^{2\newn p^s} - \lambda \rangle$, where
    $0 \le i,j \le p^s$\ are integers.
    Let $C = \langle (x^\newn  - \xi)^i(x^\newn  + \xi)^j \rangle \subset
    \F_{p^m}[x] / \langle x^{2\newn p^s} - \lambda \rangle$.
    If $(i,j) = (0,0)$, then $C$\ is the whole space $\F_{p^m}^{2\newn p^s}$,
    and if $(i,j) = (p^s,p^s)$, then $C$\ is the zero space $\{\mathbf{0}\}$.
    For the remaining values of $(i,j)$, the Hamming distance of $C$\
    is given in Table \ref{table.Reducible.Hamming.Distance}.
\end{theorem}

\begin{remark}\label{Reducible.Remark.Hamming.Distance.table}
    There are some symmetries in most of the cases, so we made the following simplification
    in Table \ref{table.Reducible.Hamming.Distance}.
    For the cases with *, i.e., the cases except 2 and 7,
    we gave the Hamming distance of $C$\ when $i \ge j$. The corresponding case with $j \ge i$
    has the same Hamming distance. For example in 1*, the corresponding case is $i= 0 $\ and
    $0 \le j \le p^s$, and the Hamming distance is $2$. Similarly in 6*, the corresponding case is
    $\beta p^{s-1} + 1 \le i \le (\beta + 1)p^{s-1}$\ and
    $p^{s}-p^{s-k}+(\tau-1)p^{s-k-1} + 1 \le j \le p^{s}-p^{s-k}+\tau p^{s-k-1}  $,
    and the Hamming distance is $2(\beta + 2)$.
\end{remark}

\begin{table}\caption{The Hamming distance of all non-trivial constacyclic codes, of the form $\langle (x^\newn -\xi)^i(x^\newn +\xi)^j \rangle$, of length $2\newn p^s$\ over $\F_{p^m}$.
The polynomials $x^\newn -\xi$\ and $x^\newn  + \xi$\ are assumed to be irreducible.
The parameters $1 \le \beta^{'} \le \beta \le p-2$, $ 1 \le \tau^{(2)} < \tau^{(1)} \le p-1$,
$1 \le \tau, \tau^{(3)}, \tau^{(4)} \le p-1$ , $1 \le k \le s-1$,
    $1 \le k^{''} < k^{'} \le s-1$\ below are integers. For the cases with *, i.e., the cases except 2 and 7, see Remark \ref{Reducible.Remark.Hamming.Distance.table}}
\label{table.Reducible.Hamming.Distance}
\centering
\newcounter{ijCase}
\setcounter{ijCase}{1}
\begin{tabular}{|l|l|l|l|}
        \hline
        \textbf{Case} & \textbf{i} & $\textbf{j}$ & $\textbf{d}_{\textbf{H}}\textbf{(C)}$ \\
        \hline
        \arabic{ijCase}*\addtocounter{ijCase}{1} & $0 < i \le p^s$ & $j=0$ & $2$\\
        \hline
        \arabic{ijCase}\addtocounter{ijCase}{1} & $ 0 \le i \le p^{s-1}$ & $0 \le j \le p^{s-1}$ & $2$\\
        \hline
        \arabic{ijCase}*\addtocounter{ijCase}{1} & $p^{s-1} < i \le 2p^{s-1} $ & $0 < j \le p^{s-1}$ & $3$\\
         \hline \arabic{ijCase}*\addtocounter{ijCase}{1} &
        $ 2p^{s-1} < i \le p^s$ & $0 < j \le p^{s-1} $ & $4$\\
         \hline
        \arabic{ijCase}*\addtocounter{ijCase}{1} & $\beta p^{s-1} + 1 \le i \le (\beta + 1)p^{s-1}$
        &  $\beta^{'} p^{s-1} + 1 \le j \le (\beta^{'} + 1)p^{s-1}$ &
                        $ \begin{array}{l}
                                               \min\{\beta+2,\\
                                               2(\beta^{'}+2)\}
                                               \end{array} $\\
        \hline \arabic{ijCase}*\addtocounter{ijCase}{1} &  $
            \begin{array}{l}
                p^{s}-p^{s-k}+(\tau-1)p^{s-k-1} \\
                + 1 \le i \le p^{s}-
                p^{s-k}+\tau p^{s-k-1}
            \end{array}$
        & $\beta p^{s-1} + 1 \le j \le (\beta + 1)p^{s-1}$ & $2(\beta + 2)$ \\
        \hline \arabic{ijCase}\addtocounter{ijCase}{1} &
        $
            \begin{array}{l}
                p^{s}-p^{s-k}+(\tau-1)p^{s-k-1}\\
                + 1 \le i
                \le p^{s}-p^{s-k}+\tau p^{s-k-1}
            \end{array}$
        & $
            \begin{array}{l}
                p^{s}-p^{s-k}+(\tau-1)p^{s-k-1} \\
                + 1 \le j
                \le p^{s}-p^{s-k}+ \tau p^{s-k-1}
            \end{array}$
        & $(\tau+1)p^{k}$\\
        \hline \arabic{ijCase}*\addtocounter{ijCase}{1} &
        $
            \begin{array}{l}
                p^{s}-p^{s-k}+(\tau^{(1)}-1)p^{s-k-1}\\
                \begin{array}{ll}
                    + 1 \le i
                    \le & p^{s}-p^{s-k}\\
                    & + \tau^{(1)} p^{s-k-1}
                \end{array}
            \end{array}$
        & $
            \begin{array}{l}
                p^{s}-p^{s-k}+(\tau^{(2)}-1)p^{s-k-1} \\
                \begin{array}{ll}
                    + 1 \le j
                    \le & p^{s}-p^{s-k}\\
                    & + \tau^{(2)} p^{s-k-1}
                \end{array}
            \end{array}$
        & $\begin{array}{l}
        \min\{\\
        2(\tau^{(2)}+1)p^{k},\\
         (\tau^{(1)} + 1)p^k \}
           \end{array}$\\
        \hline \arabic{ijCase}*\addtocounter{ijCase}{1} &
        $
            \begin{array}{l}
                p^{s}-p^{s-k^{'}}+(\tau^{(3)}-1)p^{s-k^{'}-1}\\
                \begin{array}{ll}
                    + 1 \le i
                    \le & p^{s}-p^{s-k^{'}}\\
                    & + \tau^{(3)} p^{s-k^{'}-1}
                \end{array}
            \end{array}$
        & $
            \begin{array}{l}
                 p^{s}-p^{s-k^{''}}+(\tau^{(4)}-1)p^{s-k^{''}-1}\\
                 \begin{array}{ll}
                     + 1 \le j
                    \le  & p^{s}-p^{s-k^{''}}\\
                    & + \tau^{(4)} p^{s-k^{''}-1}
                 \end{array}
            \end{array}$
        & $2(\tau^{(4)}+1)p^{k^{''}} $\\
        \hline \arabic{ijCase}*\addtocounter{ijCase}{1} &
            $i = p^s$ & $\beta p^{s-1} + 1 \le j \le (\beta + 1) p^{s-1} $ & $2(\beta + 2)$\\
        \hline \arabic{ijCase}*\addtocounter{ijCase}{1} &
            $i = p^s$ & $\begin{array}{l}
                p^{s}-p^{s-k}+(\tau-1)p^{s-k-1}\\
                \begin{array}{ll}
                    + 1 \le j
                    \le & p^{s}-p^{s-k}\\
                    & + \tau p^{s-k-1}
                \end{array}
            \end{array}$
        & $2(\tau + 1)p^k$\\
        \hline
\end{tabular}
\end{table}

The results in Table \ref{table.Reducible.Hamming.Distance}
still hold when the polynomials $x^\newn  + \xi$\ and $x^\newn  - \xi$ are reducible
except the fact that the cases
in Table \ref{table.Reducible.Hamming.Distance} do not cover
all the $\lambda$-cyclic codes of length $2\newn p^s$\ over $\F_{p^m}$.
\begin{remark}
  Note that $\id{ (x^\newn  - \xi)^i(x^\newn  + \xi)^j }, 0 \le i,j\le p^s$\ are ideals of $\sR$\
  independent of the fact that $x^\newn  - \xi$\ and $x^\newn  - \xi$\ are irreducible over $\F_{p^m}$.
  So the above results
 from Lemma \ref{Reducible.Lemma.Weight.3} till Corollary \ref{Reducible.Corollary.ips.j.tau.K}
 hold even when the polynomials $x^\newn  - \xi$\ and $x^\newn  + \xi$\ are reducible.
 But in this case, there are more $\lambda$-cyclic codes than the ones of the form
 $\id{ (x^\newn  - \xi)^i(x^\newn  + \xi)^j }, 0 \le i,j\le p^s$\ and their Hamming distance is not given in this paper.
\end{remark}


In the last part of this section, we determine the Hamming distance of some polycyclic codes
of length
$2\newn p^s$\ over $GR(p^a,m)$\ whose canonical images are as above.
In particular, this gives us the Hamming distance of certain constacyclic codes of length
$2\newn p^s$\ over $GR(p^a,m)$.
Let $\lambda_0$, $\xi_0 \in GR(p^a,m)$\ be units and $\overline{\lambda}_0 = \lambda$, $\overline{\xi}_0 = \xi$.
So $ \xi_0^{2p^s} = \lambda_0$\ and, $x^\newn  - \overline{\xi}_0$\ and $x^\newn  + \overline{\xi}_0$\ are irreducible.
The polynomial $x^{2\newn p^s} - \lambda_0$\ factors into two coprime polynomials as
$$ x^{2\newn p^s}-\lambda_0 = x^{2\newn p^s} - \xi_0^{2p^s} = (x^{\newn p^s} - \xi_0^{p^s})(x^{\newn p^s} + \xi_0^{p^s}). $$

Let $f_1(x) = (x^{\newn} - \xi_0)^{p^s} + p\beta_1(x) $\ and $f_2(x) = (x^{\newn} - \xi_0)^{p^s} + p\beta_2(x)$\
with $\deg( \beta_1(x) ), \deg( \beta_2(x) ) < \newn p^s$.
Let $f(x) = f_1(x)f_2(x)$\ and
$
  \sR_0 = \frac{GR(p^a,m)[x]}{ \id{f(x)} }.
$
Note that $f_1(x)$\ and $f_2(x)$\ are primary regular polynomials and therefore we can use the arguments of
Section \ref{sect.main}.

By Proposition \ref{prop.princ_code}, we get
$\sR_0 = \id{f_1(x)} \oplus \id{f_2(x)}$.
Additionally, by Proposition \ref{prop.princ_code}, we know that
$ \id{f_1(x)} \cong \frac{GR(p^a,m)[x]}{\id{f_2(x)}}$\ and
$ \id{f_2(x)} \cong \frac{GR(p^a,m)[x]}{\id{f_1(x)}}$\ are local rings
and the maximal ideals of $\sR_0$\ are
$ \id{ p, x^\newn  + \xi_0 }$\ and
$\id{ p, x^\newn  - \xi_0 }$.

Now given $g(x) \in \sR_0$, we will see how to determine $\overline{ \id{g(x)}}  \subset \sR$.
Since $ \overline{ \id{g(x)}} = \overline{ \id{(x^\newn  - \xi)^{j_0}(x^\newn  + \xi)^{j_1}} }$, we have
$ \bar{g}(x) = (x^\newn  - \bar{\xi})^{j_0}(x^\newn  + \bar{\xi})^{j_1}u(x)$\ where $u(x)$\
is a unit in $\sR$.
In order to determine $j_0$, we consider the substitution
$x^i = (x^\newn  - \xi_0 + \xi_0)^{d_i}x^{\ell_i}$\ for every $i \ge \newn$, we get
\be
    g(x) & = & a_Lx^L + \cdots + a_{\newn} x^{\newn}  + a_{\newn-1}x^{\newn-1} + \cdots + a_0\nn\\
        & = & (x^\newn - \xi_0)^{d_L}h_{d_L}(x) + (x^\newn - \xi_0)^{d_L -1}h_{d_L -1}(x) + \cdots + h_0(x)\nn
\ee
where $h_i(x)$\ are polynomials such that $\deg(h_i(x)) < \newn$\ for $d_L \ge i \ge 0$.
Then $j_0$\ is the least integer with the property $p \nmid h_{j_0}(x)$.
Similarly, via the substitution
$x^i = (x^\newn + \xi_0 - \xi_0)^{d_i}x^{\ell_i}$\ for every $i \ge \newn$,
the integer  $j_1$\ can be determined.

Let $ C = \id{g_1(x),\dots ,g_r(x)} \vartriangleleft \sR_0$\ be a polycyclic code,
where the generators are as in Theorem \ref{theo.main1}.
By Theorem \ref{Groebner.Corollary.Hamming.Distance.Using.SGB}, we have
$d_H(C) = d_H( \overline{\id{ g_r(x) }})$.
The canonical image $\overline{\id{g_r(x) }}$\ of $\id{ g_r(x) }$\ can be
determined as described above. Say
$ \overline{ \id{ g_r(x) } } = \overline{ \id{( x^\newn - \xi )^{\hat{i}}( x^\newn + \xi )^{\hat{j}}} }$\ for some
$0 \le \hat{i},\hat{j} \le p^s$. 
Then $d_H( \overline{ \id{( x^\newn - \xi )^{\hat{i}}( x^\newn + \xi )^{\hat{j}}} } )$\
can be determined using Theorem \ref{Reducible.Theorem.Main}.

\begin{remark}
  Note that $x^{\newn p^s} - \xi_0^{p^s} = (x^\newn - \xi_0)^{p^s} + p \hat{\beta}_1(x)$\ and
  $x^{\newn p^s} + \xi_0^{p^s} = (x^{\newn} + \xi_0)^{p^s} + p \hat{\beta}_2(x)$\ for some
  $\hat{\beta}_1(x),\hat{\beta}_2(x) \in \sR_0$.
 In the above setup, if we take $f_1(x) = (x^\newn - \xi_0)^{p^s} + p \hat{\beta}_1(x) $\ and
 $f_2(x) = (x^\newn + \xi_0)^{p^s} + p \hat{\beta}_2(x)$, then we obtain
 the Hamming distance of $\lambda$-cyclic codes of length $2\newn p^s$\ over $GR(p^a,m)$.
\end{remark}

\section*{Acknowledgments}

Ferruh \"{O}zbudak is partially supported by T\"{U}B\.{I}TAK under Grant No. TBAG-109T672.



\end{document}